\documentclass[12pt]{amsart}

\usepackage{svn}
\SVN $Revision: 299 $

\usepackage{amsmath,latexsym,amssymb,amsfonts,stmaryrd, enumerate}
\usepackage{enumitem}
\usepackage{graphicx}
\usepackage{mathtools}
\usepackage{tensor}
\usepackage{MnSymbol}
\usepackage[colorlinks,anchorcolor=black,citecolor=black,linkcolor=black]{hyperref}
\usepackage[hypcap]{caption}
\usepackage[a4paper,margin=2cm]{geometry}
\usepackage{tikz}
\usetikzlibrary{positioning,tikzmark,fit}

\theoremstyle{plain}
\newtheorem{theorem}{Theorem}
\newtheorem{thm}[theorem]{Theorem}
\newtheorem{lemma}[theorem]{Lemma}
\newtheorem{lem}[theorem]{Lemma}
\newtheorem{prop}[theorem]{Proposition}
\newtheorem{corollary}[theorem]{Corollary}
\theoremstyle{definition}
\newtheorem{definition}[theorem]{Definition}
\newtheorem{defi}[theorem]{Definition}
\newtheorem{remark}[theorem]{Remark}
\newtheorem{exa}[theorem]{Example}
\newcommand\itemspacing{\vspace{2mm}\noindent}   

\numberwithin{equation}{section}
\numberwithin{theorem}{section}

\usepackage{shuffle}
\newcommand{\NCtensorT}{\mathrm{\shuffle}}    \newcommand{\monoM}{\mathrm{M}}       \newcommand{\freeF}{\mathrm{F}}     \newcommand{\boolB}{\mathrm{B}}    \newcommand{\freeLie}{\mathrm{FL}}    \newcommand{\exch}{\mathcal{E}}   \def\cS{\mathcal{S}}

\let\C\undefined
\newcommand{\C}{\mathbb{C}}  \newcommand{\R}{\mathbb{R}}   \newcommand{\Z}{\mathbb{Z}}    \newcommand{\N}{\mathbb{N}}   

\DeclareMathOperator{\IE}{\mathbf{E}}   

\newcommand{\abs}[1]{\left\lvert #1 \right\rvert}  \newcommand{\restr}{\rfloor}  

\newcommand{\alg}[1]{\mathcal{#1}}   

\DeclareMathOperator{\Hom}{\mathrm{Hom}}   
\DeclareMathOperator{\conc}{\mathrm{conc}}   \DeclareMathOperator{\Id}{\mathrm{Id}}   

\renewcommand\sharp{\#}   \renewcommand{\phi}{\varphi}    \newcommand{\Nfin}{\N_{\mathrm{fin}}}   

\newcommand{\SP}{\mathcal{P}} \newcommand{\NC}{\mathcal{NC}}  \newcommand{\IP}{\mathcal{I}}  

\newcommand{\OSP}{\mathcal{OP}}   \newcommand{\EOSP}{\mathcal{OPP}}    \newcommand{\ONP}{\mathcal{ONC}}   \newcommand{\MP}{\mathcal{M}}  \newcommand{\OIP}{\mathcal{OI}}   

\DeclareMathOperator{\outintmax}{\mathrm{\iota_{\max}}}
\DeclareMathOperator{\intmax}{\mathrm{\tilde{\iota}_{\max}}}
\DeclareMathOperator{\outncmax}{\mathrm{\nu_{\max}}}
\DeclareMathOperator{\ncmax}{\mathrm{\tilde{\nu}_{\max}}}
\newcommand{\Outer}{\mathrm{Outer}}   \newcommand{\Inner}{\mathrm{Inner}}    

\DeclareMathOperator{\IncAlg}{\mathfrak{I}}   \newcommand{\Bool}{\mathcal{B}}   \DeclareMathOperator{\downset}{\downarrow\!}   

\newcommand{\SG}{\mathfrak{S}}   

\DeclareMathOperator{\Des}{Des}  \DeclareMathOperator{\Asc}{Asc}   \DeclareMathOperator{\Plat}{Pla}   

\DeclareMathOperator{\des}{\mathrm{des}}   \DeclareMathOperator{\asc}{\mathrm{asc}}     
\DeclareMathOperator{\plat}{\mathrm{pla}}

\newcommand{\tensorT}{\mathrm{T}}

\def\uk{\underline{k}}
\def\us{\underline{s}}
\def\ut{\underline{t}}
\def\uu{\underline{u}}
\def\uz{\underline{z}}

\def\uN{\underline{N}}
\def\uM{\underline{M}}
\def\uX{\underline{X}}
\def\uZ{\underline{Z}}

\newcommand{\Lie}{\mathcal{L}}   \DeclareMathOperator*{\odott}{\odot}   \newcommand{\egf}{\mathcal{F}}    \newcommand{\lap}{\mathcal{L}}     \newcommand{\NCP}{\operatorname{AlgP}}   

\def\fA{{\mathbf{A}}} \def\bfX{{\mathbf X}} \def\fC{\mathbf{C}} \def\gold{g}

\title{Cumulants, Spreadability and the Campbell-Baker-Hausdorff Series}
\author{Takahiro Hasebe}
\address{
Department of Mathematics, Hokkaido University, Kita 10, Nishi 8, Kita-ku, Sapporo 060-0810, Japan
}
\email{thasebe@math.sci.hokudai.ac.jp}
\author{Franz Lehner}
\address{Institut f\"ur Diskrete Mathematik, Technische Universit\"at Graz,
Steyrergasse 30, 8010 Graz, Austria}
\email{lehner@math.tugraz.at}
\date{}

\DeclareMathVersion{normal2} 
\mathversion{normal2} 
\begin{document}


\begin{abstract}
We define spreadability systems as a generalization of exchangeability systems
in order to unify various notions of independence and cumulants known
in noncommutative probability. In particular, our theory covers monotone
independence and monotone cumulants which do not satisfy exchangeability.
To this end we study generalized zeta and M\"obius functions in the context of
the incidence algebra of the semilattice of ordered set partitions and prove an
appropriate variant of Fa\`a di Bruno's theorem.
With the aid of this machinery we show that our cumulants cover most of the previously
known cumulants.  Due to noncommutativity of independence the behaviour
of these cumulants with respect to independent random variables
is more complicated than in the exchangeable case and the appearance
of Goldberg coefficients exhibits the role of the Campbell-Baker-Hausdorff
series in this context.
Moreover, we exhibit an interpretation of the Campbell-Baker-Hausdorff
series as a sum of cumulants in a particular spreadability system, thus
providing a new derivation of the Goldberg coefficients.
\end{abstract}

\maketitle

\tableofcontents{}

\section{Introduction}

\subsection{Background: independence and cumulants}

Cumulants were introduced by Thiele in the late 19th century as a combinatorial
means to describe independence of classical random variables.
In free probability existence of cumulants was indicated by Voiculescu
\cite{V85} and described explicitly by Speicher
\cite{Speicher:1994:multiplicative}.
Free cumulants are one of the cornerstones
in free probability, complementing the analytic machinery of Cauchy transforms, 
see \cite{NicaSpeicher:2006:lectures} for many applications.
Later on other kinds of cumulants were introduced in noncommutative
probability,
e.g.,
Boolean cumulants were defined in \cite{SW97} in the context of Boolean independence
(see also \cite{vonWaldenfels:1973:approach,vonWaldenfels:1975:interval});
various kinds of $q$-deformed cumulants were considered in
\cite{Nica:1995:oneparameter,Anshelevich:2001:partition}
in order to interpolate between classical and free cumulants (however no $q$-convolution has been found so far);
conditionally free cumulants were defined in \cite{BLS96} which generalize both
free and Boolean cumulants.
The second-named author gave a unified theory of the cumulants mentioned above
in the framework of so-called \emph{exchangeability systems} 
as a general notion of independence \cite{L04}.  
In the present paper we develop a yet more general framework which
comprises also Muraki's monotone independence~\cite{M01}, 
which is not covered by the approach of \cite{L04} because
it does not satisfy exchangeability, the obstruction being that
monotone independence is sensitive to the order on random variables:
\begin{equation}\label{eq:difference-monotone}
\text{Independence of $X$ and $Y$ is not equivalent to independence of $Y$ and $X$.} \tag{O}
\end{equation}
Hence in order to avoid misinterpretations we say  ``the (ordered) pair $(X,Y)$ is
monotone independent'' rather than  ``$X$ and $Y$ are monotone independent''.
This property sharply distinguishes monotone independence from classical, free,
and Boolean independences. 
Despite lack of exchangeability, the first-named author 
together with H.~Saigo managed to define monotone cumulants \cite{HS11a,HS11b}
relying only on the property called \emph{extensivity} defined below.
If we denote the monotone cumulants by $K^{\monoM}_n(X_1,X_2,\dots, X_n)$ with
respect to a noncommutative probability space $(\mathcal{A},\varphi)$,
then we have the moment-cumulant formula
\begin{equation}\label{monotone moment cumulant}
 \varphi(X_1 X_2 \cdots X_n) = \sum_{\pi\in\MP_n} \frac{1}{\abs{\pi}!}K_{(\pi)}^{\monoM}(X_1,X_2,\dots,X_n)  \tag{MC}
\end{equation}
where the set
$\MP_n$ of monotone partitions is a 
subclass of ordered set partitions rather  than set partitions.
The factor $\frac{1}{\abs{\pi}!}$ accounts for the number of possible
orderings of the blocks of the underlying set partition of $\pi$
and cancels out in the case of exchangeability,
like classical, free or Boolean independence.

Cumulants carry essential information on independence, in particular the
\emph{vanishing} of mixed cumulants, that is, of cumulants with independent
entries, characterize independence \cite[Prop.~3.5]{L04},
which is the major reason for their usefulness in free probability
\cite{NicaSpeicher:2006:lectures}.
More precisely, if a finite family of random variables
$X_1,X_2, \dots, X_n$ can be partitioned into two mutually independent
subfamilies 
(in the general sense of Definition~\ref{def:exchangeableindependence} below)
then
\begin{equation}\label{V}
K_n(X_1,X_2,\dots,X_n)=0.\tag{V} 
\end{equation} 
As a consequence, cumulants are additive, that is, the cumulant of the sum of two independent tuples
 $(X_1,X_2,\dots,X_n)$ and
$(Y_1,Y_2,\dots,Y_n)$ decomposes as
\begin{equation}
\label{eq:additivity}
K_n(X_1+Y_1,X_2+Y_2,\dots,X_n+Y_n) = 
K_n(X_1,X_2,\dots,X_n) + K_n(Y_1,Y_2,\dots,Y_n). 
\end{equation} 
By contrast, because of property \eqref{eq:difference-monotone},  monotone
cumulants do not satisfy additivity and thus mixed cumulants do not necessarily
vanish. Instead, they satisfy the weaker notion of \emph{extensivity}: 
if $\{(X_1^{(j)}, X_2^{(j)},\dots, X_n^{(j)})\}_{j=1}^\infty$
is a sequence of monotone independent random vectors such that
$(X_1,X_2, \dots,X_n)\overset{d}{=}(X_1^{(j)},X_2^{(j)},\dots, X_n^{(j)})$
for any $j \geq1$, then  
\begin{equation}\label{E}
K_n^{\monoM}(N.X_1, N.X_2,\dots, N.X_n) = N K_n^{\monoM} (X_1,X_2,\dots,X_n), \tag{E}
\end{equation}
where $N.X_i=X_i^{(1)}+X_i^{(2)}+\cdots+X_i^{(N)}$ is the sum of i.i.d.~copies.

Extensivity is strictly weaker than the property of vanishing of mixed cumulants, 
but extensivity (together with some other properties) still suffices
to prove uniqueness of cumulants even in the case of exchangeability. 
Therefore extensivity is a natural generalization of the property of 
vanishing of mixed cumulants.

\subsection{Main objectives of the present paper}
The first goal of this paper is the unification of second-named author's 
approach to cumulants based on exchangeability systems and first-named author's 
monotone cumulants based on universal products of states. 
The second-named author's definition of cumulants includes $q$-deformed
cumulants as well as tensor (or classical), free and Boolean cumulants. On the
other hand, the approach of the first-named author and Saigo comprises monotone
cumulants as well as tensor, free and Boolean cumulants, but not $q$-deformed
cumulants. In the present paper 
we establish a unified theory based on the  concept of spreadability
which has been considered recently by K\"ostler~\cite{Koe10} in the noncommutative context. Similar to the transition from symmetric to quasisymmetric functions,
the concept of \emph{spreadability systems} naturally arises as a generalization of
exchangeability systems and allows to unify various
kinds of independence and cumulants, including conditionally monotone
independence \cite{Hasebe:2011:conditionally} and two other generalized
notions of independence \cite{H10}, with a generalization of the moment-cumulant
formula \eqref{monotone moment cumulant} and the property of extensivity
\eqref{E}. 

Our approach is combinatorial on the basis of \emph{ordered set partitions}. 
An ordered set partition is defined as an ordered \emph{sequence} of disjoint subsets
whose union is the entire set (say $\{1,2,\dots,n\}$).  
We first investigate the structure of the semilattices $\OSP_n$
of ordered set partitions (with respect to dominance order) 
and show that they locally look like the lattices $\IP_n$ of interval partitions,
in the sense that any interval in $\OSP_n$ is isomorphic to an interval
in $\IP_k$ for an appropriate $k$.
This property is crucial for the study of multiplicative functions 
in the incidence algebra of ordered set partitions and we establish
an isomorphism with the composition algebra of generating
functions analogous to the well known formula of Fa\`a di Bruno.
In particular, we obtain the fundamental convolution identity
for the generalized zeta and M\"obius functions.

Our interest in ordered set partitions was stipulated by the appearance of
monotone partitions in the classification of independence \cite{M03} 
and their role for monotone cumulants \cite{HS11a}, 
but presently it turns out that the structure of ordered set partitions is
actually easier to describe than that of monotone partitions. 
For example we do not have a good understanding of the structure of intervals
$[\sigma,\pi]\subseteq \MP_n$  for $\sigma,\pi\in\MP_n$ and in particular
the values of the M\"obius function remain mysterious.

The second goal is the application of our results to \emph{free Lie algebras}.
Connections of free probability to formal groups have been pointed out early by
Voiculescu~\cite{V85} and more recently by Friedrich and
McKay~\cite{FriedrichMcKay:2015:homogeneous} and 
in the context of monotone probability by
Manzel and Sch\"urmann~\cite{ManzelSchurmann:2016:noncommutative};
see also an approach via shuffle algebras
\cite{EbrahimiFardPatras:2017:monotone}.

Here we obtain new formulations of some well known identities in terms of
ordered set partitions and cumulants.
In particular our cumulants turn out to coincide with the homogenous components
of the \emph{Campbell-Baker-Hausdorff formula}, if an appropriate spreadability
system is chosen. Thus a new derivation of the coefficients of the
Campbell-Baker-Hausdorff formula (``Goldberg coefficients'')
drops out as a by-product of our results on spreadability systems.

We hope that our results will stimulate more connections to combinatorics, 
in particular Hopf algebras and noncommutative quasi-symmetric functions
\cite{BergeronZabrocki:2009:hopf}.
Our proofs only use elementary and at times tedious calculations, yet
we suspect that many of our results have been obtained in different contexts
before and that those with the right knowledge will find easier and more
conceptual proofs.
After a first draft of the present paper was published,
this direction was taken up in \cite{LehnerNovelliThibon:2020}.

\subsection{Main results and structure of the paper}
Most definitions and results of the present paper are expressed in the language
of ordered set partitions.
In order to avoid a large overhead and streamline the presentation all relevant
material concerning these and related objects has been collected in
Appendix~\ref{app:partition composition}.
This includes in particular the description of intervals in the poset of
ordered set partitions, incidence algebras, multiplicative functions, generalized zeta functions and the corresponding M\"obius functions.

In Section~\ref{sec:spreadabilitydefinition}, we define spreadability systems
(Definition \ref{def:spreaded}) and give examples coming from the four (or
five) universal products of linear maps (Section \ref{exaspread}).
Roughly speaking, a
\emph{spreadability system} for a noncommutative probability space $(\alg{A},\varphi)$
consists of a larger space $(\alg{U}, \tilde \varphi)$ containing
copies $X^{(i)}, i\in \N,$ of each element $X \in \alg{A}$ whose joint
distribution in invariant under spreading, i.e., such that
\begin{equation}\label{eq:main1}
\tilde \varphi(X_1^{(i_1)} X_2^{(i_2)} \cdots X_n^{(i_n)})  =  \tilde \varphi(X_1^{(j_1)} X_2^{(j_2)} \cdots X_n^{(j_n)})
\end{equation}
whenever the entries of the multiindices
$(i_1,i_2,\dots, i_n)$ and $(j_1,j_2,\dots, j_n)$
are in one-to-one correspondence via an order preserving map;
in other words, whenever the kernel ordered set partitions
$\kappa(i_1,i_2,\dots, i_n)$ and $\kappa(j_1,j_2,\dots, j_n)$ 
(Definition~\ref{app:def:kernelpartition}) coincide.
In this case the common value of \eqref{eq:main1} only depends
on said ordered set partition  $\pi = \kappa(i_1,i_2,\dots, i_n)$ and is
denoted by
\begin{equation}
  \label{eq:phipi=tildephii}
\varphi_\pi(X_1,X_2,\dots, X_n)=
\tilde \varphi(X_1^{(i_1)} X_2^{(i_2)} \cdots X_n^{(i_n)})
.
\end{equation}
In typical examples (like the ones presented in Section~\ref{exaspread}),
$\alg{U}$ is the free product or tensor product of copies of $\alg{A}$,
$\tilde \varphi$ is a universal product (e.g., the free product)
of the copies of $\varphi$, and $X^{(i)}$ is the
natural embedding of $X$ into $\alg{U}$ as the  $i$th component; then
\eqref{eq:main1}  holds trivially.
The main idea is to regard $\{X^{(i)}, i\in \N\}$ as i.i.d.\
copies of the random variable $X$ and condition \eqref{eq:main1} gives rise to a generalized notion of i.i.d.\
sequences $\alg A^{(i)}:=\{X^{(i)}: X\in \alg A\}, i \in \N$ of subalgebras of $\alg{U}$.
Using this, we can induce a notion of independence on the original algebra
$\alg{A}$ by generalizing the following elementary observation from classical probability:
two random variables $X$ and $Y$ are independent if and only if the joint distribution of the
random vector $(X,Y)$ is the same as that of $(X^{(1)},Y^{(2)})$,
where $(X^{(1)},Y^{(1)})$ and  $(X^{(2)},Y^{(2)})$ are i.i.d.\ copies of
$(X,Y)$.
More precisely, our algebraic notion of independence
can be formulated in lattice theoretic terms as follows:
a sequence of subalgebras $(\alg{A}_i)_{i\in I}$ of $\alg{A}$, where  $I\subseteq \N$, is said to be
  \emph{$\cS$-independent} if for any tuple of indices
  $(i_1,i_2,\dots,i_n)\in I^n$, 
  any tuple of random variables $(X_1,X_2,\dots,X_n)$
  with $X_j\in  \alg{A}_{i_j}$
  and any ordered set partition $\pi\in\OSP_n$, we have
  \begin{equation} %\label{S-indep} 
  \phi_\pi(X_1,X_2,\dots,X_n) = \phi_{\pi\curlywedge \kappa(i_1,i_2,\dots,i_n)}(X_1,X_2,\dots,X_n),   
  \end{equation}
  where $\curlywedge$ is roughly the usual operation $\wedge$ for the underlying set partitions, equipped with the lexicographic order on blocks, see Definition  \ref{app:def:ordered_set_partitions}. 
For several specific spreadability systems this definition reproduces previous
notions of independence (Proposition \ref{prop:equivalence_independences}).

In Section~\ref{sec:cumulants}, we define cumulants associated to a
spreadability system and express cumulants in terms of moments and vice
versa (Theorem \ref{thm:mc}). More precisely, we adapt Rota's
\emph{dot operation} from umbral calculus \cite{RotaTaylor:1994:classical}
for our purpose and write
\[
N. X := X^{(1)} + X^{(2)} + \cdots + X^{(N)}, 
\]
for ``the sum of i.i.d.\ random variables''.
Using the notation
\[
\varphi_{\pi}(X_1^{(i_1)}, X_2^{(i_2)}, \dots, X_n^{(i_n)}) :=
\varphi_{\pi\curlywedge \kappa(i_1,i_2,\dots, i_n)}(X_1,X_2,\dots, X_n),
\]
we show that for each $\pi \in \OSP_n$
the partitioned expectation  $\varphi_\pi (N.X_1, N. X_2, \dots, N.X_n)$ is a
polynomial in $N$ and the  coefficient of
the lowest order term $N^{|\pi|}$ is the $\pi$-cumulant with respect to the
spreadability system $\alg{S}$,  denoted by $K_\pi(X_1,X_2,\dots, X_n)$.
 These coefficients satisfy the fundamental properties of
 cumulants listed in Definition~\ref{def:spreadableK2}
 and are related to the partitionend moments
 \eqref{eq:phipi=tildephii} 
 by \emph{generalized M\"{o}bius inversion},
 \begin{align}
 K_{\pi}(X_1,X_2,\ldots,X_n)  &= \sum_{\substack{\sigma \in \OSP_n \\ \sigma \leq \pi}}  \varphi_{\sigma}(X_1,X_2,\ldots,X_n)\,  \widetilde{\mu}(\sigma,\pi), \label{eq:mci}  \\
  \varphi_{\pi}(X_1, X_2,\ldots, X_n) &= \sum_{\substack{\sigma \in \OSP_n \\ \sigma \leq \pi}} K_{\sigma}(X_1,X_2,\ldots,X_n)\,\widetilde{\zeta}(\sigma,\pi), \label{eq:mcii}
  \end{align}
where $\widetilde{\mu}$ and $\widetilde{\zeta}$ are generalized M\"obius and
zeta functions, respectively (see Definition \ref{app:def:factorials}).
In the
special case of the maximal element $\pi= \hat{1}_n$ these coefficients
evaluate to $\widetilde{\mu}(\sigma, \hat{1}_n) =
\frac{(-1)^{n-|\sigma|}}{|\sigma|}$ and $\zeta(\sigma, \hat{1}_n) =
\frac{1}{|\sigma|!}$, cf.\ formula \eqref{monotone moment cumulant}.

These concepts have a certain affinity to formal group laws
and exhibit a Lie algebraic flavour, which is
confirmed by the considerations in Section~\ref{sec:freeLiealg}.

We also prove extensivity (Proposition~\ref{ad}) and uniqueness of
cumulants (Theorem \ref{thm:unique}). The previous moment-cumulant formulas for the free, tensor, Boolean, monotone and c-monotone cases in the
literature are listed as special cases (Propositions \ref{prop19} and \ref{prop20}); for instance,
the partitioned cumulants $K_\pi^{\rm M}$ associated with the monotone spreadability system are shown to satisfy multiplicativity 
\begin{equation}\label{eq:monotone_multiplicative}
  K^{\monoM}_\pi=
  \begin{cases}
    0, &\pi \in \OSP_n \setminus \MP_n, \\
    K^{\monoM}_{(\pi)},& \pi \in \MP_n,  
  \end{cases}
\end{equation} 
which, together with  \eqref{eq:mcii}, implies the known formula
\eqref{monotone moment cumulant}. 
While associativity of monotone independence is crucial for the proof of
\eqref{eq:monotone_multiplicative},
this property is not required
for the construction of cumulants for general spreadability systems.       
     
In Section~\ref{sec:partialcumulant} we establish recursive differential
equations for the time evolution of moments. This generalizes for example the
complex Burgers equation in free probability \cite{V86}. Motivations for this section come especially from the effective use of differential equations in monotone probability theory.

 In Sections~\ref{sec:vanishing} and \ref{sec:freeLiealg} we encounter the ``Goldberg coefficients'' in two different ways whose connections are still unclear. 
In Section~\ref{sec:vanishing}, 
we compute \emph{mixed cumulants}, i.e., 
we express cumulants of random variables which split into ``independent subsets'' in terms of lower order cumulants (Theorem \ref{cum vanishing}):  
 a sequence of subalgebras $(\alg{A}_i)_{i\in I}$ of $\alg{A}$, where 
  $I\subseteq \N$, is $\cS$-independent if and only if for any tuple $(i_1,i_2,\dots,i_n)\in I^n$, any random variables $(X_1,X_2,\dots,X_n)\in \alg{A}_{i_1} \times \alg{A}_{i_2} \times \cdots \times \alg{A}_{i_n}$ and any ordered set partition $\pi\in\OSP_n$, we have 
\begin{equation} \label{eq:partial_vanishing}
  K_\pi(X_1,X_2, \ldots, X_n)
  - \sum_{\tau \in \OSP_n} K_\tau(X_1,X_2,\ldots,X_n)
  \, \gold(\tau,\kappa(i_1,\dots,i_n),\pi) = 0,     
\end{equation}
where $\gold(\tau,\eta,\pi)$ is what we call the Goldberg coefficient. In the special case $\pi =\hat 1_n$ it is given by  
\begin{equation}\label{eq:goldberg}
\gold(\tau,\eta, \hat 1_n)
= 
\begin{cases}\displaystyle
\frac{1}{q_1! q_2! \cdots q_r!}\int_{-1}^0 x^{\des_\eta(\tau)}(1+x)^{\asc_\eta(\tau)}\prod_{j=1}^r P_{q_j}(x)\,dx, &\bar{\tau} \leq \bar{\eta},\\
0, & \bar{\tau} \not\leq \bar{\eta}, 
\end{cases}
\end{equation}
where $r$ and $q_j$ are certain integers determined by $\tau$ and $\eta$,
$\des_\eta(\tau), \asc_\eta(\tau)$
are the numbers of  descents and ascents of the multiset permutations
naturally associated with $(\eta,\tau)$.
$
P_q(x)=\sum_{k=1}^q k! \, S(q,k)\, x^{k-1}
$
are the homogeneous Euler polynomials and $S(q,k)$ are the Stirling numbers of
the second kind.
For general $\pi$, $\gold(\tau,\eta,\pi)$ is either zero or the product of $\gold(\tau\restr_P,\eta\restr_P, \hat 1_{|P|})$ over $P \in \pi$. 
Formula \eqref{eq:partial_vanishing} sheds light on the gap between extensivity \eqref{E} 
and the vanishing property \eqref{V} and reveals the role of
the Camp\-bell-Baker-Hausdorff formula.
 
In Section~\ref{sec:freeLiealg} we present an
operator-valued ``unshuffle'' spreadability system $\cS_{\NCtensorT}$
which is interesting from a combinatorial point of view:
The corresponding
cumulants reproduce the homogeneous components of the Camp\-bell-Baker-Hausdorff
formula; in particular, the specialization of our cumulants  
to free algebras corresponds to Lie projectors.
The definition of the spreadability system $\cS_{\NCtensorT}$ is rather simple:
$\alg{A}$ is any unital algebra, $\varphi \colon \alg{A} \to \alg{A}$ is the
identity map, $\alg{U}$ is the infinite tensor product of the copies of
$\alg{A}$, and $X^{(i)} = 1^{\otimes(i-1)} \otimes X \otimes 1^{\otimes \infty}
\in \alg{U}$ is the natural embedding of $X \in \alg{A}$ into $\alg{U}$ as the
$i$-th tensor factor.  The linear map $\tilde \varphi \colon \alg{U}\to
\alg{A}$ is given by concatenation of words:
$\tilde \varphi(X_1 \otimes X_2 \otimes \cdots \otimes X_n \otimes 1^\infty):=
X_1 X_2 \cdots X_n$.
Consequently, the mixed expectation yields a rearrangement (``unshuffle'')
$$
\tilde{\varphi}(X_1^{(i_1)} X_2^{(i_2)}\cdots X_n^{(i_n)}):= X_{P_1} X_{P_2}\cdots X_{P_k},
$$
where $\kappa(i_1,i_2,\ldots,i_n)=(P_1,P_2,\ldots,P_k)$ and $X_P$ is the
ordered product of $X_p$'s over the elements $p \in P$, see \eqref{product}.

For this specific spreadability system, $\alg{\cS}_{\NCtensorT}$-independence
of subalgebras $(\alg{A}_i)_{i\in I}$ turns out to be equivalent to the
commutativity $[\alg{A}_i, \alg{A}_j]=0$ for $i\ne j$ (Remark
\ref{rem:commutativity=vanishing}). The second cumulant $K_{\hat 1_2}(X,Y)$ is
the commutator $\frac1{2}[X,Y]$.

Finally we give a new derivation of the coefficients of the Campbell-Baker-Hausdorff formula 
(also known as ``Goldberg coefficients'') using the moment-cumulant formulas
\eqref{eq:mci} and \eqref{eq:mcii} for the particular spreadability system $\cS_{\NCtensorT}$.   
To this end specialize $\alg{A}$ to the free associative algebra generated by a
set of noncommuting variables $a_1,a_2,\ldots$,
fix $r\in \N$ and $q_j,i_j \in \N$ for $j\in [r]$ such that $i_j \neq
i_{j+1}$ for $j\in[r-1]$,
then the coefficient of the monomial
$a_{i_1}^{q_1}a_{i_2}^{q_2}\cdots a_{i_r}^{q_r}$
appearing in $\log(e^{a_1}\cdots e^{a_n})$ is given by 
\begin{equation} \label{eq:CBH_coefficient}
\frac{1}{q_1!q_2!\cdots q_r!}\int_{-1}^0 x^{\des(\underline{i})} (1+x)^{\asc(\underline{i})} \prod_{j=1}^r P_{q_j}(x)\,dx, 
\end{equation}
where $\underline{i}$ stands for the sequence $(i_1,i_2,\dots, i_n)$ and $P_q(x)$ are the homogeneous Euler polynomials already encountered in \eqref{eq:goldberg}.

In Section~\ref{CLT} we briefly discuss the central limit theorem associated to 
spreadability systems satisfying a certain singleton condition.

We conclude the paper with a few open problems.
One of them is to find an explanation or unified proof of the surprisingly coinciding formulas \eqref{eq:goldberg} and \eqref{eq:CBH_coefficient}.

\section{Spreadability systems and independence}\label{sec:spreadabilitydefinition}

\subsection{Notation and terminology for noncommutative probability}
\label{Notation} 

From now on we denote by $\mathcal{A}$ and $\mathcal{B}$ associative algebras over the field 
$\C$.  Elements of $\mathcal{A}$ are called \emph{random variables}, and elements of $\mathcal{A}^n, n\in\N$ are called
\emph{random vectors}. An (algebraic) $\mathcal{B}$-valued \emph{expectation} is a linear map 
$$
\varphi:\mathcal{A}\to \mathcal{B} 
$$
and  we call the pair $(\mathcal{A},\varphi)$ an \emph{(algebraic)
  $\mathcal{B}$-valued noncommutative probability space} ($\alg{B}$-ncps).
In the case where $\mathcal{B}$ is a subalgebra of $\mathcal{A}$ and $\varphi$
is a $\mathcal{B}$-module map in the sense that
$\varphi(bab')=b\,\varphi(a)\,b'$ for all $a\in\mathcal{A}$ and
$b,b'\in\mathcal{B}$, the map $\varphi$ is called \emph{conditional
  expectation}. This property will however not be crucial in the context of the
present paper.

\begin{remark} \ 
\begin{enumerate}[label=\rmfamily(\roman*),leftmargin=1cm]
\item In order to be able to include Boolean and monotone
    products and some other examples (Section~\ref{exaspread})
    we do not a priori assume the unitality. In some examples $\alg A$ and $\alg B$ are naturally unital; then we also assume that $\varphi$ is also unital, i.e., $\varphi(1_{\alg A})=1_{\alg B}$ and also that involved subalgebras of $\alg A$ contain the unit of $\alg A$. 
\item 
Usually the involved algebras are $\ast$-algebras and the linear maps satisfy positivity.  
However positivity is not essential here and we stick to the algebraic
$\alg{B}$-valued setting, which allows to include the interesting
example of Lie polynomials in Section~\ref{sec:freeLiealg}.

\end{enumerate}
\end{remark}

We say that two sequences $(X_i)_{i=1}^\infty, (Y_i)_{i=1}^\infty\subseteq \mathcal{A}$
have \emph{the same distribution} if 
$$
\varphi(X_{i_1} X_{i_2} \dotsm X_{i_n})
  = \varphi(Y_{i_1} Y_{i_2} \dotsm Y_{i_n})  
$$ 
for any tuple $(i_1,i_2,\dots,i_n) \in \N^n, n \in\N$, and in this case we
write $(X_i)_{i=1}^\infty \overset{d}{=} (Y_i)_{i=1}^\infty$.

Alternatively, in the categorical approach (see, e.g.,  \cite{ManzelSchurmann:2016:noncommutative}),
a random variable is a unital homomorphism $\iota: \mathcal{D} \to \mathcal{A}$ 
from some unital algebra $\mathcal{D}$ into $\mathcal{A}$.
This definition extends also to
random vectors. 
Indeed, given a random vector $(X_1,X_2,\dots,X_n)$, we get a homomorphism $\iota: \mathcal{D}\to\mathcal{A}$ defined by $\iota(x_i)=X_i$, where $\mathcal{D}$ is the nonunital algebra freely generated by noncommuting indeterminates $x_1,x_2,\dots,x_n$. 
Let $(\mathcal{A}_1,\varphi_1)$,
$(\mathcal{A}_2,\varphi_2)$ be two $\alg{B}$-ncps such that $\varphi_1,\varphi_2$ take
values in a common algebra $\mathcal{B}$. Sequences of random variables
$(\iota_1^{(i)})_{i=1}^\infty \subseteq \Hom(\mathcal{D}, \mathcal{A}_1),
(\iota_2^{(j)})_{j=1}^\infty \subseteq \Hom(\mathcal{D}, \mathcal{A}_2)$ have
\emph{the same distribution} if  
$$
\varphi_1(\iota_1^{(i_1)}(x_1)
          \iota_1^{(i_1)}(x_2)
          \dotsm
          \iota_1^{(i_n)}(x_n))
 = \varphi_2(\iota_2^{(i_1)}(x_1)
             \iota_2^{(i_2)}(x_2)
             \cdots 
             \iota_2^{(i_n)}(x_n))
$$
for any $(i_1,i_2,\dots,i_n)\in\N^n$ and any $x_1,x_2,\dots,x_n \in \mathcal{D}$. In this case we write 
$$
(\iota_1^{(i)})_{i=1}^\infty \overset{d}{=} (\iota_2^{(j)})_{j=1}^\infty.  
$$

Given $X_1, X_2, \ldots, X_n \in \mathcal{A}$ and
$P=\{p_1, p_2, \dots, p_k\} \subseteq [n]$ with $p_1 < p_2 < \cdots < p_k$, 
it will be convenient to introduce the notation 
\begin{equation}\label{product}
  X_P := X_{p_1} X_{p_2} \cdots X_{p_k}
\end{equation}
for the ordered product.   For a $k$-linear functional $M\colon \alg{A}^k \to \alg{B}$, we denote 
\begin{equation}\label{notation:linear}
M(X_P):=M(X_{p_1}, X_{p_2}, \ldots, X_{p_k}).  
\end{equation}

Recall that the tensor product has the universal property that any multilinear map
$$
T:\alg{A}^n\to \alg{B}
$$
has a unique lifting to a linear map
$$
\tilde{T}:\alg{A}^{\otimes n}\to\alg{B}
$$
such that on rank 1 tensors we have $\tilde{T}(a_1\otimes
a_2\otimes\dots\otimes a_n)=T(a_1,a_2,\dots,a_n)$. We will
tacitly identify $T$ with $\tilde{T}$ in
order to simplify notation.

\subsection{Spreadability systems}
In this subsection we introduce the notation necessary to generalize
the notions of partitioned moment and cumulant functionals of \cite{L04}
from the exchangeable setting to the spreadable setting.

\begin{defi} \label{def:spreaded}Let $(\mathcal{A}, \varphi)$ be a $\mathcal{B}$-ncps. 
\begin{enumerate}[label=\arabic*., leftmargin=1cm]
 \item A \emph{spreadability system} for $(\mathcal{A}, \varphi)$ is a triplet
  $\cS=(\alg{U}, \tilde{\varphi}, (\iota^{(i)})_{i=1}^\infty)$ satisfying the
  following properties:

\begin{enumerate}[label=\rm(\roman*)]
 \item $(\alg{U}, \tilde{\varphi})$ is a $\alg{B}$-valued ncps.

 \item $\iota^{(i)}: \mathcal{A} \to \alg{U}$ is a homomorphism
 such that $\varphi = \tilde{\varphi} \circ \iota^{(i)}$ for
  each $i \geq 1$. For simplicity, $\iota^{(i)}(X)$ is denoted by $X^{(i)}$, $X \in
  \mathcal{A}$, and we denote by $\mathcal{A}^{(i)}$ the image
  of $\mathcal{A}$ under $\iota^{(i)}$.
 \item The identity
  \begin{equation}\label{eq:spreadability}
    \tilde{\varphi}(X_1^{(i_1)}X_2^{(i_2)}\cdots X_n^{(i_n)}) = \tilde{\varphi}(X_1^{(h(i_1))}X_2^{(h(i_2))}\cdots X_n^{(h(i_n))}) 
  \end{equation}
  holds for any $X_1, X_2,\ldots, X_n \in \mathcal{A}$, any $i_1, i_2, \ldots, i_n \in
  \N$ and any \emph{order preserving} map $h:\{i_1, i_2, \ldots, i_n\}
  \to \N$, that is, $i_p < i_q$ implies $h(i_p)<h(i_q)$.
\end{enumerate}
\item A triplet $\exch=(\alg{U}, \tilde{\varphi}, (\iota^{(i)})_{i=1}^\infty)$ is called an
 \emph{exchangeability system} if, in addition to (i), (ii) above, eq.\
 (\ref{eq:spreadability}) holds for any $X_1, X_2, \ldots, X_n \in \mathcal{A}$, any $i_1,
 i_2, \ldots, i_n \in \N$ and any \emph{permutation}
 $h \in \SG_\infty:=\bigcup_{n\geq 1}\SG_n$. 
\end{enumerate}
\end{defi}
\goodbreak{}
\begin{remark}   \  
\begin{enumerate}[label=\arabic*.,leftmargin=1cm]
\item In some examples where $\alg A, \alg B, \phi$ are unital, the extended algebra $\alg{U}$ is also unital and $\iota^{(i)}$ are unit-preserving. To be specific, this applies to the tensor and free spreadability systems in Section \ref{exaspread} and the unshuffle spreadability system in Section \ref{sec:freeLiealg}.  
    \item It is easy to see that condition (\ref{eq:spreadability}) can be rephrased as follows:
    \begin{equation*}
(\iota^{(1)}, \iota^{(2)},\dots)
      \overset{d}{=} (\iota^{(n_1)}, \iota^{(n_2)},\dots)
    \end{equation*}
    for any strictly increasing sequence $(n_i)_{i=1}^\infty \subseteq \N$. This
    is the definition given in \cite{Koe10}.

   \item It is straightforward to extend the definition of exchangeability
    systems (resp., spreadability systems) from  $\N$
    to an arbitrary set (resp., arbitrary totally ordered set).
\end{enumerate}
\end{remark}

\begin{defi}  \ 
  \begin{enumerate}[label=\rm(\roman*),leftmargin=1cm]
\item 
    Using the concept of kernel partition from Definition~\ref{app:def:kernelpartition} the condition of spreadability
    \eqref{eq:spreadability} 
    is equivalent to the requirement that
    \begin{equation}
      \label{eq:phiXi1i2=phiXj1j2}
      \tilde\varphi(X_1^{(i_1)} X_2^{(i_2)} \cdots X_n^{(i_n)}) = \tilde\varphi(X_1^{(j_1)} X_2^{(j_2)} \cdots X_n^{(j_n)})
    \end{equation}
    holds whenever the kernels
    coincide, i.e.,  $\kappa(i_1,i_2, \ldots, i_n) =
    \kappa(j_1,j_2,\ldots,j_n)$,
    see Definition~\ref{app:def:kernelpartition}.
    That is, the expectation (\ref{eq:spreadability}) 
    only depends on the ordered kernel set partition  $\kappa(i_1,i_2,\ldots,i_n)$.
    Thus for every ordered set partition $\pi \in \OSP_n$ 
    we can define a multilinear functional
    $\varphi_\pi\colon \mathcal{A}^n \to \C$ by choosing any representative sequence
    $(i_1, i_2, \ldots, i_n)$ with $\kappa(i_1, i_2, \ldots, i_n)=\pi$ and setting
    \begin{equation}
      \label{eq:phipi}
      \phi_{\pi}(X_1, X_2, \ldots,X_n)=\tilde{\phi}(X_1^{(i_1)} X_2^{(i_2)} \cdots X_n^{(i_n)}). 
    \end{equation}
    The invariance \eqref{eq:phiXi1i2=phiXj1j2} ensures that
    this definition is consistent and does not depend on the choice of the representative. 

\item 
This generalizes the corresponding notions from exchangeability systems \cite{L04}:
given an exchangeability system $\exch=(\alg{U}, \tilde{\varphi}, (\iota^{(i)})_{i=1}^\infty)$, we can define a multilinear functional $\phi_\pi$, 
this time for any set partition $\pi \in \SP_n$,   
\begin{equation*}
\phi_{\pi}(X_1, X_2, \ldots,X_n)=\tilde{\phi}(X_1^{(i_1)} X_2^{(i_2)} \cdots X_n^{(i_n)}), 
\end{equation*}
where $(i_1,i_2,\dots,i_n)$ is any representative such that $\overline{\kappa}(i_1,i_2,\dots,i_n)=\pi$,  see Definition \ref{app:def:kernelpartition}. 
  \end{enumerate}
  
\end{defi}

\begin{remark} 
The algebra $\alg{U}$ and the homomorphisms $(\iota^{(i)})_{i\in\N}$ of a
spreadability system can always be chosen to be the tensor algebra of the
copies of $\alg A$ (see \eqref{eq:coproduct} below) with the natural
embeddings, respectively, in the following sense.
Given a spreadability system $\cS=(\alg{U}, \tilde{\varphi}, (\iota^{(i)})_{i=1}^\infty)$ for a $\alg{B}$-ncps $(\alg A,\varphi)$, we set $\hat{\alg{U}} := \bigsqcup_{i=1}^\infty\mathcal{A}_i$ where $\alg A_{i}$ are copies of $\alg A$ and set $\hat\iota^{(i)} \colon \alg A =\alg A_i \to \hat{\alg{U}}$ to be the natural embeddings. Then we can define $\hat \varphi \colon \hat{\alg{U}}  \to \alg B$ by setting
\[
\hat \varphi (\hat\iota^{(i_1)}(X_1)  \hat\iota^{(i_2)}(X_2)  \cdots \hat\iota^{(i_n)}(X_n)) := \tilde{\varphi} (X_1^{(i_1)} X_2^{(i_2)} \cdots X_n^{(i_n)}) 
\]
for every $i_1, i_2,\dots, i_n$ with $i_k \ne i_{k+1}~(k\in[n-1])$ and
$X_1,X_2,\dots, X_n \in \alg A$, giving rise to the spreadability system $\hat
\cS=(\hat{\alg{U}}, \hat \phi, (\hat \iota^{(i)})_{i\ge1})$ for $(\alg A,
\phi)$. Now $(\hat\iota^{(i)})_{i\ge1}\overset{d}{=} (\iota^{(i)})_{i\ge1}$ and
thus replacing $\cS$ with $\hat\cS$ does not change basic results; in
particular, the cumulants for $\cS$ and for $\hat \cS$ are identical, see
Theorem \ref{thm:mc}.

In the majority of the examples below we choose $\alg{U}$ to be the tensor algebra,
in a few cases however, it is more natural to choose another $\alg{U}$.
A notable exception is the case of unital algebras, where 
the unital free product $\ast_{i \in \N} \alg A$ and the infinite tensor product
$\bigotimes_{i\in \N} \alg A$ are more convenient.
\end{remark}

\subsection{Examples from universal products of linear functionals}
\label{exaspread}
Spreadability systems typically arise as universal products of linear
functionals defined on free products of $\ast$-algebras. 
Universal products can be formulated for various categories of noncommutative
probability spaces, e.g., the category of
 $\ast$-algebras with restricted states (i.e., such that the unital extension to the
unitization is positive) or the category of unital $\ast$-algebras with states.
In the present paper we concentrate
on combinatorial aspects and skip questions about positivity. On the other hand,
recently multiple linear functionals on single algebras turned into focus which
give rise to nontrivial spreadability systems.
We therefore consider the
category $\NCP_d$ consisting of objects $(\mathcal{A},\varphi)$, where
$\alg{A}$ is an  algebra and $\varphi=(\varphi^1,\dots, \varphi^d)$ is a
$d$-tuple of $\C$-valued linear functionals on $\alg{A}$.
An arrow 
$f\colon (\alg{D}, \psi^1, \dots, \psi^d) \to (\alg{A}, \varphi^1,\dots,
\varphi^d)$
in this category is an algebra homomorphism $f\colon \alg{D}\to\alg{A}$  such that $\varphi^i \circ f = \psi^i$ for all $i \in [d]$. 

A \emph{universal product} is a bifunctor $\odott$ on $\NCP_d$ of the form    
\[
\left((\alg{A}_1, \varphi_1), (\alg{A}_2, \varphi_2) \right) \mapsto (\alg{A}_1, \varphi_1)\odott (\alg{A}_2, \varphi_2)= (\alg{A}_1\sqcup \alg{A}_2,  \varphi_1\odott\varphi_2),   
\]
where $\alg{A}_1 \sqcup \alg{A}_2$ is the coproduct (also called the nonunital
free product) in the category of associative algebras. This means that $\odot$ is a binary operation on $\NCP_d$ such that 
\begin{equation} \label{eq:bifunctor}
(\varphi_1 \odott \varphi_2) \circ (f_1 \sqcup f_2) = (\varphi_1 \circ f_1) \odott (\varphi_2 \circ f_2)  \tag{U0}
\end{equation}
for any arrows $f_k\colon \alg{D}_k \to \alg{A}_k, k=1,2$, where $f_1 \sqcup f_2$ is the canonical arrow $\alg{D}_1 \sqcup \alg{D}_2 \to \alg{A}_1 \sqcup \alg{A}_2 $. 
For a universal product $\odott$ some of the following conditions are often imposed \cite{ManzelSchurmann:2016:noncommutative}.

\begin{enumerate}[label=\rm(U\arabic*)]

\item\label{U1} \emph{Restriction property}: $(\varphi_1 \odott \varphi_2) \circ \iota_k = \varphi_k$ for each $k\in \{1,2\}$, where $\iota_k$ is the canonical embedding $\alg{A}_k \to \alg{A}_1\sqcup \alg{A}_2$. 

\item\label{U2}  \emph{Associativity}: $(\varphi_1 \odott \varphi_2)\odott \varphi_3 = \varphi_1 \odott (\varphi_2\odott \varphi_3)$  
under the natural isomorphism of $(\alg{A}_1 \sqcup \alg{A}_2) \sqcup \alg{A}_3$ and $\alg{A}_1 \sqcup (\alg{A}_2 \sqcup \alg{A}_3)$.   

\item\label{U3} \emph{Factorization} on length two (``stochastic
 independence''):

 $ (\varphi_1 \odott \varphi_2)(a b) = \varphi_1(a) \varphi_2(b)= (\varphi_1 \odott \varphi_2)(b a)$  for all $a \in \alg{A}_1$ and $b \in \alg{A}_2$. 

\item\label{U4} \emph{Symmetry}: $\varphi_1 \odott \varphi_2 = \varphi_2 \odott \varphi_1$ under the natural isomorphism of $\alg{A}_1 \sqcup \alg{A}_2$ and $\alg{A}_2 \sqcup \alg{A}_1$.

\end{enumerate}

Under condition \ref{U2}, the product $\odott_{i\in [p]}\varphi_i$ can be naturally defined on $\sqcup_{i\in [p]} \alg A_i$ for any $p\in\N$. 
The following property is satisfied by many examples. 
\begin{enumerate}[label=\rm(U\arabic*)]
\setcounter{enumi}{4}
\item\label{U6} \emph{Universal coefficients}: there exists a family of complex numbers 
\[
\{u^j(\sigma, f; \pi): j \in [d], n \in \N, \pi \in \OSP_n, \sigma \in \SP_n, \sigma \le \bar \pi,  f\colon \sigma \to [d]\}
\] such that for every $p\in \N$, $(\alg A_i, \varphi_i) \in \NCP_d
~(i\in[p])$, $n \in \N, i_1, i_2, \dots, i_n\in[p]$ and $X_k \in
\alg{A}_{i_k}~(k \in[n])$, we have for $j \in [d]$ 
\[
\left( \odott_{i\in [p]}\varphi_i\right)^j(X_1 X_2 \cdots X_n) =  \sum_{\substack{\sigma \in \SP_n\\ \sigma \le \bar\pi}} \sum_{f\colon\, \sigma \to [d]} u^j (\sigma, f; \pi) \prod_{S \in \sigma} \varphi_{i(S)}^{f(S)}(X_S),    
\]
where $\pi := \kappa(i_1,i_2,\dots, i_n)$ and $i(S)$ is the common value $i_k$ for $k \in S$ (this number is independent of a choice of $k \in S$ because $\sigma \le \bar\pi$ and $k\mapsto i_k$ is constant on each block of $\pi$). 
\end{enumerate}
In fact,  condition \ref{U6} for $p=2$ only is sufficient: then associativity implies condition \ref{U6} for general $p$. 

\begin{remark} For $d=1$ a binary operation $\odott$ satisfying conditions \eqref{eq:bifunctor}--\ref{U3} and \ref{U6} is called a \emph{quasi-universal product}  in \cite{Muraki:2002:five}. (Note that \eqref{eq:bifunctor} easily follows from \ref{U6}).
  On the other hand, a binary operation $\odott$ satisfying conditions \eqref{eq:bifunctor}--\ref{U3} is called a \emph{natural product} \cite{M03}.
 The main result of \cite{M03} is that \ref{U6} follows from \eqref{eq:bifunctor}--\ref{U3}, i.e., a natural product is a quasi-universal product. 
\end{remark}

\begin{remark}\label{rem:multiface} A more general setup was discussed in
  \cite{ManzelSchurmann:2016:noncommutative}, where the algebras are allowed to
  have an additional structure of \emph{faces}. The theory of cumulants is also
  developed in \cite{ManzelSchurmann:2016:noncommutative} for universal
  products with multistates and multifaces.  It seems that some modifications
  are needed in order to adapt the notion of
  spreadability system to the structure of faces but this issue is not discussed in the present paper.  
\end{remark}

The universal products with \ref{U1}--\ref{U4} for $d=1$ were classified by Ben Ghorbal and Sch\"urmann into three types: tensor, free, Boolean  \cite{BGS2002}.  The universal products with \ref{U1}--\ref{U3} for $d=1$ were then classified by Muraki \cite{M03} into five types: monotone and anti-monotone products in addition to the above three.  
The anti-monotone product is essentially the reversion of the monotone product and therefore omitted from the discussion below. 
With \ref{U1}--\ref{U2} for $d=1$, more examples arise and according to
\cite[p.~7]{Lachs15} and \cite[p.~3]{GL15} the classification is not yet complete.
On the other hand, no classification results are known for $d \ge2$.  

Given a universal product with \ref{U1} and \ref{U2} for $\NCP_d$
one can construct a $d$-tuple of ($\mathbb{C}$-valued) spreadability systems in the following
way: for a single object $(\alg{A}, \varphi)$ of $\NCP_d$, take
countably many copies $(\alg{A}_i, \varphi_i)= (\alg{A}, \varphi), i\in \N$ and
set   
\begin{equation}\label{eq:spreadability_universal_product}
(\alg{U}, \tilde{\varphi}^j)
:=\left(\bigsqcup_{i=1}^\infty\mathcal{A}_i, \left(\odott_{i\in \N} \varphi_i\right)^j\right)
\end{equation}
 for each
$j \in [d]$ with the natural embeddings $\iota^{(i)}\colon \alg{A} = \alg{A}_i
\to \alg{U}, i\in \N$. Note here that the coproduct over $\N$ can be represented
as the tensor algebra
\begin{equation} \label{eq:coproduct}
\bigsqcup_{i=1}^\infty\mathcal{A}_i = \bigoplus_{n \in \N} \bigoplus_{\substack{i_1,i_2,\dots, i_n \in \N\\ i_j \ne i_{j+1} \text{~for all~} j \in [n-1]}} \alg{A}_{i_1} \otimes \alg{A}_{i_2} \otimes \cdots \otimes \alg{A}_{i_n}
\end{equation}
and $\odott_{i\in \N} \varphi_i$ is naturally defined due to
associativity. Condition \ref{U1} readily implies that $\varphi_i^j =
\tilde{\varphi}^j \circ \iota^{(i)}$ for all $i\in \N$ and functoriality of $\odott$ yields
\begin{equation}\label{eq:universal}
\left(\odott_{i\in \N} \varphi_i \right) \circ \left(\bigsqcup_{i\in \N} f_i \right) = \odott_{i\in \N} (\varphi_i \circ f_i)
\end{equation}
for all arrows $f_i\colon \alg{D}_i \to \alg{A}_i, i \in \N$, where $\sqcup_{i\in \N} f_i\colon \sqcup_{i\in\N} \alg{D}_i \to \sqcup_{i\in\N} \alg{A}_i$ is the canonical arrow.  Equation \eqref{eq:universal} guarantees the invariance \eqref{eq:spreadability}, and therefore $\mathcal{S}_{\odott}^j := (\alg{U}, \tilde{\varphi}^j, (\iota^{(i)})_{i=1}^\infty)$ is a spreadability system for the ncps $(\alg A, \varphi^j)$ for every $j\in[d]$. Moreover, it becomes an exchangeability system if $\odott$ is symmetric.  
The preceding construction works for both unital and non-unital algebras.

Each universal product $\odott$ gives rise to a notion of independence. 

\begin{defi} \label{defi:universal_independence}  Let $d \in\N$ and $\odott$ be
  a universal product for $\NCP_d$ satisfying conditions \ref{U1} and
  \ref{U2}. Let $\alg A$ be an algebra and $\varphi = (\varphi^1, \dots,
  \varphi^d)$ be a tuple of $\C$-linear functionals on $\alg A$.

\begin{enumerate}[label=\rm(\roman*),leftmargin=1cm]
\item A sequence $(\alg A_i)_{i\in [p]}$ of subalgebras of $\alg A$ is said to be \emph{$\odott$-independent} if the identity
 \[
\phi^j \circ \tilde{\bigsqcup_{i\in[p]}}\gamma_i = \left[\odott_{i\in [p]} (\phi\circ \gamma_i) \right]^j   \quad \text{on}\quad \bigsqcup_{i\in [p]} \alg A_i,  
 \]
 holds for all $j\in[d]$, where  $\gamma_i\colon \alg A_i \hookrightarrow \alg A$ is the embedding for $i\in [p]$ and $\tilde\sqcup_{i\in[p]}\gamma_i$ is the canonical arrow $\sqcup_{i\in[p]} \alg A_i \to \alg A$.   
 
If we further assume \ref{U6} the above definition is equivalent to the
requirement that for any  $n \in \N, i_1, i_2, \dots, i_n \in [p], X_k \in
\alg{A}_{i_k}~(k \in[n])$ we have  
\begin{equation}\label{eq:univ_calculation}
\varphi^j(X_1 X_2 \cdots X_n) =  \sum_{\substack{\sigma \in \SP_n\\ \sigma \le \bar\pi}} \sum_{f\colon\, \sigma \to [d]} u^j (\sigma, f; \pi) \prod_{S \in \sigma} \varphi^{f(S)}(X_S),    \qquad j \in [d]. 
\end{equation}
 
\item  For an arbitrary totally ordered set $I$,
 a family of subalgebras $(\alg A_i)_{i\in I}$ is called $\odott$-independent
 if any finite subfamily is  $\odott$-independent in the sense above.
 If the universal product $\odott$  in addition satisfies \ref{U4}
 then one can define  $\odott$-independence for a family of subalgebras $(\alg A_i)_{i\in I}$ with \emph{any index set $I$} because one can define $\odott_{i\in I} (\phi\circ \iota^{(i)})$ for any index set $I$. 
 
\item A family $(S_i)_{i \in I}$ of subsets of $\alg A$ with totally ordered index set $I$ is said to be $\odot$-independent if the family of subalgebras $\alg A_i$ generated by $S_i$ is $\odot$-independent. 
\end{enumerate}
\end{defi}

After these general considerations we briefly discuss some explicit examples.

\subsubsection{Tensor exchangeability system}
\label{ssec:tensor}
Our first example  of a universal product is the tensor product $\otimes$ of unital algebras and unital linear
functionals.
Let $(\alg A,\phi)$ be a unital ncps,
 $\alg{U}$ the algebraic
  infinite tensor product $\alg{U}:=\otimes_{i=1}^\infty\mathcal{A}$ and
  $\tilde{\varphi}:=\otimes_{i=1}^\infty \varphi$ be the infinite tensor product
  of copies of $\phi$. 
  Let $\iota^{(j)}$ be the embedding of $\mathcal{A}$ into the $j$th tensor
  component: 
  $$
 \iota^{(j)}(X):= 1^{\otimes(j-1)} \otimes X  \otimes 1 \otimes 1\otimes \cdots.   
  $$  Then $\exch_{\tensorT}=(\alg{U},
  \tilde{\varphi}, (\iota^{(i)})_{i=1}^\infty)$ is an exchangeability system
  for $(\mathcal{A}, \varphi)$, which we call the \emph{tensor exchangeability
    system}. In order to emphasize that $\exch_{\tensorT}$ is a spreadability system, we may write $\cS_{\tensorT}$ instead of $\exch_{\tensorT}$ and call $\cS_{\tensorT}$ the tensor spreadability system. 

A family of subalgebras $(\alg A_i)_{i \in I}$ of $\alg A$ is  $\otimes$-independent if for every $i_1, i_2,\dots, i_n \in I$ and $X_k \in \alg A_{i_k}, k\in [n]$ we have 
\begin{equation}
  \label{eq:otimesindep:factor}
\varphi(X_1 X_2 \cdots X_n) = \prod_{P \in \bar\kappa(i_1,i_2,\dots, i_n)}\varphi(X_P). 
\end{equation}
In other words, the universal coefficients for the tensor spreadability system are given by $u^1 (\bar{\pi}; \pi) = 1$ for all $\pi \in \OSP_n$ and $u^1(\sigma; \pi)=0$ for all $\sigma \in \SP_n \setminus \{\bar \pi\}$. Note that we here omit $f$ because it is unique. 

\subsubsection{Free exchangeability system} The reduced free product of unital linear
functionals is another example of a universal
product.
Let $\alg{U}:=\ast_{i=1}^\infty \mathcal{A}$ be the unital free
  product of infinitely many copies of a unital algebra $\mathcal{A}$
  and let $\tilde{\varphi}:=\ast_{i=1}^\infty \varphi$ be the free product
  of copies of a unital linear functional $\phi$
  \cite{Avitzour:1982:free,V85}.
  Let $\iota^{(i)}$ be the embedding of $\mathcal{A}$ into the $i$th
  component $\mathcal{A}$ of $\alg{U}$. Then $\exch_{\freeF}=(\alg{U},
  \tilde{\varphi}, (\iota^{(i)})_{i=1}^\infty)$ (or we may write $\cS_{\freeF}$ when emphasizing the spreadability) 
  is an exchangeability system for $(\mathcal{A}, \varphi)$, called the \emph{free exchangeability (or spreadability) system}.

\subsubsection{Boolean exchangeability system}
\label{sssec:boolean} 
  The Boolean product
  $\tilde{\varphi}=\diamond_{i=1}^\infty \varphi$ is defined on the
  nonunital free product (i.e.\ the tensor algebra) 
  $\alg{U}:=\bigsqcup_{i=1}^\infty \mathcal{A}$ by the following rule \cite{Bozejko:1986:positive}: if $X_1, X_2,\dots, X_n \in \mathcal{A}$ and
  $i_{k}\neq i_{k+1}$ for any $1 \leq k \leq n-1$, then
  $$
  \tilde{\varphi} (X_1^{(i_1)} X_2^{(i_2)} \cdots X_n^{(i_n)})
  = \varphi (X_1)\, \varphi(X_2) \cdots \varphi (X_n).  
  $$
  As before, $\iota^{(j)}$ is the
  embedding of $\mathcal{A}$ into the $j$th component $\mathcal{A}$ of
  $\alg{U}$.
  The triplet $\exch_{\boolB}=(\alg{U}, \tilde{\varphi}, (\iota^{(i)})_{i=1}^\infty)$ (or we may write $\cS_{\boolB}$) is  called the \emph{Boolean exchangeability (or spreadability) system}.

\subsubsection{Monotone spreadability system}
\label{sssec:monotone} 
 Let $(\alg{U}, (\iota^{(i)})_{i=1}^\infty)$ be as in Section~\ref{sssec:boolean}. 
 The monotone product
  $\tilde{\varphi}=\triangleright_{i=1}^\infty \varphi$ is defined on
  $\alg{U}$ by the following recursive rules \cite{M00}: for every $n \in \N$, $X_1,X_2,\dots, X_n \in \alg A$ and $i_1,i_2,\dots, i_n \in \N$, 
\begin{enumerate}[label=\rm(\roman*)]

 \item\label{item:monotone1} $\tilde{\varphi}(X_1^{(i_1)})=\varphi (X_1)$;

 \item $\tilde{\varphi} (X_1^{(i_1)} X_2^{(i_2)} \cdots X_n^{(i_n)})
  = \varphi (X_1)\,\tilde{\varphi}(X_2^{(i_2)} \cdots X_n^{(i_n)})$ \text{~if~}  $n\ge2$ and $i_1 > i_2$; 
 
 \item\label{item:monotone3} $\tilde{\varphi} (X_1^{(i_1)} X_2^{(i_2)} \cdots X_n^{(i_n)})
   = \tilde{\varphi}(X_1^{(i_1)} X_2^{(i_2)} \cdots X_{n-1}^{(i_{n-1})})\,\varphi (X_n)$
    \text{~if~}  $n\ge2$ and $i_n > i_{n-1}$; 
    
 \item $\tilde{\varphi} (X_1^{(i_1)} X_2^{(i_2)} \cdots X_n^{(i_n)})
  =\tilde{\varphi} (X_1^{(i_1)} X_2^{(i_2)} \cdots  X_{j-1}^{(i_{j-1})} X_{j+1}^{(i_{j+1})} \cdots X_n^{(i_n)})\, \varphi
  (X_j)$
  if $n\ge3$, $2 \leq j \leq n-1$ and $i_{j-1} < i_j > i_{j+1}$.
\end{enumerate} 
Then $\cS_{\monoM}=(\alg{U}, \tilde{\varphi}, (\iota^{(i)})_{i=1}^\infty)$ is called the \emph{monotone spreadability system} for $(\mathcal{A}, \varphi)$.
It is a proper spreadability system, i.e., it does not satisfy exchangeability.

\subsubsection{Conditionally monotone spreadability system}\label{sssec:cmonotone}
 Let $(\alg{U}, (\iota^{(i)})_{i=1}^\infty)$ be as in Section~\ref{sssec:boolean}. 
The conditionally monotone product is an associative universal product for $d=2$.   
The infinite conditionally monotone product $(\alg{U},\tilde{\varphi},
 \tilde{\psi})=\triangleright_{i=1}^\infty (\alg{A},\varphi, \psi)$ 
is defined as follows  \cite{Hasebe:2011:conditionally}. 

\begin{itemize}
\item $\tilde{\psi}= \triangleright_{i=1}^\infty \psi$ is the monotone product of $\psi$ according to Section~\ref{sssec:monotone}; 
\item $\tilde\varphi$ is determined by rules
 \ref{item:monotone1}--\ref{item:monotone3} as in the monotone case from
 Section~\ref{sssec:monotone}
 but with the last rule modified into
\end{itemize}

\begin{enumerate}[label=\rm(\roman*')]
  \setcounter{enumi}{3}
 \item 
$
\begin{multlined}[t]
    \tilde{\varphi}  (X_1^{(i_1)} X_2^{(i_2)} \cdots X_n^{(i_n)})
\\
    = \tilde{\varphi}  (X_1^{(i_1)} X_2^{(i_2)} \cdots X_{j-1}^{(i_{j-1})})\, 
      (\varphi (X_j) - \psi (X_j)) \,
      \tilde{\varphi} (X_{j+1}^{(i_{j+1})} \cdots X_n^{(i_n)}) 
\\
    + \psi (X_j)\, \tilde{\varphi}(X_1^{(i_1)} X_2^{(i_1)}\cdots X_{j-1}^{(i_{j-1})} X_{j+1}^{(i_{j+1})} \cdots X_n^{(i_n)})
  \end{multlined}
  $

\noindent
if $n\ge3$, 
$2 \leq j \leq n-1$  and $i_{j-1} < i_j > i_{j+1}$.
\end{enumerate}

Then $\cS_{\mathrm{CM}}=(\alg{U}, \tilde{\varphi}, (\iota^{(i)})_{i=1}^\infty)$ is a
spreadability system for $(\mathcal{A}, \varphi)$ 
which does not satisfy exchangeability.  
It is called the \emph{c-monotone spreadability system}. 
  
\begin{remark}
More examples may be extracted from associative universal products in \cite{BLS96,H10} for $d=2$ or $d=3$, but we omit them here.   
\end{remark}

\subsubsection{$V$-monotone spreadability system}
\label{sssec:vmonotone}
Recently Dacko introduced
the concept of $V$-monotone
independence and constructed a corresponding $V$-monotone product of
probability spaces \cite{Dacko:2019}.
These notions are based on the notion of $V$-shaped sequences and partitions.
A sequence of numbers $i_1,i_2,\dots,i_n$ is called \emph{$V$-shaped}
if there exists an index $1\leq r\leq n$ such that
$i_1>i_2>\dots>i_r<i_{r+1}<\dots<i_n$. 
Given a \emph{unital} $\C$-ncps $(\alg A,\varphi)$,  the $V$-monotone product $\tilde \phi=\ovee_{i\in\N}\phi$ is
defined on the \emph{nonunital} free product $\alg{U} := \bigsqcup_{i=1}^\infty \alg A$  
and characterized by the following factorization properties. Let $\alg{A}_i \subseteq \alg{U}$ denote the embedded image $\iota^{(i)}(\alg A)$ of $\alg A$ with unit denoted by 
$I_i$ and $\varphi_i$ be the induced linear functional $\varphi\circ (\iota^{(i)})^{-1}$ on $\alg A_i$. Let $n\in\N$ and $Y_j\in\alg{A}_{i_j}$, $j=1,2,\dots,n$ be arbitrary elements.
\begin{enumerate}[label=\rm(\roman*),leftmargin=0.5cm]
 \item $\tilde\phi(Y_1Y_2\dotsm Y_n)=0$ whenever $i_j\ne i_{j+1}$ for all $j$
  and $\phi_{i_j}(Y_j)=0$.
 \item In addition, for every $j \in[n]$, 
  $$
 \tilde \phi(Y_1Y_2\dotsm Y_{j-1}I_{i_j}Y_{j+1}\dotsm Y_n)=
  \begin{cases}
    \tilde\phi(Y_1Y_2\dotsm Y_{j-1}Y_{j+1}\dotsm Y_n) &\text{if $(i_1,i_2,\dots,i_j)$
      is $V$-shaped,}\\
      0 &\text{otherwise,}
  \end{cases}
  $$
  whenever $\phi_{i_1}(Y_1)=\phi_{i_2}(Y_2)=\dots=\phi_{i_{j-1}}(Y_{j-1})=0$.
\end{enumerate}
It is shown in  \cite{Dacko:2019} that
associativity does not hold; yet  identity \eqref{eq:universal} holds
and the $V$-monotone product gives  rise to a spreadability system.

\subsection{Spreadability systems with calculation rules}

Motivated by Axiom \ref{U6} in the previous subsection, we are lead to the following class of spreadability systems which provides various examples.

\begin{defi}\label{defi:universal} Let $\cS=(\alg{U}, \tilde{\varphi},(\iota^{(i)})_{i \geq 1})$ be
  a spreadability system for a $\C$-ncps $(\mathcal{A}, \varphi)$.   For a set partition $\pi\in\SP_n$ we define a multiplicative extension
  $\varphi_{(\pi)}$ of $\varphi$ by setting
  \begin{align} \label{eq:multiplicative}
    &\varphi_{(\pi)}(X_1,X_2,\ldots, X_n) :=\prod_{P\in \pi} \varphi(X_P) \quad \text{for}\quad  X_1, X_2,\dots, X_n \in \alg A,
  \end{align}
  where we used notation \eqref{product}  for the ordered product $X_P$.

 Then $\cS$ is said to have a \emph{calculation rule} if there exists a family of complex numbers
  $\fC= \{s(\sigma; \pi)\in\C: n \in \N, \pi\in \OSP_n, \sigma\in\mathcal{P}_n, \sigma \leq
  \bar{\pi}\}$ such that the equality
\begin{equation}\label{eq:calculation_rule1}
\varphi_\pi=\sum_{\substack{\sigma\in\mathcal{P}_n,\\ \sigma \leq \bar{\pi}}} s(\sigma; \pi) \, \varphi_{(\sigma)}
\end{equation}
holds as functionals on $\mathcal{A}^n$ for all $\pi \in \OSP_n$ and $n\in \N$. Note that we can always take $s(\{1\}; (1))=1 $ because of $\tilde \varphi \circ \iota^{(i)} = \phi$. 
 \end{defi}
 
 \begin{remark}
   \begin{enumerate}[label=\arabic*.,leftmargin=1cm]
    \item []
    \item 
 It is straightforward to see that a spreadability system $\cS_\odot^1$
 constructed from a universal product $\odot$ on $\NCP_1$ subject to axioms
 \ref{U1}, \ref{U2} and \ref{U6} has a calculation rule with $s(\sigma;
 \pi):=u^1(\sigma; \pi)$ (where $f$ is omitted because there is a unique
 $f\colon \sigma \to [1]$).  

\item  
Actually we can construct such a spreadability system for any given family of constants  $\fC =\{s(\sigma; \pi)\in\C: n \in \N, \pi\in \OSP_n, \sigma\in\mathcal{P}_n, \sigma \leq
  \bar{\pi}\}$ with  $s(\{1\}; (1))=1 $ and any given $\C$-ncps $(\alg A,
  \varphi)$. Let $\alg{U}$ be the coproduct (i.e., the tensor algebra) of the
  countable copies of $ \alg{A}$ in the category of associative algebras and 
  $\iota^{(i)}\colon \alg A\to \alg{U}$ the natural embedding as the $i$th component. Then we can define a linear functional $\tilde{\varphi}$ on $\alg{U}$ as follows: for each tuple $X_j \in \alg{A}~(1 \le j \le n)$ and indices $i_1,i_2,\dots, i_n \in \N$ with $i_k \ne i_{k+1}$ ($k\in[n-1]$) we set 
  \begin{equation}\label{eq:calculation_rule}
  \tilde \varphi (X_1^{(i_1)} X_2^{(i_2)} \cdots X_n^{(i_n)}) := \sum_{\substack{\sigma\in\mathcal{P}_n,\\ \sigma \leq \bar{\kappa}(i_1,\dots, i_n) }} 
             s(\sigma; \kappa(i_1,\dots, i_n)) \, \varphi_{(\sigma)} (X_1,X_2,\dots, X_n). 
  \end{equation}
The fact that the value of \eqref{eq:calculation_rule} depends only on the partition $\kappa(i_1,i_2,\dots,i_n)$ induced by the sequence $(i_1,i_2,\dots, i_n)$ yields that $\cS_{\fC}:= (\alg{U}, \tilde\varphi, (\iota^{(i)})_{i\ge1})$ is indeed a spreadability system.  
\item 
Muraki's example \cite{Muraki:private} yields a spreadability system with
calculation rule coming from a nonassociative universal product.
The V-monotone spreadability system is also such an example. We will investigate other instructive examples later in Examples \ref{exa:nonassociative0} and  \ref{exa:nonassociative}. 

\item 
 It is worth mentioning that there are spreadability
 systems without calculation rules. For instance, the algebra generated by left and right creation operators acting
 on $q$-Fock space gives rise to an exchangeability system
 \cite{Lehner:2005:cumulants3}, but it was shown in \cite{LeeuwenMaassen:1996:obstruction}
 that in this case the individual distributions of independent elements with respect to the vacuum expectation 
 is not sufficient to determine the joint distribution
 with respect to $(\varphi_\pi)_{\pi}$,
 and in particular there is no ``$q$-convolution''.  The c-monotone spreadability system $\cS_{\rm CM}$ also does not have a calculation rule since the evaluation of $\varphi_\pi$ depends on both $\varphi$ and $\psi$ in general.
\end{enumerate}
\end{remark}

\subsection{$\cS$-independence}   
Let us first recall the notion of independence associated to exchangeability systems \cite[Definition~1.8]{L04}. 
Roughly speaking  independence of a pair $(X,Y)$ means that 
the joint distribution of $(X,Y)$ coincides with the joint distribution of $(X^{(1)},Y^{(2)})$, where the couples
$(X^{(1)},Y^{(1)})$ and $(X^{(2)},Y^{(2)})$ are exchangeable copies of the couple $(X,Y)$.
This property can be reformulated in a lattice theoretical way as follows. 
\begin{defi}[$\exch$-independence]
  \label{def:exchangeableindependence}   
   Let $\exch=(\alg{U},\tilde{\phi},(\iota^{(i)})_{i\ge1})$ be an exchangeability system for a $\alg{B}$-ncps $(\alg{A},\phi)$.  Let $I\subseteq \N$. 
  
\begin{enumerate}[label=\rm(\roman*),leftmargin=1cm]
\item Subalgebras $(\alg{A}_i)_{i\in I}$ of $\alg{A}$ are said to be
  \emph{$\exch$-independent} if for any tuple of indices $(i_1,i_2,\dots,i_n)\in I^n$, any tuple of random variables $(X_1,X_2,\dots,X_n)$ with $X_j\in \alg{A}_{i_j}$ and any set partition $\pi\in\SP_n$, we have  
  \begin{equation}\label{eq:E-indep}
  \phi_\pi(X_1,X_2,\dots,X_n) = \phi_{\pi\wedge \bar{\kappa}(i_1,i_2,\dots,i_n)}(X_1,X_2,\dots,X_n). 
  \end{equation}

\item Subsets  $(S_i)_{i \in I}$ of $\alg A$  are said to be
 \emph{$\exch$-independent} if the algebras $\alg A_i$ generated by 
 $S_i$ are 
 $\exch$-independent. In particular, a family
 $
 \bigl(
 (X_1(i),X_2(i),\dots,X_n(i))
 \bigr)_{i\in I}$ of random vectors
 is said to be
 \emph{$\exch$-independent} if the subalgebras $\alg A_i$
 generated by its respective entries  $\{X_1(i),X_2(i), \dots, X_n(i)\}$
 are $\exch$-independent.  
\end{enumerate}
\end{defi}

\begin{remark} In the case where $\alg A, \alg B, \alg{U}$ and $\phi,
  \iota^{(i)}$ are unital,
  it is more natural to assume that the subalgebras $\alg A_i$ contain the unit
  of $\alg A$.
  The same applies to Definition \ref{def:S-independence}. 
\end{remark}

In order to generalize independence from exchangeability systems to spreadability systems
we replace set partitions by ordered set partitions.
This time independence of an ordered pair $(X,Y)$ means that 
the joint distribution of $(X,Y)$ coincides with the joint distribution of
$(X^{(1)},Y^{(2)})$
(but not necessarily $(X^{(2)},Y^{(1)})$), 
where the couples
$(X^{(1)},Y^{(1)})$ and $(X^{(2)},Y^{(2)})$ are spreaded copies of the couple $(X,Y)$.

\begin{defi}[$\cS$-independence]  \label{def:S-independence}
  Let $\cS=(\alg{U},\tilde{\varphi},(\iota^{(i)})_{i\ge1})$ be a spreadability system for a given
  $\alg{B}$-ncps $(\alg{A},\phi)$. Let $I \subseteq \N$. 
  
  \begin{enumerate}[label=(\roman*), leftmargin=1cm]
\item  A sequence of subalgebras $(\alg{A}_i)_{i\in I}$ of $\alg{A}$ is said to be
  \emph{$\cS$-independent} if for any tuple of indices
  $(i_1,i_2,\dots,i_n)\in I^n$, 
  any tuple of random variables $(X_1,X_2,\dots,X_n)$
  with $X_j\in  \alg{A}_{i_j}$
  and any ordered set partition $\pi\in\OSP_n$, we have
  \begin{equation}\label{S-indep} 
  \phi_\pi(X_1,X_2,\dots,X_n) = \phi_{\pi\curlywedge \kappa(i_1,i_2,\dots,i_n)}(X_1,X_2,\dots,X_n).  
  \end{equation}
  
\item A sequence of subsets $(S_i)_{i \in I}$ is said to be
\emph{$\cS$-independent} if the sequence  of subalgebras $\alg A_i$ generated by $S_i$ is $\cS$-independent. 
In particular, a sequence of random vectors $((X_1(i),X_2(i),\dots,X_n(i)))_{i\in I}$ is said to be \emph{$\cS$-independent} if the sequence of subalgebras $\alg A_i$ generated by $\{X_1(i),X_2(i),\dots, X_n(i)\}$ is $\cS$-independent. 
\end{enumerate}
\end{defi}

\begin{exa}
For two subalgebras $\cS$-independence reads as follows.
A pair of subalgebras
$(\mathcal{A}_1, \mathcal{A}_2)$ of $\mathcal{A}$ is
$\cS$-independent if the following condition holds.
Given elements $X_1,X_2,\dots,X_n\in \alg{A}_1\cup \alg{A}_2$,
let $\rho=(B_1,B_2)$ be an ordered set partition of the index set $[n]$ such that
$X_i\in \alg{A}_1$ for $i\in B_1$ and $X_i\in \alg{A}_2$ for $i\in B_2$.
Then for any ordered set partition $\pi\in \OSP_n$ we have
$$
\phi_\pi(X_1,X_2,\dots, X_n) = \phi_{\pi\restr_{B_1}\pi\restr_{B_2}}(X_1,X_2,\dots, X_n).  
$$
\end{exa}

\begin{remark}
Let $\exch=(\alg{U},\tilde{\varphi},(\iota^{(i)})_{i\ge1})$ be an
exchangeability system for a given $\alg{B}$-ncps $(\alg{A},\phi)$. It can be
regarded as a spreadability system and we denote it by  $\cS$.  Then, obviously,  a sequence of subalgebras $(\alg{A}_i)_{i\in I}$ of $\alg{A}$ is $\exch$-independent in the sense of Definition  \ref{def:exchangeableindependence} if and only if it is $\cS$-independent in the sense of Definition \ref{def:S-independence}. 
\end{remark} 

\begin{remark}
  \begin{enumerate}[label=\arabic*., leftmargin=5ex]
   \item []
   \item 
With the notation introduced in Definition~\ref{def:dot} below, 
equation (\ref{S-indep}) may be rewritten as   
  $$
  \phi_\pi(X_1,X_2,\dots,X_n) = \phi_{\pi}(X_1^{(i_1)},X_2^{(i_2)},\dots,X_n^{(i_n)}).  
  $$
This condition means that the random vectors $(X_k)_{k\in [n]}$ and
$(X_k^{(i_k)})_{k\in [n]}$ have the same ``distribution'' with respect to
$(\varphi_\pi)_\pi$.  This is compatible with the concept of spreadability,
which is to regard the sequence $(\iota^{(j)}(\mathcal{A}))_{j \in I}$ as
independent copies $\alg A$ constructed in $(\alg{U}, \tilde{\varphi})$. 
\item 
An alternative natural definition of $\cS$-independence would be to require the simplified condition that for any $(i_1,i_2,\dots,i_n)\in I^n$ and any 
 $X_j\in  \alg{A}_{i_j}, j \in [n]$, 
\begin{equation}\label{eq:weaker}
  \phi (X_1 X_2  \cdots X_n) = \tilde\phi(X_1^{(i_1)} X_2^{(i_2)} \cdots X_n^{(i_n)}), 
\end{equation}
i.e., requiring \eqref{S-indep} only for $\pi =\hat1_n$.
In fact, for the spreadability systems associated with  a large class of
universal products,
\eqref{eq:weaker} implies the $\cS$-independence, see Proposition
\ref{prop:equivalence_independences}.
In general however, condition \eqref{eq:weaker} does not imply $\cS$-in\-de\-pen\-den\-ce, see Example \ref{exa:nonassociative0}. 
\item  In the case of exchangeability, $\exch$-independence is equivalent to the
vanishing of mixed cumulants (see Proposition \ref{prop:mixedvanish}).

  \end{enumerate}
 \end{remark}

The following result shows that the definitions of $\exch$- and $\cS$-independence coincide with the usual definitions for typical examples, e.g.\ tensor, free, Boolean, monotone, c-free and c-monotone independences.  Associativity is crucial. 

\begin{prop} \label{prop:equivalence_independences}  Let $d \in \N$ and $\odott$ be a universal product in the category $\NCP_d$ satisfying \ref{U1}, \ref{U2} and \ref{U6}. Let $(\alg A, \varphi^1,\varphi^2, \dots, \varphi^d)$ be an object in $\NCP_d$ and $(\alg A_i)_{i \in I}$ be a family of subalgebras of $\alg A$ with $I\subseteq \N$. Then $(\alg A_i)_{i \in I}$ is $\odot$-independent if and only if it is $\cS_{\odott}^j$-independent for all $j\in[d]$, where 
 $\cS_{\odott}^j$ is the spreadability system for the $\C$-ncps $(\alg A, \varphi^j)$ constructed in Section \ref{exaspread}. 
\end{prop} 

\begin{proof} 
Suppose that $(\alg A_i)_{i \in I}$ is $\odot$-independent.  This means that  the identity  
\begin{equation}\label{eq:S-indep}
  \phi_\pi^j(X_1,X_2,\dots,X_n) = \phi_{\pi\curlywedge \kappa(i_1,i_2,\dots,i_n)}^j(X_1,X_2,\dots,X_n)  
\end{equation}
holds for $\pi = \hat 1_n$, any $j \in[d]$,  any tuple $(i_1,i_2,\dots,i_n)\in I^n$ and  any tuple of random variables $(X_1,X_2,\dots,X_n)$
  with $X_k\in  \alg{A}_{i_k}$, cf.\ \eqref{eq:univ_calculation}.   The goal is to verify \eqref{eq:S-indep} for any ordered set partition $\pi =(P_1,P_2,\dots, P_p)\in\OSP_n$. 
  
  Let us denote $\kappa(i_1,i_2,\dots,i_n) = (Q_1,Q_2,\dots, Q_q)$. Then $\pi\curlywedge \kappa(i_1,i_2,\dots,i_n) = (P_1 \cap Q_1, P_1 \cap Q_2, \dots, P_1 \cap  Q_q, P_2\cap Q_1, \dots, P_p \cap Q_q)$.  The RHS of \eqref{eq:S-indep} is exactly the value $\varphi^j(X_1X_2 \cdots X_n)$ when 
  \[
  (\{X_i\}_{i\in P_1\cap Q_1}, \{X_i\}_{i\in P_1 \cap Q_2}, \dots, \{X_i\}_{i\in P_1 \cap Q_q}, \{X_i\}_{i\in P_2 \cap Q_1}, \{X_i\}_{i\in P_2 \cap Q_2}, \dots)
  \]
  is assumed to be $\odot$-independent. The associativity of $\odot$ allows us to compute the RHS of \eqref{eq:S-indep} in the following two steps:    
  \begin{enumerate}[label=\rm(\alph*)]
\item\label{stepa} First compute $\varphi^j(X_1X_2 \cdots X_n)$ assuming that $(\{X_i\}_{i\in P_1}, \{X_i\}_{i\in P_2}, \dots, \{X_i\}_{i\in P_p})$ is $\odot$-independent. The result is exactly the RHS of \eqref{eq:univ_calculation}.   
\item\label{stepb} Then compute each factor $\varphi^{f(S)}(X_{S})$ by additionally assuming that 
\[
(\{X_i\}_{i\in S\cap Q_1}, \{X_i\}_{i\in S\cap Q_2}, \dots, \{X_i\}_{i\in S\cap Q_q})\] 
is $\odot$-independent.  
  \end{enumerate}
  This is exactly how the LHS of \eqref{eq:S-indep} is computed, so that \eqref{eq:S-indep} holds as desired.  
  
Conversely,  suppose that  $(\alg A_i)_{i \in I}$ is  $\cS_{\odott}^j$-independent for all $j\in[d]$. Formula \eqref{eq:S-indep} for  $\pi = \hat1_n$ is exactly the desired independence relation \eqref{eq:univ_calculation}. 
\end{proof}

\begin{exa} \label{exa:nonassociative0}
Let $(\alg{A}, \varphi)$ be a $\C$-ncps and let $\alg{U}:= \sqcup_{i=1}^\infty \alg{A}$ be the coproduct in the category of algebras (see \eqref{eq:coproduct}) and $\iota^{(i)} \colon \alg{A} \to \alg{U}$ be the natural embedding as $i^{\rm th}$ component. We define a linear functional $\tilde{\varphi}$ on $\alg{U}$ by, for each $X_1, X_2 , \dots, X_n \in \alg{A} $ and $i_1,i_2,\dots, i_n \in \N$ with $i_j \ne i_{j+1}$ for all $j =1,2,\dots, n-1$, setting 
\[
\tilde{\varphi}(X_1^{(i_1)} X_2^{(i_2)} \cdots  X_n^{(i_n)}) 
= 
\begin{cases} 
\phi(X_1) & \text{~if $n=1$}, \\ 
\phi(X_1) \phi(X_2) &  \text{~if $n =2$}, \\
0& \text{~if $n\ge3$}. 
\end{cases}
\]
Then we get an exchangeability system $(\alg{U}, \tilde{\varphi}, (\iota^{(i)})_{i \in \N})$ for $(\alg{A}, \varphi)$ with calculation rule.

For  $\pi \in \OSP_n$ and $X_1,X_2,\dots, X_n \in \alg{A}$ it is easy to see that  
 \begin{equation} \label{eq:phi_pi}
      \phi_\pi(X_1,X_2,\dots,X_n) =   
      \begin{cases} 
      \phi_{(\pi)}(X_1,X_2,\dots,X_n) & \text{~if~$\pi \in \IP_n$ with $|\pi|\le2$,} \\ 
0 & \text{~otherwise}. 
      \end{cases}
    \end{equation} 
Suppose that subalgebras $\alg B_1, \alg B_2 \subseteq \alg A$ are $\exch$-independent.  Then for $X \in \alg B_1$ and $Y \in \alg B_2$ we have 
$$
\varphi_{\pi}(X, Y)  = \varphi_{\pi \wedge \bar\kappa(1,2)}(X,Y)  
$$  
and the case $\pi =\hat1_2$ yields $\varphi(X Y) = \varphi(X) \varphi(Y)$. 
Moreover, independence implies 
$$
\varphi_{\pi}(X, Y, X)  = \varphi_{\pi \wedge \bar\kappa(1,2,1)}(X,Y,X). 
$$
On the other hand,
we infer from \eqref{eq:phi_pi} that
for $\pi = \{\{1,2\}, \{3\}\}$ the left hand side equals $\varphi(X Y) \varphi(X)
=\varphi(X)^2\varphi(Y)$, while the RHS equals zero.
This is only possible if either $\varphi \restr_{\alg B_1}=0$ or
$\varphi\restr_{\alg B_2}=0$ holds,
so that only the trivial examples satisfy $\exch$-independence. However, nontrivial subalgebras $\{\alg C_1, \alg C_2\}$ satisfying condition \eqref{eq:weaker} exist because \eqref{eq:weaker} simply reads: if $X_j \in \alg C_{i_j}$ then for $\pi := \bar \kappa (i_1,i_2,\dots, i_n)$
 \begin{equation} \label{eq:phi_pi2}
      \phi(X_1X_2\cdots X_n) =   
      \begin{cases} 
      \phi_{(\pi)}(X_1,X_2,\dots,X_n) & \text{~if~$\pi \in \IP_n$ with $|\pi|\le2$,} \\ 
0 & \text{~otherwise}. 
      \end{cases}
    \end{equation} 
 For example we can start from any $\C$-ncps $(\alg C_i,\phi_i), i=1,2$ and then construct $\phi$ on $\alg A := \alg C_1 \sqcup \alg C_2$ so that  \eqref{eq:phi_pi2} holds.    
 \end{exa}

\section{Cumulants for spreadability systems}
\label{sec:cumulants}

\subsection{Definition and uniqueness of cumulants}
Cumulants provide a powerful tool to describe independence of random
variables.
Let us first recall the case of exchangeability systems.
\begin{definition}
  \label{def:exch:cumulants}
Let $\exch=(\alg{U},\tilde{\varphi},(\iota^{(i)})_{i\ge1})$ be an
exchangeability system for a $\alg{B}$-ncps $(\alg{A},\phi)$.  Cumulants are functionals satisfying the
the following requirements.
\begin{enumerate}[label=\rm(E\arabic*)] 
 \item\label{it:Ecumulants:multi}
  \emph{Multilinearity}: Cumulants are multilinear functionals $K_\pi\colon \alg{A}^n\to\alg{B}$, indexed
  by set partitions $\pi\in\SP_n$,   $n\in\N$.
   \item \emph{Universality}:
  There are universal coefficients $c(\sigma,\pi) \in \C$ such that
  \begin{equation}
    \label{eq:phipi=Kpi+sum}
  \phi_\pi = K_\pi + \sum_{\substack{\sigma\in\SP_n\\ \sigma<\pi}} c(\sigma,\pi) K_\sigma.
  \end{equation}

 \item\label{it:Ecumulants:vanish}  
  \emph{Vanishing of mixed cumulants}:
  given  a family $(X_1,X_2,\dots,X_n)\in\alg{A}^n$ and
  a partition $\pi\in\SP_n$,
  such that some block of $P\in\pi$ can be partitioned
  $P=P' \dot\cup P''$ nontrivially in such a way that
  $\{X_i\mid i\in P'\}$ and   $\{X_i\mid i\in P''\}$ are $\exch$-independent,
  then
  $$
  K_\pi(X_1,X_2,\dots,X_n)=0
  .
  $$
  Equivalently, if $\rho\in\SP_n$  partitions the family $(X_1,X_2,\dots,X_n)$
  into
  mutually $\exch$-independent subfamilies,
  then $K_\pi(X_1,X_2,\dots,X_n)=0$ unless $\pi\leq\rho$.
\end{enumerate}
\end{definition}

\begin{remark}  \ 
  \begin{enumerate}[label=\rm(\roman*),leftmargin=1cm]
\item \emph{Additivity}.
    As a consequence of
    \ref{it:Ecumulants:multi}
    and
    \ref{it:Ecumulants:vanish}
    cumulants of sums of independent random variables are additive:
    Denote by $K_n$ the cumulant functional $K_{\hat{1}_n}$, then
    \begin{equation}
      \label{eq:exch:additivity}
    K_n(X_1+Y_1,X_2+Y_2,\dots,X_n+Y_n) =
    K_n(X_1,X_2,\dots,X_n) +K_n(Y_1,Y_2,\dots,Y_n)
    .
    \end{equation}

   \item
    The system of equations \eqref{eq:phipi=Kpi+sum} is in triangular form and
    can be solved recursively and transformed into the equivalent system
    \begin{equation*}
K_\pi = \phi_\pi  + \sum_{\substack{\sigma\in\SP_n\\ \sigma<\pi}} \tilde{c}(\sigma,\pi) \phi_\sigma
    \end{equation*}
    where the matrix $\tilde{c}(\sigma,\pi)$ is inverse to $c(\sigma,\pi)$.
  \end{enumerate}
\end{remark}
In \cite{L04} the second named author established 
a unified theory of cumulants for ex\-change\-abi\-li\-ty systems
based on a kind of finite Fourier transform,
known as Good's formula in the mathematics literature  \cite{Good:1975:new}
and Cartier's formula for the so-called \emph{Ursell functions} in the physics literature
\cite{Percus:1975:correlation,Simon:1993:statistical}.
This approach apparently fails in the present, non-exchangeable setting; 
however in their study of monotone cumulants~\cite{HS11b,HS11a}
the first named author and Saigo found a good replacement in Rota's dot operation
from umbral calculus \cite{RotaTaylor:1994:classical},
i.e., a weak version of       \eqref{eq:exch:additivity},
which we take as a starting point for the definition of cumulants in full
generality in Definition \ref{def:spreadableK2} below.
\begin{defi}\label{def:dot}
 Let $(\alg{U}, \tilde{\varphi},(\iota^{(i)})_{i \geq 1})$ be a spreadability system for a $\alg{B}$-ncps $(\mathcal{A}, \varphi)$. 
\begin{enumerate}[label=\rm(\roman*),leftmargin=1cm]
\item  \label{eq:deltaAX} 
 Given a noncommutative random variable $X \in \mathcal{A}$ 
 and a finite subset $A \subseteq \N$ we define
 \begin{equation*} 
   \delta_A(X)=\sum_{i\in A} X^{(i)}   
 \end{equation*}
i.e., the sum of i.i.d.\ copies of $X$.
In the case $A=[N]$ we will also write $\delta_N(X)$ and frequently
abbreviate it using Rota's \emph{dot operation} 
$$
N.X:=X^{(1)} + X^{(2)} + \cdots + X^{(N)}
$$ 
whenever it is convenient.

\item We extend the partitioned functionals $\phi_\pi$ to $\bigcup \iota^{(i)}(\alg{A})$ by setting
 \begin{equation}
   \label{eq:phipixjij}
 \phi_{\pi}(X_1^{(i_1)}, X_2^{(i_2)}, \ldots,X_n^{(i_n)}):=\phi_{\pi \curlywedge \kappa(i_1,i_2,\ldots,i_n)}(X_1,X_2,\ldots,X_n),    
 \end{equation}
and 
\begin{equation}
     \label{eq:phipi(NX)}
 \phi_\pi(N_1.X_1,N_2.X_2,\ldots,N_n.X_n):=
\sum_{i_1=1}^{N_1} 
\sum_{i_2=1}^{N_2} 
\dots
\sum_{i_n=1}^{N_n} 
 \phi_{\pi}(X_1^{(i_1)}, X_2^{(i_2)}, \ldots,X_n^{(i_n)}). 
\end{equation}

\end{enumerate}

\end{defi}

\begin{remark}
  Note that \eqref{eq:phipixjij} is actually an abuse of notation,
  because $\phi_\pi$ is defined for elements of $\mathcal{A}$ only 
  (see~\eqref{eq:phipi});
  here we pretend that 
  $\cS=(\alg{U}, (\iota^{(i)})_{i=1}^\infty, \tilde{\varphi})$
  can be interpreted as a spreadability system for the algebra
  $\mathcal{A}^{(1,2,\dots,N)}$  generated by the images
  $\mathcal{A}^{(1)},\mathcal{A}^{(2)},\dots,\mathcal{A}^{(N)}$.
  This is true in the case of universal product construction in \eqref{eq:spreadability_universal_product} 
  but needs justification otherwise;
  yet \eqref{eq:phipixjij} is well-defined
  and convenient to keep notation manageable.
\end{remark}

\begin{defi}
  \label{def:spreadableK2}
  Let $(\alg{A},\phi)$ be a $\alg{B}$-ncps and $\cS=(\alg{U},\tilde{\varphi},(\iota^{(i)})_{i\ge1})$ a
  spreadability system for $(\alg{A},\varphi)$.
  \emph{Cumulants} are multilinear functionals $K_\pi$ indexed
  by ordered set partitions $\pi\in\OSP_n$, $n\in\N$,
  which satisfy the following axioms.

\begin{enumerate}[label=\rm(S\arabic*)] 
 \item\label{S1}
\emph{Multilinearity}: Cumulants are multilinear functionals $K_\pi\colon \alg{A}^n\to\alg{B}$, indexed
  by ordered set partitions $\pi\in\OSP_n$,   $n\in\N$.
 \item\label{S2}
\emph{Universality}:
  There are universal coefficients $c(\sigma,\pi) \in \C$ such that
  for every $\pi\in\OSP_n$
  \begin{equation*}
\phi_\pi = K_\pi + \sum_{\substack{\sigma\in\OSP_n\\ \sigma<\pi}} c(\sigma,\pi) K_\sigma
  \end{equation*}
  or, equivalently,
  there are universal coefficients $\tilde{c}(\sigma,\pi) \in \C$ such that
    \begin{equation}
      \label{eq:SKpi=phipi+sum}
      K_\pi = \phi_\pi  + \sum_{\substack{\sigma\in\OSP_n\\ \sigma<\pi}} \tilde{c}(\sigma,\pi) \phi_\sigma. 
    \end{equation}

   \item\label{S3}
\emph{Extensivity}:
  given  a family $(X_1,X_2,\dots,X_n)\in\alg{A}^n$ and
  an ordered partition $\pi\in\OSP_n$,
  \begin{equation}
    \label{eq:extensivity}
    K_{\pi}(N.X_1,N.X_2,\ldots,N.X_n) =N^{\abs{\pi}}K_{\pi}(X_1,X_2,\ldots,X_n), 
  \end{equation}
  where
  the  extension of $K_\pi$ to arguments from $\bigcup \iota^{(i)}(\alg{A})$ is defined via formula
  \eqref{eq:phipi(NX)}
  and
  \eqref{eq:SKpi=phipi+sum}
  as
\begin{align*}
K_\pi(X_1^{(i_1)},X_2^{(i_2)},\ldots,X_n^{(i_n)})
&:=  \sum_{\substack{\sigma\in\OSP_n\\\sigma \leq \pi}}
    \tilde{c}(\sigma,\pi)\,\varphi_{\sigma}(X_1^{(i_1)},X_2^{(i_2)},\ldots,X_n^{(i_n)}) \\
&= \sum_{\substack{\sigma\in\OSP_n\\\sigma \leq \pi}}
    \tilde{c}(\sigma,\pi)\,\varphi_{\sigma\curlywedge \kappa(i_1,i_2,\ldots,i_n)}(X_1,X_2,\ldots,X_n)
\intertext{and thus  }
K_\pi(N_1.X_1,N_2.X_2,\ldots,N_n.X_n)&:=  \sum_{\substack{\sigma\in\OSP_n\\\sigma \leq \pi}}
     \tilde{c}(\sigma,\pi)\,\varphi_{\sigma}(N_1.X_1,N_2.X_2,\ldots,N_n.X_n),
\end{align*}
where the diagonal coefficients are $\tilde{c}(\pi,\pi)=1$.
\end{enumerate}
\end{defi}

\begin{theorem}\label{thm:unique}
  Cumulants are uniquely determined by axioms \ref{S1}--\ref{S3} from
  Definition~\ref{def:spreadableK2}.
\end{theorem}
\begin{proof}
  Assume that there are two sets of cumulants $K_\pi$ and $K'_\pi$ satisfying
  axioms \ref{S1}--\ref{S3} with universal coefficients $c(\sigma,\pi)$ and
  $c'(\sigma,\pi)$, respectively. Then the axioms imply that
      \begin{multline*}
        \varphi_{\pi}(N.X_1,N.X_2,\ldots,N.X_n)
        \\
      \begin{aligned}[t]
      &= N^{\abs{\pi}}K_{\pi}(X_1,X_2,\ldots,X_n) + \sum_{\sigma < \pi} N^{\abs{\sigma}}K_{\sigma}(X_1,X_2,\ldots,X_n)\,c(\sigma,\pi)  \\
      &= N^{\abs{\pi}}K'_{\pi}(X_1,X_2,\ldots,X_n) + \sum_{\sigma < \pi} N^{\abs{\sigma}}K'_{\sigma}(X_1,X_2,\ldots,X_n)\,c'(\sigma,\pi)
      \end{aligned}
      \end{multline*}
  for any $N\in\N$
  and the coefficients of the leading term $N^{\abs{\pi}}$ must coincide.
\end{proof}

\subsection{Construction of Cumulants via factorial M\"obius and zeta functions}

In this subsection, we assume that $(\alg{A},\phi)$ is a $\alg{B}$-ncps and $\cS=(\alg{U},\tilde{\phi},(\iota^{(i)})_{i\ge1})$ is a
  spreadability system for $(\alg{A},\varphi)$.

Spreadability implies that the value 
   \eqref{eq:phipixjij}
is invariant under order preserving changes of the indices;
however, for partitioned expectations invariance holds under the weaker
assumption that the relative order of indices is preserved on every individual block.
We will only be concerned with the following particular case and therefore
refrain from formulating this fact in full generality.
\begin{lemma}
  \label{lem:quasimeetshift}
  Let $\pi\in \OSP_n$ be an ordered set partition.
  Fix a block $P$ of $\pi$ and a number $m\in\N$
  and let $(i_1,i_2,\dots,i_n)$,  $(i'_1,i'_2,\dots,i'_n)$ be $n$-tuples
  such that
  $$
  i'_k=
  \begin{cases}
    i_k + m & \text{if $k\in P$}\\
    i_k     & \text{if $k\not\in P$}
  \end{cases}
  $$
  Then
  $\pi\curlywedge \kappa(i_1,i_2,\dots,i_n)
  = \pi\curlywedge \kappa(i'_1,i'_2,\dots,i'_n)$.
\end{lemma}

\begin{remark}
  Note that without performing the quasi-meet operation in the lemma above
  the kernel partitions
  $\kappa(i_1,i_2,\dots,i_n)$ 
  and $\kappa(i'_1,i'_2,\dots,i'_n)$ 
  may well be nontrivial permutations of each other or even 
  $\bar\kappa(i_1',i_2',\dots,i_n')\ne\bar\kappa(i_1,i_2,\dots,i_n)$.
\end{remark}

\begin{thm}  
  \label{thm:iidsum}  
  For any ordered set partitions $\rho \le \pi \in \OSP_n$ and numbers
  $N_1,N_2,\ldots\in\N$ we have
  \begin{equation}
    \label{eq:iidsum0}
  \varphi_{\rho}(N_{\pi(1)}.X_1,N_{\pi(2)}.X_2,\ldots,N_{\pi(n)}.X_n)
    = \sum_{\sigma \leq \rho}\varphi_{\sigma}(X_1,X_2,\ldots,X_n)\,\gamma_{\uN}(\sigma,\rho, \pi),
\end{equation}
  with $\gamma_{\uN}(\sigma, \rho, \pi)$ as in Definition~\ref{app:def:factorials} and $\pi(i)$ as in Definition \ref{defi:OP}\ref{item:OP2}. 
Specializing to the case $\rho=\pi$ yields
 \begin{align}
    \label{eq:iidsum}
  \varphi_{\pi}(N_{\pi(1)}.X_1,N_{\pi(2)}.X_2,\ldots,N_{\pi(n)}.X_n)
    = \sum_{\sigma \leq \pi}\varphi_{\sigma}(X_1,X_2,\ldots,X_n)\,\beta_{\uN}(\sigma,\pi)
\end{align}
and this is a polynomial in $N_1,\dots, N_{\abs{\pi}}$ without constant term.
Furthermore, specializing to the case $N:=N_1 = N_2 = \cdots$ entails 
  $$
  \varphi_{\pi}(N.X_1,N.X_2,\ldots,N.X_n)
    = \sum_{\sigma \leq \pi}\varphi_{\sigma}(X_1,X_2,\ldots,X_n)\,\beta_N(\sigma,\pi),
  $$
 which is a polynomial in $N$ consisting of monomials of degree at least $\abs{\pi}$. 
   \end{thm}

   \begin{proof}
     We begin with \eqref{eq:iidsum} and then verify the general case \eqref{eq:iidsum0}.  The proof boils down to the enumeration of the set
  \[
 S_{\uN}(\sigma,\pi):= \{(i_1,i_2,\ldots,i_n)\in [N_{\pi(1)}] \times [N_{\pi(2)}] \times \cdots \times [N_{\pi(n)}]: \pi \curlywedge \kappa(i_1,i_2,\ldots,i_n)=\sigma
  \}. 
  \]
  Pick an arbitrary block
  $P_i=\{p_1,p_2,\ldots,p_k\}$ of $\pi = (P_1,P_2,\dots)$.  There are $\binom{N_i}{\sharp(\sigma\restr_{P_i})}$ possible
  ways to choose a tuple $(i_{p_1},i_{p_2}, \ldots, i_{p_k})$ such that
  $\kappa(i_{p_1},i_{p_2}, \ldots, i_{p_k})$ 
  defines the partition $\sigma\restr_{P_i}$:
  for each block of $\sigma\restr_{P_i}$ we have to choose
  a distinct label from $[N_i]$ respecting the order prescribed by the
  labels of the blocks of $\sigma\restr_{P_i}$.
  That is, we have to choose a subset of $[N_i]$ of cardinality 
  $\#(\sigma\restr_{P_i})$.
  This can be done for every block of $\pi$ independently
  and thus
  \begin{equation}\label{eq cardinality} 
  \sharp S_{\uN}(\sigma,\pi)=\prod_{i =1}^{\abs{\pi}} \binom{N_i}{\sharp(\sigma\restr_{P_i})}=\beta_{\uN}(\sigma,\pi).
  \end{equation}
   If $N_i <\sharp(\sigma\restr_{P_i})$ for some $i$, then $\pi \curlywedge \kappa(i_1,i_2,\ldots,i_n)$ can never be equal to $\sigma$, and so 
    the cardinality of $ S_{\uN}(\sigma,\pi)$ is 0, in accordance
    with the generally adopted convention that the generalized
    binomial coefficient $\binom{N}{k}$ is zero when $N<k$. 
    
    The general formula \eqref{eq:iidsum0} requires the enumeration of the set
    \[
 S_{\uN}(\sigma,\rho,\pi) = \{(i_1,i_2,\ldots,i_n)\in [N_{\pi(1)}] \times [N_{\pi(2)}] \times \cdots \times [N_{\pi(n)}]: \rho \curlywedge \kappa(i_1,i_2,\ldots,i_n)=\sigma
  \}.   
  \]
This is similar to $S_{\uN}(\sigma,\rho)$; the difference is that the blocks of $\rho$ contained in a block $P_i \in \pi$ are endowed with the common number $N_i$. This results in the product formula
\[
\sharp S_{\uN}(\sigma,\rho,\pi) = \prod_{i=1}^p  S_{N_i}(\sigma \restr_{P_i},\rho \restr_{P_i}),  
\]
which equals $\gamma_{\uN}(\sigma,\rho,\pi)$ from Remark \ref{rem:beta_gamma}\ref{item:beta_gamma}.      
   
\end{proof}

\begin{defi}\label{def:cumulants} 
  Given an ordered set partition $\pi=(P_1,P_2,\dots,P_k) \in \OSP_n$
  with $k$ blocks, 
  we define the \emph{partitioned cumulant} $K_{\pi}(X_1,X_2,\ldots,X_n)$ to be the coefficient of
$N_1N_2\dotsm N_k$ in the polynomial
    \eqref{eq:iidsum}.
\end{defi}
The next theorem shows that this definition 
is a natural generalization of \cite[Definition~2.6]{L04}, cf.~Corollary \ref{cor:mc} below.

\begin{thm}[Moment-cumulant formulas]\label{thm:mc} For any $\pi \in \OSP_n$, we have 
  \begin{align}
  K_{\pi}(X_1,X_2,\ldots,X_n)  &= \sum_{\sigma \leq \pi}  \varphi_{\sigma}(X_1,X_2,\ldots,X_n)\,  \widetilde{\mu}(\sigma,\pi),    \label{eq:Kpi=sumphisi}  \\  
  \varphi_{\pi}(X_1, X_2,\ldots, X_n) &= \sum_{\sigma \leq \pi} K_{\sigma}(X_1,X_2,\ldots,X_n)\,\widetilde{\zeta}(\sigma,\pi)  \label{eq:phipi=sumKsi}    
  \end{align}
 with $ \widetilde{\mu}, \widetilde{\zeta}$ as in Definition~\ref{app:def:factorials}. 
\end{thm}
\begin{proof} For $0<k<N$, 
there is no constant term in $\binom{N}{k}= \frac{N(N-1)\cdots (N-k+1)}{k!}$ 
(regarded as a polynomial in $N$), and the coefficient of its linear term is 
$\frac{(-1)(-2)\cdots(-k+1)}{k!} = \frac{(-1)^{k-1}}{k}$. 
Therefore the coefficient of $N_1 \cdots N_{\abs{\pi}}$ in $\beta_{\uN}(\sigma,\pi)=\prod_{i=1}^{\abs{\pi}}\binom{N_i}{\sharp(\sigma\restr_{P_i})}$ is equal to $\frac{(-1)^{\abs{\pi}-\abs{\sigma}}}{[\sigma:\pi]}$. Note that $\sum_{P\in\pi}\sharp(\sigma\restr_P)=\abs{\sigma}$. 
Comparing with  \eqref{eq:mutilde} the claimed formula \eqref{eq:Kpi=sumphisi}  follows.
The same argument holds true if we look at the monomial $N^{\abs{\pi}}$ 
when $N_1 =N_2= \cdots =N$.

The inverse formula \eqref{eq:phipi=sumKsi} expressing moments in terms of
cumulants
is an immediate consequence of the fact that the modified M\"{o}bius function
$\widetilde{\mu}$
is the inverse of the modified zeta function $\widetilde{\zeta}$
(Corollary~\ref{app:cor:inversion}).

\end{proof}

\subsection{Verification of cumulant axioms}
In this subsection, we keep the assumption that $(\alg{A},\phi)$ is a
$\alg{B}$-ncps and $\cS=(\alg{U},\tilde{\phi},(\iota^{(i)})_{i\ge1})$ is a 
spreadability system for $(\alg{A},\varphi)$.

\begin{theorem}
  The cumulants defined in Definition~\ref{def:cumulants} verify the axioms \emph{\ref{S1}--\ref{S3}} from
  Definition~\ref{def:spreadableK2}.
\end{theorem}
Axiom \ref{S1} -- multilinearity and \ref{S2} -- universality are immediate
consequences of the relations     \eqref{eq:Kpi=sumphisi}
and \eqref{eq:phipi=sumKsi}. It remains to show axiom \ref{S3} -- extensivity, which is proven in Proposition~\ref{ad} in a slightly generalized form. 
For this, we iterate the construction \eqref{eq:phipi(NX)}.
\begin{defi}   For $\pi\in\OSP_n$ and $M_i, N_i \in \N, i=1,2,\dots,n$, we define
 \begin{multline*}
\phi_\pi(M_1.(N_1.X_1),M_2.(N_2.X_2), \ldots,M_n.(N_n.X_n)) \\
:= \sum_{(i_1,i_2,\ldots,i_n) \in[M_1] \times [M_2] \times \cdots \times [M_n]} \phi_{\pi\curlywedge \kappa(i_1,i_2,\ldots,i_n)}(N_1.X_1, N_2.X_2,\ldots,N_n.X_n). 
  \end{multline*}
\end{defi}

\begin{remark}
We would like to alert the reader that the formal definition of the expectation
 \[
 \phi_\pi(M_1.(N_1.X_1),M_2.(N_2.X_2),\ldots,M_n.(N_n.X_n))
 \] 
is not necessarily related
to ``$M.(N.X)$'' (which formally would represent the element
$
(N.X)^{(1)}+(N.X)^{(2)}+\cdots +(N.X)^{(M)}
,
$
i.e., 
``the sum of i.i.d.\ copies of $N.X$'')
as an element of an enlarged space,
obtained, e.g., by iterating a product construction.
Yet we will show below that the vector
$(M_{\pi(1)}N_{\pi(1)}.X_1,\ldots, M_{\pi(n)}N_{\pi(n)}.X_n)$ and
the virtual vector ``$(M_{\pi(1)}.(N_{\pi(1)}.X_1),
\ldots,M_{\pi(n)}.(N_{\pi(n)}.X_n))$'' have the same distribution with respect to $\varphi_\pi$.
This property implies extensivity of cumulants, see Proposition~\ref{ad} below. 
It is a consequence of the associativity of the corresponding universal product in the case of classical, free, monotone and Boolean independence \cite{HS11b}.
For general spreadability systems it holds on a formal level
(Lemma~\ref{prop5}), even when it comes from a nonassociative
universal product, like the example of V-monotone independence from 
Section~\ref{sssec:vmonotone}.
\end{remark}

\begin{lem}\label{prop5} Let $\pi \in \OSP_n$, $X_i \in \mathcal{A}$ and $M_i,N_i \in \N, i=1,2,\dots, \abs{\pi}$. 
\begin{enumerate}[label=\rm(\roman*),leftmargin=1cm]

\item\label{Polynomial} The value
 $\phi_\pi(M_{\pi(1)}.(N_{\pi(1)}.X_1),\ldots,M_{\pi(n)}.(N_{\pi(n)}.X_n))$ is
 a polynomial in the variables $M_i,N_i, i=1,2,\dots, \abs{\pi}$,  and 
$K_\pi(N_{\pi(1)}.X_1,\ldots,N_{\pi(n)}.X_n)$ is the coefficient of $M_1 \cdots M_{\abs{\pi}}$. 

\item \label{Associative-like} The dot operation gives rise to an action of the
 multiplicative semigroup $\N^\infty$
 \begin{multline*}
   \phi_\pi(M_{\pi(1)}N_{\pi(1)}.X_1,\ldots,M_{\pi(n)}N_{\pi(n)}.X_n)
   \\
   = \phi_\pi(M_{\pi(1)}.(N_{\pi(1)}.X_1),\ldots,M_{\pi(n)}.(N_{\pi(n)}.X_n)). 
 \end{multline*}

\end{enumerate}
\end{lem}
\begin{proof}
\ref{Polynomial}\,\, Let $\uM:=(M_1,\dots, M_{\abs{\pi}},0,0,\dots)\in \N^\infty$ and similarly $\uN \in \N^\infty$. 
  Then the following expansion holds:
  \begin{multline}\label{eq20}
     \phi_\pi(M_{\pi(1)}.(N_{\pi(1)}.X_1),\ldots,M_{\pi(n)}.(N_{\pi(n)}.X_n))\\
    \begin{aligned}
      &= \sum_{(i_1,\dots, i_n) \in [M_{\pi(1)}]\times \cdots \times [M_{\pi(n)}]} \phi_{\pi \curlywedge \kappa(i_1,i_2,\ldots,i_n)}(N_{\pi(1)}.X_1,\ldots,N_{\pi(n)}.X_n) \\ 
      &=\sum_{\sigma \leq \pi} \sum_{\substack{(i_1,\dots, i_n) \in [M_{\pi(1)}]\times \cdots \times [M_{\pi(n)}],\\ \pi \curlywedge \kappa(i_1,i_2,\ldots,i_n)=\sigma}} \phi_{\sigma}(N_{\pi(1)}.X_1,\ldots,N_{\pi(n)}.X_n) \\
      &=  \sum_{\sigma \leq \pi} \phi_{\sigma}(N_{\pi(1)}.X_1,\ldots,N_{\pi(n)}.X_n) \,\beta_{\uM}(\sigma,\pi), 
    \end{aligned}
    \end{multline}
 where we used the identity \eqref{eq cardinality}. Hence $\phi_\pi(M_{\pi(1)}.(N_{\pi(1)}.X_1),\ldots,M_{\pi(n)}.(N_{\pi(n)}.X_n))$ is a polynomial in $M_i,N_i, i=1,2,\dots, \abs{\pi}$, 
  and from the proof of Theorem~\ref{thm:mc} 
we infer that the coefficient of $M_1\cdots M_{\abs{\pi}}$ is equal to $K_\pi(N_{\pi(1)}.X_1,\ldots,N_{\pi(n)}.X_n)$. 

\itemspacing
\ref{Associative-like}\,\, We proceed with the computation of \eqref{eq20}:
\begin{multline*}
     \phi_\pi(M_{\pi(1)}.(N_{\pi(1)}.X_1),M_{\pi(2)}.(N_{\pi(2)}.X_2),\ldots,M_{\pi(n)}.(N_{\pi(n)}.X_n))\\
    \begin{aligned}
         &= \sum_{\rho \leq \pi} 
             \phi_{\rho}(N_{\pi(1)}.X_1,N_{\pi(2)}.X_2,\ldots,N_{\pi(n)}.X_n)
             \,\beta_{\uM}(\rho,\pi) \\
          &= \sum_{\rho \leq \pi} 
             \sum_{\sigma \le \rho }
              \phi_{\sigma}(X_1,X_2, \ldots,X_n) 
              \, \gamma_{\uN}(\sigma,\rho,\pi)
              \, \beta_{\uM}(\rho,\pi)\\ 
          &= \sum_{\sigma\leq \pi}
             \phi_{\sigma}(X_1, X_2,\ldots,X_n)
             \biggl(  \sum_{\substack{\rho\in\OSP_n  \\ \sigma \leq \rho \leq \pi}}
               \gamma_{\uN}(\sigma,\rho,\pi)
               \, \beta_{\uM}(\rho,\pi) 
             \biggr)\\ 
          &= \sum_{\sigma \leq \pi} \phi_\sigma(X_1, X_2,\ldots,X_n) (\gamma_{\uN} \otriangle \beta_{\uM}) (\sigma,\pi)\\
          &= \sum_{ \sigma \leq \pi} \phi_\sigma(X_1, X_2,\ldots,X_n) \,\beta_{\uM \circ \uN}(\sigma,\pi) \\
          &= \phi_{\pi}(M_{\pi(1)}N_{\pi(1)}.X_1,M_{\pi(2)}N_{\pi(2)}.X_2,\ldots,M_{\pi(n)}N_{\pi(n)}.X_n), 
    \end{aligned}
\end{multline*}
    where  \eqref{eq:iidsum0} was used in the second equality and Corollary~\ref{app:cor:inversion} was used in the next to last line.
\end{proof}

\begin{prop}\label{ad}
  The cumulants from 
 Definition~\ref{def:cumulants}  satisfy extensivity \eqref{eq:extensivity} and more generally, 
  \begin{equation*}
  K_{\pi}(N_{\pi(1)}.X_1,N_{\pi(2)}.X_2,\ldots,N_{\pi(n)}.X_n) =N_1N_2 \cdots N_{\abs{\pi}}K_{\pi}(X_1,X_2,\ldots,X_n). 
  \end{equation*}
\end{prop}
\begin{proof}
  By definition, $\phi_\pi$ applied to the same arguments has the expansion
  \begin{multline*}
  \phi_{\pi}(M_{\pi(1)}N_{\pi(1)}.X_1, M_{\pi(2)}N_{\pi(2)}.X_2,\ldots,M_{\pi(n)}N_{\pi(n)}.X_n)
\\
  =\Biggl(
      \prod_{i=1}^{\abs{\pi}}M_i N_i
      \Biggr)
      K_{\pi}(X_1,X_2,\ldots,X_n) + (\text{sum of monomials in $M_i N_i$ of higher degrees}). 
  \end{multline*}
    On the other hand, from Lemma~\ref{prop5}\ref{Polynomial} we infer
    \begin{multline*}
      \phi_{\pi}(M_{\pi(1)}.(N_{\pi(1)}.X_1),
      M_{\pi(2)}.(N_{\pi(2)}.X_2),\ldots,M_{\pi(n)}.(N_{\pi(n)}.X_n))
      \\
      = M_1 M_2 \cdots M_{\abs{\pi}} K_{\pi}(N.X_1,N.X_2,\ldots,N.X_n)
      \\
+ (\text{sum of monomials in $M_i$ of higher degree}). 
    \end{multline*}
We have thus computed the expectation in two ways and
it follows from Lemma~\ref{prop5}\ref{Associative-like} that the coefficients of $M_1 \cdots M_{\abs{\pi}}$  coincide. 
\end{proof}

\subsection{Examples}
Let us now briefly review instances of moment-cumulant formulas 
coming from noncommutative notions of independence in the light
of Theorem~\ref{thm:mc}. Let $(\alg{A},\phi)$ be a $\alg{B}$-ncps and $\cS=(\alg{U},\tilde{\phi},(\iota^{(i)})_{i\ge1})$ be a
  spreadability system for $(\alg{A},\varphi)$.

  \subsubsection{Cumulants for exchangeability systems}

Let start with the verification that the definition of cumulants for spreadability
systems is consistent with the previous definition for exchangeability systems
from \cite{L04}.  

\begin{prop}\label{inv:permutation}
  If a spreadability system $\cS$ satisfies exchangeability, 
  then both the partitioned expectations and cumulants are invariant under permutations of the
  blocks.
  More precisely,
  for any partition $\pi=(P_1,P_2, \dots,P_k) \in \OSP_n$ and permutation $h \in \SG_k$,
  $\varphi_\pi = \varphi_{h(\pi)}$ and
  $K_\pi = K_{h(\pi)}$,   where $h(\pi)=(P_{h(1)}, P_{h(2)}, \dots,P_{h(k)})$,
  Thus $\varphi_\pi$ and $K_\pi$ are completely determined by the underlying (unordered) set partition $\bar{\pi}$.
\end{prop}
\begin{proof}
   Invariance of $\varphi_\pi$ under permutations of the blocks of $\pi$ is an immediate consequence of
   exchangeability and this invariance extends to the extension  \eqref{eq:phipixjij}.
   Indeed let $\pi=(P_1,P_2,\dots,P_k)\in\OSP_n$ and  $i_1,i_2,\dots,i_n \in \N$,
   then for any $h \in \SG_k$ there exists a permutation $g \in \SG_\infty$
   such that
   $h(\pi)\curlywedge \kappa(i_1,i_2,\dots,i_n) = g(\pi\curlywedge \kappa(i_1,i_2,\dots,i_n) )$. Therefore, 
\begin{align*}
\varphi_{h(\pi)}(X_1^{(i_1)}, X_2^{(i_2)},\dots, X_n^{(i_n)}) &= \varphi_{h(\pi)\curlywedge \kappa(i_1,i_2,\dots,i_n)}(X_1,X_2,\dots, X_n) \\
&= \varphi_{g(\pi\curlywedge \kappa(i_1,i_2,\dots,i_n))}(X_1,X_2,\dots, X_n)  \\
&= \varphi_{\pi\curlywedge \kappa(i_1,i_2,\dots,i_n)}(X_1,X_2,\dots,X_n) \\
&=\varphi_{\pi}(X_1^{(i_1)}, X_2^{(i_2)},\dots, X_n^{(i_n)}). 
\end{align*}
Consequently also $\varphi_\pi(N.X_1,N.X_2,\dots,N.X_n) =
\varphi_{h(\pi)}(N.X_1,N.X_2,\dots,N.X_n)$ and,
by definition,  $K_\pi(X_1,X_2,\dots, X_n) = K_{h(\pi)}(X_1,X_2,\dots, X_n)$. 
\end{proof}

From the above observation and with the help of
Proposition~\ref{app:prop:f*tzeta=fbar*zeta}
we can recover the moment-cumulant formulas in the exchangeable case
\cite[Definition 2.6 and Proposition~2.7]{L04}.

\begin{corollary}\label{cor:mc}
  If a spreadability system $\cS$ satisfies exchangeability
  then  for $\pi \in \SP_n$  we simply write $K_\pi$ for the partitioned
  cumulant functional uniquely determined in Proposition~\ref{inv:permutation}.
  Then for any $\pi \in \SP_n$, we have 
  \begin{align}
  K_{\pi}(X_1,X_2,\ldots,X_n)  &= \sum_{\substack{\sigma \in \SP_n \\ \sigma \leq \pi}}  \varphi_{\sigma}(X_1,X_2,\ldots,X_n)\,  \mu_{\SP}(\sigma,\pi),    \label{eq:Kpi=sumphisi2}  \\  
  \varphi_{\pi}(X_1, X_2,\ldots, X_n) &= \sum_{\substack{\sigma \in \SP_n \\ \sigma \leq \pi}} K_{\sigma}(X_1,X_2,\ldots,X_n), \label{eq:phipi=sumKsi2}    
  \end{align}
where $\mu_{\SP}$ is the M\"obius function on the posets $\SP_n, n\in \N$, see \eqref{app:eq:SPmoebius}. 
\end{corollary}

\subsubsection{Some moment-cumulant formulas of low order}

\begin{exa}[Cumulants in terms of moments] 
  We will write the ordered kernel set partition $\kappa(i_1,\dots, i_n)$
  simply as the multiset permutation $i_1i_2\dots i_n$. For example,
  $211=(\{2,3\},\{1\})$.  
  Examples of Theorem~\ref{thm:mc} are given by 
\begin{align*}
K_1(X) &= \phi_1(X), \\
K_{11}(X, Y) &= \phi_{11}(X,Y) - \frac{1}{2}(\phi_{12}(X,Y) + \phi_{21}(X,Y)), \\
K_{12}(X,Y) &= \phi_{12}(X,Y), \\
K_{21}(X,Y) &= \phi_{21}(X,Y), \\
K_{111} &= \phi_{111} - \frac{1}{2}(\phi_{112}+\phi_{121}+\phi_{122}+\phi_{211}+\phi_{212}+\phi_{221}) \\ \notag
&\qquad+\frac{1}{3}(\phi_{123}+\phi_{132}+\phi_{213}+\phi_{231}+\phi_{312}+\phi_{321}), \\
K_{112} &= \phi_{112} - \frac{1}{2}(\phi_{123}+\phi_{213}), 
\end{align*}
where for the sake of compactness the arguments $(X,Y,Z)$ are omitted 
in the last two formulas.
\end{exa}

\begin{exa}[Moments in terms of cumulants] Examples of Theorem~\ref{thm:mc} are given by 
\begin{align*}
\phi_1(X) &= K_1(X), \\
\phi_{11}(X, Y) &= K_{11}(X,Y) + \frac{1}{2!}(K_{12}(X,Y) + K_{21}(X,Y)), \\
\phi_{12}(X,Y) &= K_{12}(X,Y), \\
\phi_{21}(X,Y) &= K_{21}(X,Y), \\
\phi_{111} &= K_{111} + \frac{1}{2!}(K_{112}+K_{121}+K_{122}+K_{211}+K_{212}+K_{221}) \\ \notag
&\qquad+\frac{1}{3!}(K_{123}+K_{132}+K_{213}+K_{231}+K_{312}+K_{321}), \\
\phi_{112} &= K_{112} + \frac{1}{2!}(K_{123}+K_{213}), 
\end{align*}
where $(X,Y,Z)$ are omitted in the last two formulas for simplicity. 
\end{exa}

\subsubsection{Cumulants for universal products}
We verify that the cumulants for associative universal products satisfy factorization properties. Here we further assume that $\mathcal{B} =\C$, but it will not be difficult for readers familiar with operator-valued
independence \cite{Speicher:1998:combinatorial,Popa2008,HS14,Ske04,Mlo02} to
generalize the results to general $\alg{B}$-valued conditional expectations in the cases of free, Boolean
and monotone products and spreadability systems.

\begin{defi}\label{def:multiplicative} 
  For $n\in\N$ the  $n$-linear maps on $\alg{A}$
$$
K_n:=K_{\hat{1}_n}, \qquad \varphi_n:=\varphi_{\hat{1}_n}
$$
are called the  $n^{\mathrm{th}}$ cumulant and expectation functional,
respectively.
The latter is simply
$$
\varphi_n(X_1, X_2, \dots, X_n)
=\varphi(X_1X_2  \cdots X_n)
$$
for $X_i \in \alg{A}$.
Following notation \eqref{eq:phipixjij},
we also write
$$
\varphi_n(X_1^{(i_1)}, X_2^{(i_2)}, \dots, X_n^{(i_n)})
=\tilde{\varphi}(X_1^{(i_1)} X_2^{(i_2)} \cdots X_n^{(i_n)})
.
$$
We extend these functionals
multiplicatively to set partitions and 
ordered set partitions as follows.

  For set partitions $\pi\in\SP_n$  we define multiplicative extensions
  $\varphi_{(\pi)}$ of $(\varphi_n)_{n\in\N}$ and $K_{(\pi)}$ of $(K_n)_{n\in\N}$ 
  by setting
  \begin{align*}
    &\varphi_{(\pi)}(Y_1,Y_2,\ldots, Y_n) :=\prod_{P\in \pi} \varphi_{\abs{P}}(Y_P)   \qquad \text{and} \qquad  K_{(\pi)}(Y_1,Y_2,\ldots, Y_n) :=\prod_{P\in \pi} K_{\abs{P}}(Y_P)\\
   &\qquad\qquad \text{for} \quad Y_1,Y_2,\dots, Y_n\in \bigcup_{i=1}^\infty \iota^{(i)}(\mathcal{A}) \quad \text{or} \quad Y_1, Y_2,\dots, Y_n \in \alg A,
  \end{align*}
  where we use notation \eqref{notation:linear}  for multilinear functionals. 
We also define $\varphi_{(\pi)}:=\varphi_{(\bar{\pi})}$ and $K_{(\pi)}:=K_{(\bar{\pi})}$ for ordered set partitions $\pi\in\OSP_n$. 
\end{defi}

These multiplicative extensions are important for the understanding of
cumulants arising from universal products.
Note that the above definition of $\varphi_{(\pi)}$ reduces to
\eqref{eq:multiplicative} when $Y_i \in \alg A$ and the use of the same
notation does not cause any conflict.  

It turns out that for specific spreadability systems certain cumulant
functionals vanish identically. This is a consequence of a certain factorization
property which is satisfied by these spreadability systems and which is
subsumed in the following lemma.

\begin{lemma}\label{lem:factorization}
 Let $\pi\in\OSP_n$ be an ordered set partition.  
  \begin{enumerate}[label=\rm(\roman*),leftmargin=1cm]
\item\label{item:factorization1}  Let $\cS=(\alg{U}, \tilde \phi, (\iota^{(i)})_{i\ge1})$ be a spreadability system for some $\alg{B}$-ncps $(\mathcal{A},\phi)$. Assume that there exists a family of complex numbers 
$
\{s(\sigma; \pi): \sigma \in \SP_n, \sigma \le \bar\pi\}
$ 
such that 
  \begin{equation}
    \label{eq:lem:factorization1} 
\varphi_\pi(Y_1, Y_2, \dots, Y_n) =   \sum_{\substack{\sigma \in \SP_n\\ \sigma \le \bar\pi }} s(\sigma; \pi) \varphi_{(\sigma)}(Y_1,Y_2,\dots, Y_n)
  \end{equation}
    holds for any tuple of random variables $Y_1, Y_2, \dots, Y_n \in  \bigcup_{i=1}^\infty \iota^{(i)}(\mathcal{A})$. 
    Then 
    \begin{equation*}
    K_\pi(X_1,X_2,\dots,X_n)=  s(\bar\pi;\pi) K_{(\pi)}(X_1,X_2,\dots,X_n)
    \end{equation*}
    for any tuple $X_1, X_2, \dots, X_n \in \alg{A}$.   \item\label{item:factorization2}
    Let $(\alg{U}, \tilde \phi, (\iota^{(i)})_{i\ge1})$ be a spreadability
    system for some $\alg{B}$-ncps $(\mathcal{A},\phi)$.
    Assume that there exist coefficients
    $
    \{t(\rho; \pi): \rho \in \OSP_n, \bar \sigma < \bar \pi\}
    $ 
    such that 
   \begin{equation*}
\varphi_\pi(Y_1, Y_2, \dots, Y_n) =  \sum_{\substack{\rho \in \OSP_n\\ \bar \rho < \bar \pi }} t(\rho; \pi) \varphi_{\rho}(Y_1,Y_2,\dots, Y_n)
  \end{equation*}
   holds for any tuple of random variables $Y_1, Y_2, \dots, Y_n \in
   \bigcup_{i=1}^\infty \iota^{(i)}(\mathcal{A})$.
   Then the cumulant $K_\pi(X_1,X_2,\dots, X_n)=0$  for any tuple $X_1, X_2, \dots, X_n \in \alg{A}$.  
\end{enumerate}
\end{lemma}
\begin{remark}
  Suppose the expansion \eqref{eq:lem:factorization1}
  holds for all ordered set partitions $\pi \in \OSP_n$, then $\cS$ is a
  spreadability system with calculation rule \eqref{eq:calculation_rule} (e.g.\
  take $Y_1,Y_2,\dots, Y_n \in \iota^{(1)}(\alg A)$).
  Now the expansion \eqref{eq:lem:factorization1} extends the calculation rule to the variables
  $Y_1,Y_2,\dots$ in $\bigcup_{i=1}^\infty \iota^{(i)}(\mathcal{A})$.
  Such an extension is not automatic; we give a counterexample in Example \ref{exa:nonassociative}. 

\end{remark}

\begin{proof}[Proof of Lemma \ref{lem:factorization}]
 \ref{item:factorization1}  Choosing $Y_k = \iota^{(i_k)}(X_k)$ and taking the sum of  \eqref{eq:lem:factorization1} over all indices $i_1, i_2,\dots, i_n \in [N]$ we get formula \eqref{eq:lem:factorization1} for $Y_k = N. X_k$ as well. The desired conclusion is  an immediate consequence by comparing the coefficients of $N^{|\pi|}$ together with the definition of cumulants (Definition~\ref{def:cumulants}). 
The point is that, since the number of blocks of any $\sigma<\bar \pi$ is strictly larger than the number of blocks of $\pi$, the coefficient of $N^{|\pi|}$ vanishes in $\phi_{(\sigma)}(N.X_1,N.X_2,\dots,N.X_n)$.    \ref{item:factorization2} is proven in a similar way. 
\end{proof}

\begin{prop}\label{prop19}   \
  \begin{enumerate}[label=\rm(\roman*),leftmargin=1cm]
\item\label{ten}
    Let $K^{\tensorT}_\pi$ be the cumulants associated to the tensor spreadability system $\cS_{\tensorT}$. Then 
    \begin{equation*}
    K^{\tensorT}_\pi=K^{\tensorT}_{(\pi)},\qquad \pi \in \OSP_n. 
    \end{equation*}

   \item\label{fre}
    Let $K^{\freeF}_\pi$ be the cumulants associated to the free spreadability system $\cS_{\freeF}$. Then  
    \begin{equation*}
    K^{\freeF}_\pi=
    \begin{cases}
      0, &\pi \notin \ONP_n, \\
      K^{\freeF}_{(\pi)},& \pi \in \ONP_n
    \end{cases}
    \end{equation*}
    where by $\ONP_n$ we denote the set of ordered noncrossing partitions,
    see Definition~\ref{app:def:OSPclass}     \ref{app:def:ONC}.

   \item\label{boole} 
    Let $K^{\tensorT}_\pi$ be the cumulants associated to the Boolean spreadability system $\cS_{\boolB}$. Then 
    \begin{equation*}
    K^{\boolB}_\pi=
    \begin{cases}
      0, &\pi \notin \OIP_n, \\
      K^{\boolB}_{(\pi)},& \pi \in \OIP_n
    \end{cases}
  \end{equation*}
    where by $\OIP_n$ we denote the set of ordered interval partitions,    
    see Definition~\ref{app:def:OSPclass}     \ref{app:def:OI}.
   \item\label{it:mon}
   Let $K^{\monoM}_\pi$ be the cumulants associated to the monotone spreadability system $\cS_{\monoM}$. Then 
    \begin{equation*}
    K^{\monoM}_\pi=
    \begin{cases}
      0, &\pi \notin \MP_n, \\
      K^{\monoM}_{(\pi)},& \pi \in \MP_n
    \end{cases}
    \end{equation*}
    where by $\MP_n$ we denote the set of monotone partitions,
    see Definition~\ref{app:def:OSPclass}     \ref{app:def:monopart}.
\end{enumerate}
These results combined with the general moment-cumulant formula (Theorem \ref{thm:mc}) reproduce the known formulas. 
\end{prop}

\begin{proof}
  Items \ref{ten}, \ref{fre} and \ref{boole} satisfy exchangeability
  and are covered  in~\cite{L04} (note that our cumulants are invariant under the permutation of the blocks of $\pi$, see Proposition~\ref{inv:permutation});
  alternatively these cases follow from Lemma~\ref{lem:factorization} and the arguments below.

  We are left with case \ref{it:mon}.
  Let $u_{\rm M}(\sigma; \pi)$ be the universal coefficients from Axiom
  \ref{U6} for monotone independence
  (where $f \colon \sigma \to [1]$ is unique and hence omitted).  
  We can prove \eqref{eq:lem:factorization1} for $s(\sigma; \pi) :=  u_{\mathrm{M}}(\sigma; \pi)$
  using the technique from the proof of
  Proposition~\ref{prop:equivalence_independences}.
  Indeed,
  associativity allows us to deduce the identity
  $$
  \varphi_{\pi}(X_1^{(i_1)},X_2^{(i_2)},\dots, X_n^{(i_n)})= \varphi_{\pi
    \curlywedge \kappa(i_1,\dots, i_n)}(X_1,X_2,\dots, X_n)
  $$
  following the two steps \ref{stepa} and \ref{stepb} there
  and the resulting value is exactly  the RHS of
  \eqref{eq:lem:factorization1} with $Y_k = X_k^{(i_k)}$. 
  It is known that $u_{\rm M}(\bar{\pi}; \pi)=1$ for $\pi \in \MP_n$ and
  $u_{\rm M}(\bar{\pi}; \pi)=0$ for $\pi \in \OSP_n \setminus \MP_n$, see
  \cite[Prop.~3.2]{Muraki:2002:five} and the claimed formula follows.     
\end{proof}

\begin{remark}
  Other examples of noncrossing cumulants come from the c-free
  exchangeability systems of \cite{BLS96}, see \cite[Section~4.7]{L04}.
  Lemma~\ref{lem:factorization} \ref{item:factorization2} provides a new
  proof that crossing cumulants vanish.
\end{remark}   
Let us consider next c-monotone spreadability system $\cS_{\mathrm{CM}}$.
Recall from Section~\ref{sssec:cmonotone} that the construction
of  $\cS_{\mathrm{CM}}$ is based on an algebra $\mathcal{A}$ with two linear
functionals $\phi,\psi$.
\begin{prop}\label{prop20} 
 Let $K^{\mathrm{CM}}_\pi$ be the $\pi$-cumulant associated to the c-monotone spreadability
 system $\cS_{\mathrm{CM}}$ defined in Section \ref{sssec:cmonotone} and
 $K^{{\monoM},\psi}_n$ the $n^{\mathrm{th}}$ monotone cumulant functional with
 respect to the linear map $\psi:\mathcal{A}\to \C$.
 Then
  \[
  \displaystyle 
  K^{\mathrm{CM}}_\pi(X_1,X_2,\ldots,X_n)=
  \begin{cases}
    \displaystyle \prod_{P \in \Outer(\pi)}K^{\mathrm{CM}}_{\abs{P}}(X_P) \prod_{P \in \Inner(\pi)}K^{{\monoM},\psi}_{\abs{P}}(X_P),& \pi \in \MP_n, 
    \\
    0, &\pi \notin \MP_n 
  \end{cases}
  \]
 with $\Outer(\pi)$ and $\Inner(\pi)$ the sets of outer blocks of $\pi$ and inner blocks of $\pi$, respectively; see Definition \ref{app:def:partitions} \ref{item:app:def:partitions2}. 
\end{prop}
\begin{proof}
We use two spreadability systems: the c-monotone spreadability system $\cS_{\mathrm{CM}}=(\alg{U}, \tilde{\varphi}, (\iota^{(i)})_{i=1}^\infty)$ for $(\mathcal{A},\varphi)$ and the monotone spreadability system $\cS_{\monoM}=(\alg{U}, \tilde{\psi}, (\iota^{(i)})_{i=1}^\infty)$ for $(\mathcal{A},\psi)$. 
  The multiplicative extension $\phi_{(\pi)}$ of $\phi$ is now modified to
  \begin{align*}
 & \varphi_{(\pi), \psi}(Y_1, Y_2,\ldots, Y_n) 
  := \prod_{P \in \Outer(\pi)}\varphi_{\abs{P}}(Y_P) \prod_{P \in \Inner(\pi)}\psi_{\abs{P}}(Y_P) \\
    & \qquad \text{for} \quad \pi\in\ONP_n,  \quad Y_1,Y_2,\dots, Y_n\in \bigcup_{i=1}^\infty \iota^{(i)}(\mathcal{A}) \quad \text{or} \quad Y_1, Y_2,\dots, Y_n \in \alg A,
  \end{align*}
  which apparently coincides with $\varphi_{(\pi)}$ if $\varphi=\psi$. 
  By the definition of c-monotone spreadability system, there exist universal constants
  $u_{\mathrm{CM}}(\sigma, f; \pi) \in \C$ depending only on  $\pi \in \OSP_n, \sigma \in\SP_n$
  with $\bar\pi > \sigma$ and a ``2-coloring'' $f$ of the blocks of $\sigma$,
  i.e.\ a function $f\colon \sigma\to \{1,2\}$, such that   for any $Y_j\in \alg A$
\begin{multline*}
    \varphi_\pi(Y_1,Y_2,\ldots,Y_n)
    =
      \varphi_{(\pi), \psi}(Y_1,Y_2,\ldots,Y_n)
    \\
      +
\sum_{\substack{\sigma\in\SP_n,  \sigma <\bar{\pi} \\  f\colon\,
          \sigma\to \{1,2\} }}u_{\mathrm{CM}}(\sigma, f; \pi)
      \Bigl(\prod_{f(S)=1}\varphi_{\abs{S}}(Y_S)\Bigr)
      \Bigl(\prod_{f(S)=2}\psi_{\abs{S}}(Y_S)\Bigr)
\end{multline*}
if $\pi \in \MP_n$ and
$$
\varphi_\pi(Y_1,Y_2,\ldots,Y_n)  =
\sum_{\substack{\sigma\in\SP_n,  \sigma <\bar{\pi} \\  f\colon\, \sigma\to     \{1,2\} }}
    u_{\mathrm{CM}}(\sigma, f; \pi)
    \Bigl(\prod_{f(S)=1}\varphi_{\abs{S}}(Y_S)\Bigr)
    \Bigl(\prod_{f(S)=2}\psi_{\abs{S}}(Y_S)\Bigr)
$$ if $\pi \notin \MP_n$.
 Similar to the monotone case, thanks to associativity and
the arguments from the proof of Proposition~\ref{prop:equivalence_independences}, 
this formula holds also for $Y_j \in  \bigcup_{i=1}^\infty
\iota^{(i)}(\mathcal{A}), j\in[n]$. Then the arguments from Lemma~\ref{lem:factorization} and Proposition \ref{prop19} applies.
\end{proof}

    Associativity played a key role in the proof of Propositions
  \ref{prop19} and \ref{prop20} above. The following example shows that indeed
  the assumptions of Lemma~\ref{lem:factorization} can fail for nonassociative universal products. 

\begin{exa} \label{exa:nonassociative}
Let $(\alg{A}, \varphi)$ be a $\C$-ncps and let $\alg{U}:= \sqcup_{i=1}^\infty \alg{A}$ be the coproduct in the category of algebras (see \eqref{eq:coproduct}) and $\iota^{(i)} \colon \alg{A} \to \alg{U}$ be the natural embedding as $i^{\mathrm{th}}$ component. We define a linear functional $\tilde{\varphi}$ on $\alg{U}$ by, for each $X_1,X_2,\dots, X_n \in \alg{A}$ and $i_1,i_2,\dots, i_n\in\N$ with $i_j \ne i_{j+1}$ for all $j =1,2,\dots, n-1$, setting 
\[
\tilde{\varphi}(X_1^{(i_1)} X_2^{(i_2)} \cdots X_n^{(i_n)}) 
= 
\begin{cases} 
0& \text{~if there are $k \ne \ell$ such that $i_k=i_\ell$}, \\
\prod_{i=1}^n \varphi(X_i) &  \text{~otherwise}.  
\end{cases}
\]
Then we get an exchangeability system $(\alg{U}, \tilde{\varphi}, (\iota^{(i)})_{i \in \N})$ for $(\alg{A}, \varphi)$.

For  $\pi \in \SP_n$ and $X_1,X_2,\dots, X_n \in \alg{A}$ it is easy to see that  
 \begin{equation}
      \phi_\pi(X_1,X_2,\dots,X_n) =   
      \begin{cases} \phi_{(\pi)}(X_1,X_2,\dots,X_n) & \text{~if~} \pi \in \IP_n, \\
      0 & \text{~otherwise}
      \end{cases}
    \end{equation}
    holds for any tuple $X_i \in \alg{A}$; however this formula does not extend
    to mixed tuples $Y_i \in \bigcup_{i=1}^\infty \iota^{(i)}(\alg{A})$. For example, if $\pi=\{\{1,3\},\{2\}\}$ then 
    \[
\varphi_\pi(X_1^{(1)}, X_2^{(2)}, X_3^{(3)}) = \varphi(X_1) \varphi(X_2) \varphi(X_3) = \varphi(X_1^{(1)}) \varphi(X_2^{(2)}) \varphi(X_3^{(3)}) 
\] 
but \[
\varphi_\pi(X_1^{(1)}, X_2^{(1)}, X_3^{(1)}) =0 \ne \varphi(X_1^{(1)}) \varphi(X_2^{(1)}) \varphi(X_3^{(1)}) 
\] 
in general.

Next we shall compute some cumulants. Straightforward calculations yield
\begin{align*}
\varphi_{\hat{1}_1} (N. X_1) &= N \varphi(X_1), \\
\varphi_{\hat{1}_2}(N.X_1,N.X_2) &= N \varphi(X_1X_2) + N(N-1) \varphi(X_1) \varphi(X_2), \\
  \varphi_{\hat{1}_3}(N.X_1, N.X_2, N.X_3)
    &= \begin{multlined}[t]
      N \varphi(X_1 X_2 X_3) + N(N-1) \varphi(X_1 X_2) \varphi(X_3)
      \\
      +  N(N-1) \varphi(X_1) \varphi(X_2 X_3)
      \\
      + N(N-1) (N-2) \varphi(X_1)\varphi( X_2) \varphi(X_3),  
    \end{multlined}
\end{align*}
so that the coefficients of $N$ yield the cumulants  
\begin{align*}
K_1 (X_1) &= \varphi(X_1), \\
K_2(X_1,X_2) &=  \varphi(X_1X_2) - \varphi(X_1) \varphi(X_2), \\
K_3(X_1, X_2, X_3) &= \varphi(X_1 X_2 X_3) - \varphi(X_1 X_2) \varphi(X_3) - \varphi(X_1) \varphi(X_2 X_3) + 2 \varphi(X_1)\varphi( X_2) \varphi(X_3).   
\end{align*}
On the other hand, for $\pi = \{\{1,3\},\{2\}\}$ we have 
\begin{align*}
\varphi_{\pi}(N.X_1, N.X_2, N.X_3) 
  &= \sum_{\substack{ i_1,i_2,i_3 \\ i_1 = i_3}}     \sum_{i_1 = i_3}
      \varphi_{\pi}(X_1^{(i_1)}, X_2^{(i_2)}, X_3^{(i_3)})
    + \sum_{\substack{i_1,i_2,i_3 \\ i_1 \ne i_3}}
       \varphi_{\pi}(X_1^{(i_1)}, X_2^{(i_2)}, X_3^{(i_3)})  \\
  &= \sum_{\substack{i_1,i_2,i_3 \\ i_1 = i_3}}
  \underbrace{\varphi_{\pi}(X_1, X_2, X_3)}_{\text{$=0$}}
  + \sum_{\substack{i_1,i_2,i_3  \\ i_1 \ne i_3}}
     \varphi_{\hat{0}_3}(X_1, X_2, X_3) \\
&= 
N^2(N-1) \varphi(X_1)\varphi(X_2) \varphi(X_3), 
\end{align*}
so that 
\[
K_{\pi}(X_1, X_2, X_3) = - \varphi(X_1)\varphi(X_2) \varphi(X_3), 
\]
which does not vanish in general although the underlying set partition $\pi$ is not interval. This does not contradict Lemma~\ref{lem:factorization} 
because condition \eqref{eq:lem:factorization1} is not satisfied.
Indeed, the only possible identity would be
$\phi_{\{13/2\}}(Y_1,Y_2,Y_3) = c \phi_{\{1/2/3\}}(Y_1,Y_2,Y_3)$
for every $Y_j\in\bigcup \iota^{(i)}(\alg{A})$.
However the combinations
$\phi_{\{13/2\}}(X_1^{(1)},X_2^{(1)},X_3^{(1)}) =
\phi_{\{13/2\}}(X_1,X_2,X_3) = 0$
and 
$\phi_{\{13/2\}}(X_1^{(1)},X_2^{(2)},X_3^{(3)}) =
\phi_{\{1/2/3\}}(X_1,X_2,X_3) =  \phi(X_1) \phi(X_2) \phi(X_3)$
show that there is no universal constant $c$ satisfying
the requirement.
\end{exa}

\begin{remark}
  We already verified in Section~\ref{sssec:vmonotone} that 
  Dacko's V-monotone independence  \cite{Dacko:2019} satisfies the axioms of a
  spreadability system, but  is non-associative.
  On the other hand, it also
  falls into the framework of the tree operad of Jekel and Liu
  \cite{JekelLiu:2020:operad} with the corresponding cumulants
  and it is an interesting question how these cumulants are related.
\end{remark}

\section{Partial cumulants and differential equations} \label{sec:partialcumulant} 
Neither the defining formula  (Definition~\ref{def:cumulants})
nor the M\"obius formula (Theorem~\ref{thm:mc}) are suitable
for the efficient calculation of cumulants of higher orders.
In the case of exchangeability systems recursive formulas are available
which are more adequate for this purpose; see \cite[Proposition~3.9]{L04}.
In the classical case, the recursion reads as follows:
$$
K_n^\tensorT(X_1,X_2,\dots,X_n) = 
\IE X_1X_2\dotsm X_n - \sum_{\substack{A\subsetneqq [n]\\ 1\in A}} K^\tensorT_{\abs{A}}(X_i:i\in A) \IE\prod_{j\in A^c}X_j
.
$$
In the univariate case this is the familiar formula
$$
\kappa_n = m_n - \sum_{k=1}^{n-1}\binom{n-1}{k-1}\kappa_km_{n-k}
$$
which for normal random variables specifies to
\emph{Stein's method}.

In the free case the recursive formula reads
\begin{equation}
  \label{eq:freerecursion}
  K_n^\freeF(X_1,X_2,\dots,X_n) 
  = \phi(X_1X_2\dotsm X_n) 
     - \sum_{\substack{A\subsetneqq [n]\\ 1\in A}} 
     K^\freeF_{\abs{A}}(X_i:i\in A)
     \,
     \phi_{\intmax(A^c)}(X_j \mid j\in A^c)
\end{equation}
(see Definition~\ref{app:def:ncmax})
which has been recently considered under the name of ``splitting process''
in the realm of combinatorial Hopf algebras \cite{EbrahimiFardPatras:2016:splitting}.

Turning to our general setting we note that already in the case of monotone
probability we lack a simple recursive formula; 
however the first author and Saigo~\cite{HS11b,HS11a} found
a good replacement in terms of differential equations.
Differential equations also play a major role in free probability, 
for example the complex Burgers' equation and its generalizations
appear in the context of free L\'evy processes \cite{V86}. 
In this section we will unify these differential equations from the point of
view of spreadability systems. 

\subsection{Recursive differential equations for evolution of moments}
 Throughout this subsection, let $(\alg{U}, \tilde \varphi, (\iota^{(i)})_{i\in\N})$ be a spreadability system for a $\alg B$-ncps $(\alg A, \varphi)$. 
For $\pi=(P_1,\dots, P_p)\in\OSP_n$ we have observed in Theorem~\ref{thm:iidsum} 
that $\phi_\pi(N_{\pi(1)}. X_1, \dots, N_{\pi(n)}.X_n)$ is a polynomial in
$N_1,\dots, N_p$ and we may formally replace $N_1, \cdots, N_p$ with real
numbers $t_1,\dots, t_p$.
Thus we obtain formal multivariate moment polynomials
\begin{equation*}
\phi^{\ut}_\pi(X_1,\dots, X_n):= \phi_\pi(t_{\pi(1)}. X_1, \dots, t_{\pi(n)}.X_n),\qquad \ut \in \R^p. 
\end{equation*}
We will now establish recursive differential equations for these moment polynomials. 
By Definition~\ref{def:cumulants}, our cumulants are given by  
 \begin{equation}
\label{eq:Kpidiff}
 K_\pi(X_1,\dots, X_n) = \left.\frac{\partial^p}{\partial t_1\partial t_2\dotsm\partial t_p}\right\rvert_{\ut=(0,\dots, 0)} \phi_\pi^{\underline{t}}(X_1,X_2,\dots,X_n)
. 
 \end{equation}
In order to get recursive differential equations we need a refinement of cumulants 
which we call \emph{partial cumulants}. 
They are obtained by taking the derivatives in \eqref{eq:Kpidiff} one at a time.

\begin{defi} Let $\pi=(P_1,P_2,\dots, P_p) \in \OSP_n, \ut=(t_1,\dots, t_p) \in \R^p, j \in[p]$.  We define the \emph{partial cumulant} to be the polynomial
$$
K_{\pi, P_j}^{(t_1,\dots, t_{j-1},1,t_{j+1}, \dots, t_p)}(X_1,\dots,X_n)
=
\left.\frac{\partial }{\partial t_j}\right\rvert_{t_j=0}\phi_{\pi}^{\ut}(X_1,\dots, X_n).
$$ 
\end{defi}
Applying the binomial formula to specific blocks
similar to the proofs of Theorems~\ref{thm:iidsum} 
and~\ref{thm:mc}
it is easy to derive the following explicit expression for the partial cumulants.

\begin{prop}
  Let $\pi=(P_1,P_2,\dots, P_p) \in \OSP_n$, $\ut=(t_1,\dots, t_p) \in \R^p$, $j \in [p]$. 
  Then 
  \begin{multline*}
    K_{\pi, P_j}^{(t_1,\dots, t_{j-1},1,t_{j+1}, \dots, t_p)}(X_1,\dots,X_n)
    \\
    = \sum_{\sigma \in \OSP_{P_j}}\phi_{(P_1,P_2,\dots,P_{j-1},\sigma,P_{j+1},\dots,P_p)}^{(t_1,\dots, t_{j-1},1, t_{j+1}, \dots, t_p)}(X_1,\dots, X_n)\, \widetilde{\mu}(\sigma, \hat{1}_{P_j}).  
  \end{multline*}
\end{prop}

We are now ready to establish partial differential equations for the evolution
of partitioned moments.

\begin{thm}\label{thm diff-eq} For $\pi =(P_1,P_2,\dots, P_p) \in \OSP_n$,
  $\ut =(t_1,t_2,\dots, t_p)$ and $j \in [p]$ we have 
\begin{align}
\frac{\partial}{\partial t_j} \phi^{\ut}_\pi(X_1,\dots, X_n) 
&= \sum_{\emptyset \neq A \subseteq P_j}K_{(P_1,\dots, P_{j-1}, A, P_j\setminus A, P_{j+1}, \dots, P_p),A}^{(t_1,\dots, t_{j-1}, 1, t_j, t_{j+1}, \dots, t_p)}(X_1,\dots, X_n)  \label{eq diff1}\\
&= \sum_{\emptyset \neq A \subseteq P_j}K_{(P_1,\dots, P_{j-1}, P_j\setminus A, A, P_{j+1}, \dots, P_p),A}^{(t_1,\dots, t_{j-1}, t_j,1, t_{j+1}, \dots, t_p)}(X_1,\dots, X_n). \label{eq diff2}
\end{align}
\end{thm}

\begin{proof}
  The main ingredient here is the invariance principle of Lemma~\ref{lem:quasimeetshift}. 
  Recall the delta/dot operation from Definition~\ref{def:dot}\ref{eq:deltaAX}, 
  tensor notations from Section~\ref{Notation} and recall that for
  any partition $\pi=(P_1,P_2,\dots, P_p) \in \OSP_n$ the multilinear map $ \phi_\pi^{\ut}\colon \alg{A}^n \to \C$ is identified with the linear lifting
  $\tilde{\phi}_\pi^{\ut}\colon \alg{A}^{\otimes n} \to \C$.
  Similar to \eqref{notation:linear},
  for a linear map $L\colon \alg{A}\to\alg{A}$ we
  we adopt the notation
  $L^{\otimes  P}\colon \alg{A}^{\otimes n} \to \alg{B}^{\otimes n}$ for the linear map
  \begin{equation*}
L^{\otimes P}=I\otimes I\otimes\dots\otimes L\otimes I\otimes \dots\otimes L\otimes I\otimes\dots\otimes L\otimes\dots\otimes I
  \end{equation*}
  with $L$ appearing exactly at position $j$ for every $j\in P$. 
  Let
  $\bfX=X_1\otimes X_2\otimes \dots\otimes X_n$, let $k + [m]:=\{k+1,k+2,\dots, k+m\}$ and let $e_j \in \R^p$ be the $j$th unit vector. Then, for each $j \in [p]$, 
  \begin{equation}\label{eq901}
   \phi_\pi^{\underline{N}+me_j}(X_1,\dots, X_n) 
   =   \phi_\pi^{\underline{N}+me_j}(\bfX)  
      = \phi_\pi\left(      \left(\delta_{[N_j+m]}^{\otimes P_j} \prod_{i\ne j}
          \delta_{[N_i]}^{\otimes P_i}\right) (\bfX)\right)
      ;
    \end{equation}
    after expansion we obtain
\begin{equation}\label{eq9010}
\delta_{[N_j+m]}^{\otimes P_j} = (\delta_{[N_j]}+\delta_{N_j+[m]})^{\otimes P_j}
= \sum_{A\subseteq P_j} \delta_{[N_j]}^{\otimes P_j \setminus A} \delta_{N_j+[m] }^{\otimes A}
\end{equation}
and thus \eqref{eq901} is equal to
\begin{align}\label{eq902}
&= \sum_{A\subseteq P_j} 
       \phi_\pi\!\left(
        \left(\delta_{[N_j]}^{\otimes P_j \setminus A} \delta_{N_j+[m] }^{\otimes A}  \prod_{i\ne j} \delta_{[N_i]}^{\otimes P_i}\right) (\bfX)
         \right).
\intertext{This is a sum of $\phi_\pi$'s with entries $X_{k}^{(i_k)}, k\in[n]$, 
where the indices $i_k$ with $k\in A$ are strictly larger than the indices
  $i_k$ with $k\in P_j \setminus A$, and therefore we may split the block $P_j \in \pi$ into two parts:
}
\nonumber
&= \sum_{A\subseteq P_j} 
       \phi_{(P_1,P_2,\dots,P_{j-1},P_j\setminus A, A, P_{j+1},\dots,P_p)}\left(
        \left( \delta_{[N_j]}^{\otimes P_j \setminus A} \delta_{N_j+[m] }^{\otimes A} \prod_{i\ne j} \delta_{[N_i]}^{\otimes P_i}\right) (\bfX) 
         \right). 
\intertext{
  Now by Lemma~\ref{lem:quasimeetshift} 
  the local shift in $A$ can be omitted without changing the value 
  and we obtain
}
\label{eq9030}
&=  \sum_{A\subseteq P_j} 
       \phi_{(P_1,P_2,\dots,P_{j-1},P_j\setminus A, A, P_{j+1},\dots,P_p)}\left(
        \left( \delta_{[N_j]}^{\otimes P_j \setminus A} \delta_{[m]}^{\otimes A} \prod_{i\ne j} \delta_{[N_i]}^{\otimes P_i}\right) (\bfX) 
         \right)
\\
    &= \sum_{A\subseteq P_j} 
       \phi_{(P_1,P_2,\dots,P_{j-1},P_j\setminus A, A, P_{j+1},\dots,P_p)}^{(N_1,N_2,\dots,N_{j-1},N_j,m,N_{j+1},\dots,N_p)}
      (\bfX).
      \nonumber
\end{align}
The analytic extension of this identity is
  \begin{align*}
    \phi_\pi^{\underline{t}+se_j}(\bfX)
    &= \sum_{A\subseteq P_j} 
       \phi_{(P_1,P_2,\dots,P_{j-1},P_j\setminus A, A, P_{j+1},\dots,P_p)}^{(t_1,t_2,\dots,t_{j-1},t_j,s,t_{j+1},\dots,t_p)}
       (\bfX)\\
    &= \phi_\pi^{\underline{t}}(\bfX)
       +
       \sum_{\emptyset\ne A\subseteq P_j} 
       \phi_{(P_1,P_2,\dots,P_{j-1},P_j\setminus A,A, P_{j+1},\dots,P_p)}^{(t_1,t_2,\dots,t_{j-1},t_j,s,t_{j+1},\dots,t_p)}
       (\bfX),  
  \end{align*}
  and the derivative satisfies the derived identity
 \eqref{eq diff2}: 
  $$
  \frac{\partial}{\partial t_j} \phi_\pi^{\underline{t}}(\bfX)
  =
  \left.
    \frac{\partial}{\partial s}
  \right\rvert_{s=0}
  \phi_\pi^{\underline{t}+se_j}(\bfX)
  = \sum_{\emptyset\ne A\subseteq P_j}
     K_{(P_1,P_2,\dots,P_{j-1},P_j\setminus A,A,P_{j+1},\dots,P_p)}^{(t_1,t_2,\dots,t_{j-1},t_j,1,t_{j+1},\dots,t_p)}
       (\bfX). 
  $$
  In order to prove the first differential equation \eqref{eq diff1}
  we replace \eqref{eq9010} by the complementary expansion
  \begin{equation*}
\delta_{[N_j+m]}^{\otimes P_j} = (\delta_{[m]}+\delta_{m+[N_j]})^{\otimes P_j}
= \sum_{A\subseteq P_j} \delta_{[m]}^{\otimes A} \delta_{m+[N_j]}^{\otimes P_j \setminus A}, 
\end{equation*}
and following the lines of \eqref{eq902}--\eqref{eq9030} we obtain 
\begin{multline*}
   \phi_\pi^{\underline{N}+me_j}(X_1,\dots, X_n) \\
   \begin{aligned}[t]
&= \sum_{A\subseteq P_j} 
       \phi_\pi\!\left(
        \left( \delta_{[m]}^{\otimes A} \delta_{m+[N_j]}^{\otimes P_j \setminus A}  \prod_{i\ne j} \delta_{[N_i]}^{\otimes P_i}\right) (\bfX)
         \right) 
   \\
  &=  \sum_{A\subseteq P_j} 
       \phi_{(P_1,P_2,\dots,P_{j-1},A, P_j\setminus A, P_{j+1},\dots,P_p)}\left(
        \left( \delta_{[m]}^{\otimes A} \delta_{[N_j]}^{\otimes P_j \setminus A}\prod_{i\ne j} \delta_{[N_i]}^{\otimes P_i}\right) (\bfX) 
         \right)
\\
    &= \sum_{A\subseteq P_j} 
       \phi_{(P_1,P_2,\dots,P_{j-1},A,P_j\setminus A, P_{j+1},\dots,P_p)}^{(N_1,N_2,\dots,N_{j-1},m,N_j,N_{j+1},\dots,N_p)}
       (\bfX).      
   \end{aligned}
\end{multline*}
By analytic continuation, we may replace $m$ and $N_i$ with $s$ and $t_i$ respectively. Taking the derivative with respect to $s$ at $0$ we obtain \eqref{eq diff1}. 
\end{proof}

\begin{remark}
  In the case of exchangeability both differential equations coincide.   
\end{remark}

\subsection{Examples from universal products}

We consider specializations of the differential
equations of Theorem~\ref{thm diff-eq} to various spreadability systems arising
from universal products. 
In all these examples an expansion of the form  \eqref{eq:lem:factorization1}
holds for partitioned expectations $\phi_\pi^{\ut}$ 
and therefore it suffices to consider $\pi=\hat{1}_n$, i.e.,
the expectation $\phi^t(X_1,\dots, X_n) := \phi((t.X_1)(t.X_2)\cdots (t.X_n))$.
\begin{exa}[Tensor independence]
  Consider the tensor spreadability system   $\cS_{\tensorT}$.
  It satisfies the factorization property   \eqref{eq:otimesindep:factor}
  and therefore 
  \begin{equation*}
    \varphi_{(A,A^c)}^{(t_1,t_2)}(X_1,\dots, X_n) = \phi^{t_1}(X_A)\, \phi^{t_2}(X_{A^c}), 
  \end{equation*}
  so we get 
  \begin{equation*}
    K_{(A,A^c), A}^{(1,t_2)}(X_1,X_2,\dots, X_n) 
    = \left.\frac{\partial}{\partial t_1}\right\rvert_{t_1=0} \phi^{t_1}(X_A) \,\phi^{t_2}(X_{A^c}) 
    = K_{|A|}^{\tensorT}(X_A)\, \phi^{t_2}(X_{A^c}). 
  \end{equation*}
  The  identities in Theorem~\ref{thm diff-eq} (for $\pi=\hat{1}_n$) read
  \begin{equation}\label{eqTdiff}
    \frac{d}{d t} \phi^t (X_1,\dots, X_n) = \sum_{\emptyset \neq A \subseteq [n]} K_{|A|}^{\tensorT}(X_A) \,\phi^{t}(X_{A^c}). 
  \end{equation}

This differential equation can be translated to (exponential) 
generating functions as follows.
Given a vector $\uu=(u_1,u_2,\dots, u_n)$  of \emph{commuting} indeterminates 
and a vector $\uX=(X_1,X_2,\dots, X_n)$ of random variables
we define the exponential moment generating function
\begin{align*}
\egf_{\uX}^t(\uu)
&:= 1+\sum_{\substack{(p_1,\dots, p_n) \in (\N \cup \{0\})^n \\  (p_1,\dots, p_n)\neq0}} 
  \frac{u_1^{p_1}\cdots u_n^{p_n}}{p_1! \cdots p_n!}\,
  \phi^t(\underbrace{X_1,\dots, X_1}_{\text{$p_1$ times}}, \dots, \underbrace{X_n,\dots X_n}_{\text{$p_n$ times}})\\
&= [\phi (e^{u_1 X_1} \cdots e^{u_n X_n})]^t
\end{align*}
and the exponential cumulant generating function as the logarithm of the previous
\begin{align*}
\lap_{\uX}(\uu)
&:=  \sum_{\substack{(p_1,\dots, p_n) \in (\N \cup \{0\})^n \\ (p_1,\dots, p_n)\neq0}} 
\frac{u_1^{p_1}\cdots u_n^{p_n}}{p_1! \cdots p_n!}
\,
K_{p_1+\dots +p_n}^{\tensorT}(\underbrace{X_1,\dots, X_1}_{\text{$p_1$ times}}, \dots, \underbrace{X_n,\dots X_n}_{\text{$p_n$ times}})\\
&= \log[\egf_{\uX}^1(\uu)]. 
\end{align*}
Then one can prove that  
\begin{equation*}
\frac{d}{dt} \egf_{\uX}^t(\uu) = \lap_{\uX}(\uu) \, \egf_{\uX}^t(\uu), 
\end{equation*}
which is equivalent to \eqref{eqTdiff}. Note that the functions $\lap_{\uX}(\uu)$ and $\egf_{\uX}^t(\uu)$ commute. 
\end{exa}

\begin{exa}[Boolean independence]\label{ex:Boole} 
  In the Boolean spreadability system $\cS_{\boolB}$ we have
  \begin{equation*}
    \varphi_{(A,A^c)}^{(t_1,t_2)}(X_1,X_2,\dots, X_n)
    = \prod_{P \in \outintmax(A)}\phi^{t_1}(X_P) \prod_{Q\in \outintmax(A^c)}\phi^{t_2}(X_Q).
  \end{equation*}
  where $\outintmax(A)$ is  the interval partition
  constructed in Definition~\ref{app:def:ncmax}
  and 
  consists of the contiguous subintervals of $A$.
  It follows that 
  \begin{equation*}
      K_{(A,A^c), A}^{(1,t_2)}(X_1,X_2\dots, X_n) 
      = 
      \begin{cases}
        K_{|A|}^{\boolB}(X_A)  \prod_{Q\in \outintmax(A^c)}\phi^{t_2}(X_Q), & \text{if $|\outintmax(A)|=1$}, \\ 
        0, & \text{if $|\outintmax(A)|>1$}.  
      \end{cases}
  \end{equation*}
The identities in Theorem~\ref{thm diff-eq} (for $\pi=\hat{1}_n$) coincide
and read
\begin{equation}\label{eqBdiff}
\frac{d}{d t} \phi^t (X_1,X_2,\dots, X_n) = \sum_{A: \text{~interval of~} [n]} K_{|A|}^{\boolB}(X_A) \prod_{Q\in \outintmax(A^c)}\phi^{t}(X_Q). 
\end{equation}
This differential equation can be interpreted in terms of generating functions
in {\it noncommuting} indeterminates $z_1,\dots, z_n$. 
To this end we define noncommutative formal power series  
\begin{align*}
M_{\uX}^t (\uz) 
&=1 +
 \sum_{m=1}^\infty \sum_{i_1, \cdots, i_m = 1} ^n 
\varphi^t(X_{i_1},X_{i_2},\dots, X_{i_m})
\,z_{i_1}\cdots z_{i_m}
\intertext{and }
K_{\uX}^{\boolB} (\uz) 
&= \left.\frac{\partial}{\partial t}\right\rvert_0 M_{(X_1,X_2,\dots, X_n)}^t (\uz) \\
&= \sum_{m=1}^\infty \sum_{i_1, \cdots, i_m = 1} ^n 
K_m^{\boolB}(X_{i_1},\dots, X_{i_m})
\,
z_{i_1}\cdots z_{i_m}. 
\end{align*} 
Then differential equation \eqref{eqBdiff} is equivalent to the identity 
\begin{equation*}
\frac{d}{d t} M_{\uX}^t(\uz) 
= M_{\uX}^t(\uz)\, K_{\uX}^{\boolB} (\uz)\,  M_{\uX}^t(\uz). 
\end{equation*}
Note that $M_{\uX}^t(\uz)$ and $K_{\uX}^{\boolB} (\uz)$ do not commute. 

\end{exa}

\begin{exa}[Monotone independence]
  \label{ex:monotone-diff-eq}
  Consider the monotone spreadability system $\cS_{\monoM}$.
  With the notation $\outintmax(A)$ introduced in 
  Definition~\ref{app:def:ncmax} we have
\begin{equation*}
\varphi_{(A,A^c)}^{(t_1,t_2)}(X_1,X_2,\dots, X_n) = \phi^{t_1}(X_A) \prod_{B \in \outintmax(A^c)} \phi^{t_2}(X_B), 
\end{equation*}
using monotone independence.
Thus
\begin{align*}
  K_{(A,A^c), A}^{(1,t_2)}(X_1,X_2,\dots, X_n)
  &=  \left.\frac{\partial}{\partial t_1}\right\rvert_{t_1=0} \phi^{t_1}(X_A)
    \prod_{B \in \outintmax(A^c)} \phi^{t_2}(X_B)
    \\
    &= K_{|A|}^{\monoM}(X_A)\prod_{B \in \outintmax(A^c)} \phi^{t_2}(X_B). 
\end{align*}
The first identity in Theorem~\ref{thm diff-eq} (for $\pi=\hat{1}_n$) reads 
\begin{equation}\label{eq diff m1}
\frac{d}{d t} \phi^t (X_1,X_2,\dots, X_n) = \sum_{ \emptyset \neq A \subseteq [n]} K_{|A|}^{\monoM}(X_A) \prod_{B \in \outintmax(A^c)} \phi^{t_2}(X_B), 
\end{equation}
which is exactly the first identity in \cite[Corollary 5.2]{HS11b}. 

On the other hand 
\begin{align*}
K_{(A^c,A), A}^{(t_1,1)}(X_1,X_2,\dots, X_n) 
&=  \left.\frac{\partial}{\partial t_2}\right\rvert_{t_2=0} \phi^{t_1}(X_A) \prod_{B \in \outintmax(A^c)} \phi^{t_2}(X_B) \\
&= 
\begin{cases}
\phi^{t_1}(X_A)\, K_{|A^c|}^{\monoM}(X_{A^c}), & \text{if $|\outintmax(A^c)|=1$}, \\ 
0, & \text{if $|\outintmax(A^c)|>1$}.   
\end{cases}
\end{align*}
 Condition $|\outintmax(A^c)|=1$ holds if and only if $A^c$ is an interval. Therefore, the second equality in Theorem~\ref{thm diff-eq} (for $\pi=\hat{1}_n$) reads 
\begin{equation}\label{eq diff m2}
\frac{d}{d t} \phi^t (X_1,X_2,\dots, X_n) 
= \sum_{B: \text{~interval of~} [n]} 
  K_{|B|}^{\monoM}(X_B)\,\phi^t(X_{B^c}), 
\end{equation}
which is exactly the second equality in \cite[Corollary 5.2]{HS11b}. 

Results on generating functions in \cite{HS11b} 
correspond to these differential equations. For noncommutative indeterminates $z_1,\dots, z_n$, we define the cumulant generating function 
\begin{align*}
K_{\uX}^{\monoM} (\uz) 
&:= \left.\frac{\partial}{\partial t}\right\rvert_0 M_{\uX}^t (\uz) \\
&= \sum_{m=1}^\infty \sum_{i_1,i_2, \cdots, i_m = 1} ^n 
    K_m^{\monoM}(X_{i_1}X_{i_2}\cdots X_{i_m})
    \,
    z_{i_1}z_{i_2}\cdots z_{i_m}. 
\end{align*}
Then it is shown in \cite[Theorem 6.3]{HS11b} that 
\begin{equation}\label{eq m convolution}
M_{\uX}^{s+t} (\uz)  =  M_{\uX}^t (\uz)\, M_{\uX}^s (z_1M_{\uX}^t (\uz) , \dots, z_n M_{\uX}^t (\uz) ).  
\end{equation}
The partial derivatives of \eqref{eq m convolution} regarding $s$ at $0$ and $t$ at $0$ become \eqref{eq diff m1} and \eqref{eq diff m2}, respectively. 

\end{exa}

\begin{exa}[Free independence] 
  \label{ex:freediff}
Consider the free spreadability system $\cS_{\freeF}$. 
One can use the formula for products of free random variables
\cite[Theorem 14.4]{NicaSpeicher:2006:lectures} to show that 
for a nonempty subset $A\subseteq [n]$ we can expand
 \begin{equation*}
 \varphi_{(A,A^c)}(X_1,X_2,\dots, X_n) = \phi(X_A)\prod_{P\in \intmax(A^c)}\phi(X_P) + R, 
 \end{equation*}
 where
 $\intmax(A^c)$ is the partition defined in
 Definition~\ref{app:def:ncmax} \ref{app:it:ncmax}
 and every term in $R$ has at least two factors from $A$, i.e., factors of the form $\phi(X_{k_1}X_{k_2}\cdots X_{k_m})$, $k_1,k_2,\dots, k_m \in A$. For example,
$\phi_{\{\{2,4\}, \{1,3,5\}\}}(X_1,X_2,\dots, X_5) = \phi(X_2 X_4)\, \phi(X_1
X_5)\, \phi(X_3) + R$,
where every term in $R$ contains the factor $\phi(X_2)\,\phi(X_4)$. 
 This implies that 
 \begin{equation*}
\varphi_{(A,A^c)}^{(t_1,t_2)}(X_1,X_2,\dots, X_n) = \phi^{t_1}(X_A)\prod_{P\in
  \intmax(A^c)}\phi^{t_2}(X_P) +O(t_1^2) \qquad \text{as $t_1\to0$}, 
\end{equation*}
so by taking the partial derivative $ \left.\frac{\partial}{\partial t_1}\right\rvert_{t_1=0}$ we get 
\begin{equation*}
K_{(A,A^c), A}^{(1,t_2)}(X_1,X_2,\dots, X_n) = K_{|A|}^{\freeF}(X_A)\prod_{P\in \intmax(A^c)}\phi^{t_2}(X_P).   
\end{equation*}
This yields the differential equation 
\begin{equation}\label{diff:free}
\frac{d}{d t} \phi^t (X_1,X_2,\dots, X_n) = \sum_{\emptyset \neq A \subseteq [n]} K_{|A|}^{\freeF}(X_A)\prod_{P\in \intmax(A^c)}\phi^t(X_P),   
\end{equation}
which is similar to the monotone case \eqref{eq diff m1}. 

Denote by $\uz=(z_1,z_2, \dots, z_n)$ a vector of noncommuting indeterminates and
by $\uX=(X_1,X_2,\dots,X_n)$ a random vector. 
In order to obtain a differential equation we need in
addition the two-sided generating function
\begin{equation*}
\tilde{M}^t_{\uX}(\uz,w):= \sum_{p,q=0}^\infty\sum_{\substack{i_1,i_2,\dots, i_p \in[n] \\ j_1,j_2,\dots, j_q \in[n]}} \phi^t(X_{i_1},X_{i_2}, \dots, X_{i_p}, X_{j_1},X_{j_2}, \dots, X_{j_q})\,z_{i_1}z_{i_2}\cdots z_{i_p}w z_{j_1}z_{j_2}\cdots z_{j_q} 
\end{equation*}
as well as the $R$-transform
\begin{equation*}
R_{\uX}(\uz)=\sum_{j=1}^\infty \sum_{i_1, \cdots, i_j = 1} ^n K_j^{\freeF}(X_{i_1},X_{i_2},\cdots, X_{i_j})\,z_{i_1}z_{i_2}\cdots z_{i_j}. 
\end{equation*}
After some computations we infer from \eqref{diff:free} that 
\begin{equation}\label{diff:free2}
\frac{\partial }{\partial t} M^t_{\uX}(\uz)
  = \tilde{M}^t_{\uX}(\uz, R(z_1M^t_{\uX}(\uz),z_2 M^t_{\uX}(\uz) ,\dots, z_n M^t_{\uX}(\uz)) (M^t_{\uX}(\uz))^{-1}). 
\end{equation}
Note that by \cite[Corollary 16.16]{NicaSpeicher:2006:lectures} the second argument can be written as
\begin{equation*}
R(z_1 M^t_{\uX}(\uz), z_2 M^t_{\uX}(\uz) ,\dots, z_n M^t_{\uX}(\uz))\,(M^t_{\uX}(\uz))^{-1} =\frac{1- M^t_{\uX}(\uz)^{-1}}{t}.
\end{equation*}
When $n=1$ the differential equation \eqref{diff:free2} is equivalent to the generalized complex Burgers equation (see \cite[p.\ 343]{V86}) 
\begin{equation*}
\frac{\partial }{\partial t} G^t_{X}(z) + \frac{R_X(G_X^t(z))}{G^t_X(z)} \frac{\partial }{\partial z} G^t_{X}(z)=0, 
\end{equation*}
where $G^t_X$ is the Cauchy transform
\begin{equation*} 
G^t_X(z)= \sum_{n=0}^\infty \phi^t(\underbrace{X,\dots,X}_{\text{$n$ fold}}) \, z^{-n-1}. 
\end{equation*}
\end{exa}

\section{Mixed cumulants and sums of independent random variables}
\label{sec:vanishing}
\subsection{Vanishing of mixed cumulants in exchangeability systems}

In the case of exchangeability systems independence
the 
vanishing of mixed cumulants
(Axiom~\ref{it:Ecumulants:vanish} in Definition~\ref{def:exch:cumulants})
is the crucial property which makes cumulants interesting,
since it characterizes independence \cite[Prop.~2.10]{L04}.
That is, if the arguments of $K_n(X_1,X_2,\dots,X_n)$  can be split into two mutually
independent families then the cumulant vanishes;
more generally,  $K_\pi(X_1,X_2,\dots,X_n)=0$  whenever the entries of one of the blocks of
$\pi$  splits into two mutually independent subsets. This is the content
of the following proposition.
\begin{prop}
  \label{prop:mixedvanish}
  Let $(\alg{A},\phi)$ be a $\alg{B}$-ncps and 
  $\exch=(\alg{U},\tilde{\varphi},(\iota^{(i)})_{i\ge1})$ an exchangeability system for $\alg{A}$.
  Given a partition $\pi\in\SP_n$ and a family
  $X_1,X_2,\dots,X_n\in \alg{A}$ such that there is a block
  $P\in \pi$ which can be partitioned into $P=P_1\dot\cup P_2$
  such that $\{X_i:i\in P_1\}$ and  $\{X_i:i\in P_2\}$ are
  independent in the sense of Definition~\ref{def:exchangeableindependence},
  we have $K_\pi(X_1,X_2\dots,X_n)=0$.
\end{prop}

For further reference we reproduce here a short lattice theoretic proof
due to P.~Zwier\-nik \cite{Zwiernik:2012:cumulants},
which will serve as a model for the spreadable case.
It is based on  \emph{Weisner's Lemma} 
(see \cite[Cor.~3.9.3]{Stanley:1986:enumerative1} for a simple version and
\cite{BarnabeiBriniRota:1986:theory} for the full version). 

Its generalization will be essential for the understanding of mixed cumulants in the spreadable setting.
\begin{lemma}[Weisner's Lemma]\label{lemma Weisner}
  In any lattice $(P,\leq)$ the M\"obius function satisfies the identity
  $$
  \sum_{\substack{x\\ x\wedge a=c}} \mu(x,b) =
  \begin{cases}
    \mu(c,b) & \text{if $a\geq b$}\\ 
    0 &\text{if $a\not\geq b$}.
  \end{cases}
  $$
\end{lemma}

\begin{proof}[Proof of Proposition~\ref{prop:mixedvanish}]
Let $\rho$ be the partition
obtained from $\pi$ by splitting the block $P$
as indicated in the proposition, then $\rho<\pi$ and by assumption
$\phi_\sigma(X_1,X_2,\dots,X_n)=\phi_{\sigma\wedge\rho}(X_1,X_2,\dots,X_n)$
for any $\sigma\in\SP_n$; hence
\begin{align*}
K_\pi(X_1,X_2,\dots,X_n)
&= \sum_{\substack{\sigma\in\SP_n \\\sigma\leq\pi}} \phi_\sigma(X_1,X_2,\dots,X_n)\, \mu_\SP(\sigma,\pi) \\
&= \sum_{\substack{\sigma\in\SP_n \\\sigma\leq\pi}} \phi_{\sigma\wedge\rho}(X_1,X_2,\dots,X_n)\, \mu_\SP(\sigma,\pi) \\
& =\sum_{\tau\in\SP_n} \phi_\tau(X_1,X_2,\dots,X_n) \sum_{\substack{\sigma \in\SP_n\\ \sigma\wedge\rho =\tau}}\, \mu_\SP(\sigma,\pi). 
\end{align*}
Now $\rho\not\geq \pi$ and the second case of Weisner's lemma applies.
\end{proof}

\subsection{Partial vanishing of mixed cumulants in spreadability systems}
Since vanishing of mixed cumulants implies  additivity of cumulants
\eqref{eq:additivity}
for sums of independent random variables,
it cannot hold for general spreadability systems, e.g., monotone convolution
is noncommutative and therefore monotone cumulants are not additive \cite{HS11b}.

In this section we investigate what remains true in the
general setting and provide a formula expressing mixed cumulants in terms of lower order cumulants.
This question is intimately related to the question of convolution -
determining the distribution of the sum of independent random variables,
which is not commutative in general.
$\OSP_n$ with the quasi-meet operation $\curlywedge$ is not a lattice
and Weisner's lemma~\ref{lemma Weisner} does not hold.
This means that mixed cumulants do not vanish,
yet they can be expanded in terms of lower order cumulants weighted by certain coefficients,
which we will study next.
\begin{defi}
  \label{defi:weisnergoldberg}
  For $\tau,\eta \in \OSP_n$ let
  \begin{align}
    \label{eq:defi:weisner}
 w(\tau,\eta)&:=\sum_{\substack{\sigma \in \OSP_n \\ \sigma \curlywedge
    \eta=\tau}}
    \widetilde{\mu}(\sigma,\hat{1}_n), 
             && \text{(Weisner coefficients)}
    \\
    \label{eq:defi:goldberg}    
\gold(\tau,\eta)&:=\sum_{\substack{\sigma \in \OSP_n \\ \sigma \geq
    \tau}}\widetilde{\zeta}(\tau,\sigma)\,w(\sigma,\eta), 
             &&\text{(Goldberg coefficients)}
  \end{align}
  and more generally for $\tau,\eta,\pi \in \OSP_n$
  define the partitioned Weisner and Goldberg coefficients
  \begin{align}
    \label{eq:defi:partitionedweisner}    
    w(\tau,\eta,\pi)
    &:=\sum_{\substack{\sigma \in \OSP_n \\ \sigma \curlywedge
    \eta=\tau\\ \sigma \leq \pi}} \widetilde{\mu}(\sigma,\pi), 
    \\
    \label{eq:defi:partitionedgoldberg}        
    \gold(\tau,\eta,\pi)
                    &:=\sum_{\substack{\sigma \in \OSP_n \\  \sigma \geq \tau}}\widetilde{\zeta}(\tau,\sigma)\,w(\sigma,\eta,\pi).     
  \end{align}
\end{defi}

\begin{remark}
  \label{rem:weisnervanish}
  It follows from the properties of the quasi-meet operation
  (cf.~Proposition~\ref{app:prop:quasimeet}
  \ref{app:it:barpiwsi=barwbar} and 
  \ref{app:it:piwsi<pi})
  that both
  $w(\tau,\eta,\pi)$ and $g(\tau,\eta,\pi)$ vanish unless
  $\bar{\tau}\leq\bar{\eta}$ and $\tau\leq\pi$.
\end{remark}
We postpone further combinatorial study of these coefficients to Section~\ref{ssec:goldberg} 
and first exhibit their role as a substitute for Weisner's lemma~\ref{lemma Weisner} in the description of cumulants of independent arguments.

\begin{prop}
  \label{cum-mom}
Let $\cS=(\alg{U},\tilde{\phi}, (\iota^{(i)})_{i\ge 1})$ be a spreadability system for a given
  $\alg B$-ncps $(\alg{A},\phi)$.  Let $\pi\in\OSP_n$ and
  $(i_1,i_2,\ldots,i_n) \in \N^n$  be a tuple with kernel
  $\eta=\kappa(i_1,i_2,\ldots,i_n)\in\OSP_n$.
  Then
  \begin{align}
    K_\pi(X_1^{(i_1)}, X_2^{(i_2)}, \ldots, X_n^{(i_n)})
    &= \sum_{\substack{\tau \in \OSP_n \\ \tau\leq\pi\\ \bar{\tau}\leq\bar{\eta}}}
        \varphi_\tau(X_1,X_2,\ldots,X_n) \, w(\tau,\eta,\pi)
      \label{eq:cum-mom}
    \\
    &= \sum_{\substack{\tau \in \OSP_n \\ \tau\leq\pi\\ \bar{\tau}\leq\bar{\eta}}}
       K_\tau(X_1,X_2,\ldots,X_n) \, \gold(\tau,\eta,\pi)
      \label{eq:cum-cum}      
      .
  \end{align}
\end{prop}
\begin{proof} 
We proceed as in the proof of Proposition~\ref{prop:mixedvanish}.
Expressing cumulants in terms of moments (see Theorem~\ref{thm:mc}), we have 
\begin{align*}
K_\pi(X_1^{(i_1)},X_2^{(i_2)}, \ldots, X_n^{(i_n)}) 
&= \sum_{\substack{\sigma \in \OSP_n \\ \sigma \leq \pi}}
\varphi_\sigma(X_1^{(i_1)}, X_2^{(i_2)},\ldots,X_n^{(i_n)})\,\widetilde{\mu}(\sigma,\pi) \\
&=  \sum_{\substack{\sigma \in \OSP_n \\ \sigma \leq \pi}} \varphi_{\sigma\curlywedge\eta}(X_1,X_2,\ldots,X_n)\, \widetilde{\mu}(\sigma,\pi)  \\
&=\sum_{\tau \in \OSP_n}  \varphi_\tau(X_1,X_2,\ldots,X_n)\,\biggl(\sum_{\substack{\sigma \in \OSP_n \\ \sigma\curlywedge\eta =\tau, \sigma \leq \pi }} \widetilde{\mu}(\sigma,\pi) \biggr) \\
&= \sum_{\tau \in \OSP_n} \varphi_\tau(X_1,X_2,\ldots,X_n) \, w(\tau,\eta,\pi)
\\
\intertext{which is \eqref{eq:cum-mom}. Now substitute the moment-cumulant
  formula     \eqref{eq:phipi=sumKsi}   to obtain }
&=  \sum_{\tau \in \OSP_n} \sum_{\substack{\sigma \in\OSP_n\\ \sigma \leq \tau}}K_\sigma(X_1,X_2,\ldots,X_n) \, \widetilde{\zeta}(\sigma,\tau)\,w(\tau,\eta,\pi) \\
&=  \sum_{\sigma \in \OSP_n} K_\sigma(X_1,X_2,\ldots,X_n) \,\biggl(\sum_{\substack{\tau \in\OSP_n\\ \tau\geq\sigma}}\widetilde{\zeta}(\sigma,\tau)\,w(\tau,\eta,\pi)\biggr) \\
&=  \sum_{\sigma \in \OSP_n} K_\sigma(X_1,X_2,\ldots,X_n)\,\gold(\sigma,\eta,\pi).  
\end{align*}
Finally Remark~\ref{rem:weisnervanish} applies and the sums \eqref{eq:cum-mom} and \eqref{eq:cum-cum}
can be restricted to $\tau\leq\pi$ and $\bar{\tau}\leq\bar{\eta}$.
\end{proof}

We can now characterize $\cS$-independence in terms of the above proposition, that is, ``semi-vanishing'' of mixed cumulants.  

\begin{thm} \label{cum vanishing} Let $\cS=(\alg{U},\tilde{\phi},(\iota^{(i)})_{i\ge1})$ be a spreadability system for a given
  $\alg B$-ncps $(\alg{A},\phi)$. A sequence of subalgebras $(\alg{A}_i)_{i\in I}$ of $\alg{A}$, where 
  $I\subseteq \N$, is \emph{$\cS$-independent} if and only if for any tuple $(i_1,i_2,\dots,i_n)\in I^n$, any random variables $(X_1,X_2,\dots,X_n)\in \alg{A}_{i_1} \times \alg{A}_{i_2} \times \cdots \times \alg{A}_{i_n}$ and any ordered set partition $\pi\in\OSP_n$, we have 
\begin{equation}\label{semi-vanishing1}
  K_\pi(X_1,X_2, \ldots, X_n)
  - \sum_{\substack{\tau \in \OSP_n\\ \tau\leq\pi}}
  K_\tau(X_1,X_2,\ldots,X_n)
  \, \gold(\tau,\kappa(i_1,\dots,i_n),\pi) = 0.  
\end{equation}
\end{thm}
\begin{remark}\label{semi-vanishing3}
We can also formulate the theorem in terms of moments; we only need to replace
\eqref{semi-vanishing1} by the equation
\begin{equation*}
  K_\pi(X_1,X_2, \ldots, X_n)
  - \sum_{\substack{\tau \in \OSP_n\\ \tau\leq\pi}}
\phi_\tau(X_1,X_2,\ldots,X_n) \, w(\tau,\kappa(i_1,\dots,i_n),\pi) = 0.
\end{equation*}
\end{remark}

\begin{proof}
  Fix a tuple $(i_1,\dots, i_n)$ with kernel
  $\eta=\kappa(i_1,\dots,i_n)$.
  Independence means that
  \begin{align*}
  \phi_\pi(X_1,X_2,\dots,X_n)
    &= \phi_{\pi\curlywedge\eta}(X_1,X_2,\dots,X_n)
      \\
    &= \phi_\pi(X_1^{(i_1)},X_2^{(i_2)},\dots,X_n^{(i_n)})
  \end{align*}
  for all $\pi \in \OSP_n$,
  and by M\"{o}bius inversion this is equivalent to the identity
  \begin{align*}
  K_\pi(X_1,X_2,\dots,X_n)
    &= K_\pi(X_1^{(i_1)},X_2^{(i_2)},\dots,X_n^{(i_n)})
      \\
    &=
      \sum_{\substack{\tau \in \OSP_n\\ \tau\leq\pi}}
      K_\tau(X_1,X_2,\ldots,X_n) \, g(\tau,\kappa(i_1,\dots,i_n),\pi)
  \end{align*}
  for all $\pi\in\OSP_n$.
\end{proof}

\begin{remark}
It is not obvious that Theorem~\ref{cum vanishing} generalizes Proposition~\ref{prop:mixedvanish}
(vanishing of mixed cumulants for $\exch$-independent subalgebras).
In fact it implies a nontrivial identity: for any $ \tau,\eta,\pi \in \OSP_n$ such that $\eta\restr_P \neq \hat{1}_P$ for some $P\in\pi$, 
\begin{equation}\label{eq permute}
\sum_{h \in \SG_{|\tau|}} \gold(h(\tau), \eta, \pi)=0, 
\end{equation}
where $h(\tau)$ is the action of the permutation $h$ on the blocks of $\tau$. 
Similarly, Remark~\ref{semi-vanishing3} (or Proposition~\ref{cum-mom})
generalizes the vanishing of mixed cumulants for $\exch$-independent
subalgebras. Consequently we must have  
\begin{equation}\label{eq permute2}
\sum_{h \in \SG_{|\tau|}} w(h(\tau), \eta, \pi)=0 
\end{equation}
under the same assumptions on $\tau,\eta,\pi.$
\end{remark}

\subsection{Combinatorial evaluation of Weisner and Goldberg coefficients}
\label{ssec:goldberg}
The combinatorial description of the coefficients introduced
in Definition~\ref{defi:weisnergoldberg}
involves certain statistics of multiset permutations.
The first systematic study of these permutation statistics is contained
in the seminal work of MacMahon \cite{MacMahon:1915:combinatory}, for a modern
treatment see \cite{Bona:2012:combinatorics}. Said statistics play a major
role in the theory of free Lie algebras \cite{Reutenauer:1993:freeLiealg},
which also seem to play a role in our context, see Section~\ref{sec:freeLiealg} below.

\begin{defi}
  A \emph{multiset} is a pair $(A,f)$ where $A$ is the \emph{underlying set} and $f:A\to\N$ is a
  function. The value $f(a)$ is called the \emph{multiplicity} of the element
  $a\in A$. Informally, a multiset is a set which contains multiple
  indistinguishable  copies of each of its elements. In the present paper
  multisets will always be based on integer  segments $A=[n]$ and
  in this case the tuple $(f(1),f(2),\dots,f(n))$ is called the 
  \emph{type} of the multiset.
  A \emph{permutation of a multiset} is a rearrangement of all its elements 
  where all multiplicities are preserved, i.e., a word with a prescribed total number
  of occurrences of each letter.
  The number of distinct permutations of a multiset of type
  $(f_1,f_2,\dots,f_n)$ is given by the multinomial coefficient
  $$
  \binom{f_1+f_2+\dotsm+ f_n}{f_1,f_2,\dots,f_n}. 
  $$
  A \emph{proper set} is a multiset of type $(1,1,\dots,1)$ and we recover
  the number of its permutations as $n!$.
\end{defi}
We will be interested in the following statistics of multiset permutations.  
\begin{defi}
  \label{def:desc}
  Let $\sigma=w_1w_2\dots w_s$ be a permutation of a multiset
  of type $(f_1,f_2,\dots,f_n)$ where $s=f_1+f_2+\dotsm+ f_n$.
  An index $1\leq i\leq s-1$ is called a
  \begin{enumerate}[label=\rm(\roman*)]
   \item 
    \emph{descent} (or \emph{drop} or
    \emph{fall}) if $w_i>w_{i+1}$. 
   \item 
    \emph{plateau} (or \emph{level}) if $w_i=w_{i+1}$.
   \item 
    \emph{ascent} (or \emph{rise}) if $w_i<w_{i+1}$.
  \end{enumerate}
  We denote these sets by
  \begin{align*}
    \Des(\sigma)&=\Des((w_i)_{i=1}^s)=\{j\in [s-1]\mid w_j > w_{j+1} \},\\
    \Plat(\sigma)&=\Plat((w_i)_{i=1}^s)=\{j\in [s-1]\mid w_j = w_{j+1} \},\\
    \Asc(\sigma)&=\Asc((w_i)_{i=1}^s)=\{j\in [s-1]\mid w_j < w_{j+1} \} 
  \end{align*} 
  and
  the respective cardinalities, i.e., 
  the number of descents (ascents, plateaux respectively) of the multiset
  permutation $\sigma$,
  by
  \begin{align*}
  \des(\sigma)=\des((w_i)_{i=1}^s)&=\abs{\Des((w_i)_{i=1}^s)},\\
  \plat(\sigma)=\plat((w_i)_{i=1}^s)&=\abs{\Plat((w_i)_{i=1}^s)}, \\
  \asc(\sigma)=\asc((w_i)_{i=1}^s)&=\abs{\Asc((w_i)_{i=1}^s)}.
  \end{align*}
\end{defi}

\begin{remark} \
  \begin{enumerate}[label=\arabic*.,leftmargin=1cm]
\item Note that some authors also count $i=s$ as a descent and  $i=0$ as an ascent.
   \item Counting permutations by descents and ascents is a classic subject in
    combinatorics. In the case of descents of multiset permutations this is also
    known as \emph{Simon Newcomb's problem}
    \cite{Riordan:1958:combinatorial,DillonRoselle:1969:simon}.
  \end{enumerate}
 \item  Clearly every
  element except the last is either a descent, an ascent, or a plateau,
  and therefore
  \begin{equation}
    \label{eq:desplaasc+1}
    \abs{\sigma} = \des(\sigma)+\plat(\sigma)+\asc(\sigma) + 1  
    .
\end{equation}

\end{remark}

\begin{defi}
  \label{def:DFR} 
  To any pair of ordered set partitions $\tau,\eta\in\OSP_n$ such that
  $\bar{\tau} \leq \bar{\eta}$, 
  we associate a multiset on $[\abs{\eta}]$ by setting 
  the multiplicity of $k\in \{1,2,\dots,\abs{\eta}\}$ to be equal to the
  number of blocks of $\tau$ contained in the $k$-th block of $\eta$.
  Replacing every block of $\tau$ by the label of the block of $\eta$
  containing it we obtain a permutation of this multiset.
  We denote by $\des_\eta(\tau)$, $\asc_\eta(\tau)$ and $\plat_\eta(\tau)$ 
  its statistics as defined in Definition~\ref{def:desc}.
More precisely,
  let $\tau,\eta\in\OSP_n$ such that  $\bar{\tau} \leq \bar{\eta}$. If $\eta=
  (E_1,E_2,\ldots,E_e)$, then the blocks of  $\bar{\tau}$ can be arranged
  $$
  \bar{\tau} = \{ E_{1,1},E_{1,2},\ldots, E_{1,l_1}, \ldots, E_{e,1},\ldots, E_{e,l_e}\}, 
  $$ 
  where each $E_i = \bigcup_{j=1}^{l_i} E_{i,j}$ is a disjoint union taken in
  canonical order of the subsets (i.e., sorted according to their minimal
  elements). So $\tau$ can be written as  
  $$
  \tau=(E_{m_1,n_1},E_{m_2,n_2}, \ldots, E_{m_s, n_s}), 
  $$
  which is a permutation of the blocks of $\bar{\tau}$. Then we denote by
  $\Des_\eta(\tau), \Plat_\eta(\tau)$ and $\Asc_\eta(\tau)$ respectively, 
  the sets
  \begin{align*}
    \Des_\eta(\tau) &=\Des((m_i)_{i=1}^s), \\
    \Plat_\eta(\tau) &=   \Plat((m_i)_{i=1}^s), \\
    \Asc_\eta(\tau)  &= \Asc((m_i)_{i=1}^s)
  \end{align*}
  and by $\des_\eta(\tau)$,
  $\plat_\eta(\tau)$ and $\asc_\eta(\tau)$ the respective cardinalities.
\end{defi}

\begin{exa}
  Consider the partitions $\eta=(E_1,E_2,E_3,E_4,E_5)\in \OSP_{10}$ with blocks
  $E_1=\{5,8\}$, $E_2=\{9,10\}$, $E_3=\{3,6\}$, $E_4=\{1,2,4\}$, $E_5=\{7\}$,
  and 
  $\tau=(T_1, T_2, T_3,T_4,T_5,T_6,T_7,T_8)\in\OSP_{10}$
  with blocks
  $T_1=\{3\}$, $T_2=\{6\}$, $T_3=\{1,4\}$, $T_4=\{7\}$, $T_5=\{10\}$, $T_6=\{5,8\}$,
  $T_7=\{9\}$, $T_8=\{2\}$. Then $\bar{\tau}\leq \bar{\eta}$. Now $T_1\subseteq
  E_3$ and thus $w_1=3$,  $T_2 \subseteq E_3$ and thus $w_2=3$, etc.; the
  multiset permutation thus induced on $\tau$ by $\eta$ is
  $$
  (w_i)_{i=1}^{8}=(3,3,4,5,2,1,2,4), 
  $$
  and the statistics are $\des((w_i)_{i=1}^8)=2$, $\plat((w_i)_{i=1}^8)=1$ and  $\asc((w_i)_{i=1}^8)=4$. 
\end{exa}

\begin{defi} Let $\sigma=(w_i)_{i=1}^s$ be a multiset permutation. An \emph{ascending
    (\emph{resp.}~descending) run}
  is a maximal contiguous subsequence $(w_i)_{i=l}^k$ which is strictly
  increasing (resp. decreasing). 
  Similarly, \emph{a level run} is a maximal subsequence $(w_i)_{i=j}^k$ such that
  $w_j=w_{j+1}=\cdots=w_k$. 
\end{defi} 
\begin{prop}

  Let $\sigma=(w_i)_{i=1}^s$ be a multiset permutation.   
  \begin{enumerate}[label=\rm(\roman*),leftmargin=1cm]
   \item $\sigma$ can be decomposed uniquely into ascending runs
    separated by descents and plateaux. The number
    of ascending runs is equal to $s-\asc(\sigma)$. 
   \item $\sigma$ can be decomposed uniquely into descending runs
    separated by ascents and plateaux. The number
    of descending runs is equal to $s-\des(\sigma)$. 
   \item $\sigma$ can be decomposed uniquely into level runs
    separated by ascents and descents. The number
    of level runs is equal to $s-\plat(\sigma)$. 
  \end{enumerate}
\end{prop}
\begin{proof}
  We only prove (i).
  The claim is clearly true when $\asc\sigma = s-1$, i.e., when the sequence is
  monotone increasing. Otherwise replacing any ascent by a descent or plateau
  splits an ascending run into two, i.e., increases the number of ascending
  runs by one.
\end{proof}
\begin{remark} 
Suppose $\bar{\tau}=\bar{\eta}$, then
the blocks of  $\tau$ are a permutation of the blocks of $\eta$.
More precisely, if $\eta=(E_1,E_2,\dots,E_e)$ then
there is a permutation $h\in \SG_e$ 
such that $\tau=h(\eta)=(E_{h(1)},E_{h(2)}, \dots,E_{h(e)})$.
Then $\des_\eta(\tau)$ and 
$\asc_\eta(\tau)$ coincide with $\des(h)$ and $\asc(h)$,
where the latter quantities are the numbers of descents and ascents of the
permutation $h$, respectively. See 
\cite{Bona:2012:combinatorics,Brenti:1993:permutation,Reutenauer:1993:freeLiealg}
and \cite[p.~25]{Stanley:1986:enumerative1}
for the uses of  $\des(h)$ and
$\asc(h)$ in the context of symmetric groups.
\end{remark}
\begin{exa}
  The  statistics of the sequence $(m_i)_{i=1}^9=(1,1,3,5,5,5,4,1,4)$
  are as follows:
  $\des((m_i)_{i=1}^9)=2$, $\plat((m_i)_{i=1}^9)=3$ and 
$\asc((m_i)_{i=1}^9)=3$. 
The decomposition into ascending runs is given by 
$$
(1), (1,3,5), (5), (5), (4), (1,4), 
$$
the decomposition into level runs is 
$$
(1,1), (3), (5,5,5), (4), (1), (4)
$$
and the decomposition into descending runs is 
$$
(1),(1),(3),(5),(5),(5,4,1),(4).
$$
\end{exa}

\begin{lem}\label{lem7}
Let $\eta, \tau \in\OSP_n$ such that $\bar{\tau} \leq \bar{\eta}$, 
then there is an ordered set partition $\sigma_{\mathrm{max}}^{\mathrm{asc}}(\tau,\eta)$ such that 
$$
\{\sigma \in\OSP_n \mid \sigma \curlywedge \eta =\tau \} =[\tau, \sigma_{\mathrm{max}}^{\mathrm{asc}}(\tau,\eta)]. 
$$
and the
restriction of the mapping $\Psi$ in Proposition~\ref{app:prop:OSPintervals}
establishes a poset isomorphism
$$
\{\sigma \in\OSP_n \mid \sigma \curlywedge \eta =\tau \}  \to \IP_{p_1} \times \IP_{p_2} \times \cdots \times \IP_{p_t}, 
$$
where $1 \leq p_i \leq n$ are the lengths of the ascending runs of the sequence $(m_i)_{i=1}^s$ from Definition~\ref{def:DFR} and $t=\abs{\tau}-\asc_\eta(\tau)$.  
\end{lem}
\begin{proof}
  Let $\eta=(E_1,E_2,\ldots,E_e)$ and write
  $\tau=(E_{m_1,n_1},E_{m_2,n_2}, \ldots, E_{m_s, n_s})$
  as in Definition~\ref{def:DFR}.
  From Proposition~\ref{app:prop:quasimeet} we infer that  
  any partition $\sigma\in\OSP_n$  such that $\sigma\curlywedge\eta=\tau$ 
  satisfies $\sigma \geq \tau$.
   From Proposition~\ref{app:prop:OSPintervals} we infer that there is an interval partition 
$\lambda=(L_1,L_2,\ldots, L_l) \in \IP_{s}$ such that 
$$
\sigma=\biggl(\bigcup_{i\in L_1} E_{m_i,n_i},\bigcup_{i\in L_2} E_{m_i,n_i}, \ldots, \bigcup_{i\in L_l} E_{m_i,n_i}\biggr). 
$$ 
In order that $ \sigma \curlywedge \eta =\tau$ it is necessary and sufficient
that every block $L\in\lambda$, say $L=\{a+1, a+2, \ldots,a+b\}$, 
induces a strictly increasing sequence $m_{a+1} < m_{a+2}< \cdots <m_{a+b}$.

Let $(m_i)_{i=1}^{i_1}, (m_i)_{i=i_1+1}^{i_2},\ldots,
(m_i)_{i=i_{t-1}+1}^s$ be the decomposition of $(m_i)_{i=1}^s$ into ascending
runs. This decomposition defines the interval blocks $A_j:= \{i_{j-1}+1, i_{j-1}+2,\ldots,
i_j\}$ ($i_0=0,i_t=s$) and hence defines an interval partition $\alpha =
(A_1,A_2,\ldots,A_t)$. The interval partition $\lambda$ consists of
increasing intervals and therefore is finer than $\alpha$.
Thus we have an isomorphism
$$
\{\sigma \in\OSP_n \mid \sigma \curlywedge \eta =\tau \}  \to \IP_{A_1} \times\IP_{A_2} \times \cdots \times \IP_{A_t}
$$
via the restriction of the map $\Psi\restr_{[\tau,\hat{1}_n]}$ from Proposition~\ref{app:prop:OSPintervals}. The number $t=\abs{\alpha}$ is equal to $\abs{\tau}-\asc_\eta(\tau)$, the integers $p_j$ are the cardinalities of $A_j$ and $\sigma_{\mathrm{max}}^{\mathrm{asc}}(\tau,\eta)$ is the ordered set partition corresponding to $\lambda =\alpha$.   
\end{proof}

With these preparations we can now combinatorially evaluates the Weisner and
Goldberg coefficients from Definition~\ref{defi:weisnergoldberg}.

\begin{prop}[Weisner coefficients] \label{prop:Weisner}     \  
\begin{enumerate}[label=\rm(\roman*),leftmargin=1cm]
\item\label{it:Weis1}
  For $\tau,\eta \in \OSP_n$ the Weisner coefficient
  \eqref{eq:defi:weisner} is 
 \begin{equation}
   \label{eq:weisnercoeff}
w(\tau,\eta)
= 
\begin{cases}\displaystyle
\int_{-1}^0 x^{\abs{\tau} -\asc_\eta(\tau)-1}(1+x)^{\asc_\eta(\tau)}\,dx=\frac{(-1)^{\abs{\tau}-\asc_\eta(\tau)-1}}{\abs{\tau}\binom{\abs{\tau}-1}{\asc_\eta(\tau)}}, &\bar{\tau} \leq \bar{\eta},\\
0, & \bar{\tau} \not\leq \bar{\eta}. 
\end{cases}
\end{equation}
\item\label{it:Weis2}  For $\tau,\eta,\pi \in \OSP_n$ the partitioned Weisner
 coefficient     \eqref{eq:defi:partitionedweisner}  is
 \begin{equation*}
  w(\tau,\eta,\pi)
= 
\begin{cases}\displaystyle
\prod_{P\in\pi} w(\tau\restr_P,\eta\restr_P), &\bar{\tau} \leq \bar{\eta},~\tau \leq \pi,\\
0, & \text{otherwise}. 
\end{cases}
 \end{equation*}
\end{enumerate}
\end{prop}
\begin{proof}
\ref{it:Weis1}\,\, If $\bar{\tau} \not\leq \bar{\eta}$, then there is no
$\sigma$ such that $\sigma \curlywedge \eta=\tau$ and the sum is empty. 
Let us therefore assume henceforth that $\bar{\tau} \leq \bar{\eta}$. 
We take up the end of the proof of Lemma~\ref{lem7} where we
established the poset isomorphism
\begin{align*}
\{\sigma \in\OSP_n \mid \sigma \curlywedge \eta=\tau \} &
 \cong\IP_{p_1} \times \IP_{p_2} \times \cdots \times \IP_{p_t} \\
& \cong \Bool_{p_1-1} \times \Bool_{p_2-1} \times \cdots \times \Bool_{p_t-1} \\
& \cong \Bool_{p_1+\cdots+p_t   -t}, 
\end{align*}
where the second isomorphism follows from Proposition~\ref{app:prop:interval-Boolean} and $p_i$ denotes the length of the $i$-th ascending run.
The latter contains $p_i-1$ rises and therefore the total number of
ascents is 
$\asc_\eta(\tau) = p_1+p_2+\cdots +p_t-t$.
In the identification above, a partition $\sigma$ is mapped to a subset 
$A\subseteq [p_1+p_2+\cdots+p_t-t]$ 
with  $\abs{A} = \abs{\sigma}-t$ elements and we have
\begin{equation*}
\widetilde{\mu}(\sigma,\hat{1}_n)= \frac{(-1)^{\abs{\sigma}-1}}{\abs{\sigma}} = \frac{(-1)^{\abs{A}+t-1}}{\abs{A}+t}. 
\end{equation*}
Performing the sum we obtain
\begin{align*}
\sum_{\substack{\sigma \in \OSP_n \\ \sigma \curlywedge \eta=\tau}}\widetilde{\mu}(\sigma,\hat{1}_n)
&=  \sum_{A \subseteq [\asc_\eta(\tau)]}\frac{(-1)^{\abs{A}+t-1}}{\abs{A}+t}\\
&= \sum_{k=0}^{\asc_\eta(\tau)}\binom{\asc_\eta(\tau)}{k} \frac{(-1)^{k+t-1}}{k+t} \\
&= \int_{-1}^0 \sum_{k=0}^{\asc_\eta(\tau)}\binom{\asc_\eta(\tau)}{k} x^{k+t-1}\,dx \\
&= \int_{-1}^0 x^{t-1} (1+x)^{\asc_\eta(\tau)}\,dx \\
&= (-1)^{t-1}B(t, \asc_\eta(\tau)+1),  
\end{align*}
where $B$ is the beta function $B(a,b)= \frac{\Gamma(a )\Gamma(b )}{\Gamma(a+b)}$
which can be written in terms of binomial coefficients as desired.

\itemspacing
\ref{it:Weis2}\,\, In order for the set $\{\sigma\in\OSP_n \mid \sigma \leq \pi,~ \sigma
\curlywedge \eta=\tau\}$ to be nonempty, 
it is necessary that $\tau \leq \pi$ and $\bar{\tau}\leq \bar{\eta}$. We adopt
the notations from Definition~\ref{def:DFR}
and infer from Proposition~\ref{app:prop:OSPintervals} that there exists an
interval partition $\rho=(R_1,R_2,\dots,R_p) \in \IP_{s}$ such that
\begin{equation*}
\pi = (P_1,P_2,\dots,P_p)= \biggl(\bigcup_{i\in R_1} E_{m_i,n_i},\bigcup_{i\in R_2} E_{m_i,n_i}, \dots, \bigcup_{i\in R_p} E_{m_i,n_i}\biggr). 
\end{equation*}
If follows from Lemma~\ref{lem7} that $\sigma$ belongs to $[\tau,
\sigma_{\max{}}^{\mathrm{asc}}(\tau,\eta)]$.
In addition, $\sigma$ must satisfy $\sigma \leq \pi$. Hence, the ascending runs considered in Lemma~\ref{lem7} are split by the blocks of $\pi$.
More precisely, for each $k\in[p]$ we decompose $(m_i)_{i\in R_k}$ into
ascending runs, which give rise to the interval partition
$\gamma_k=(G_{k,1},\dots, G_{k,u_k}) \in \IP_{R_k}$ where $G_{k,j}$ 
consists of the indices $i$ of the $j$-th ascending run of $(m_i)_{i\in R_k}$. 
Then
\begin{equation*}
\sigma_{\max{}}^{\mathrm{asc}}(\tau\restr_{P_k},\eta\restr_{P_k})= \biggl(\bigcup_{i\in G_{k,1}} E_{m_i,n_i},\bigcup_{i\in G_{k,2}} E_{m_i,n_i},\dots, \bigcup_{i\in G_{k,u_k}} E_{m_i,n_i}\biggr)
\end{equation*} 
and for each $P\in\pi$, we pick an arbitrary $\sigma_P \in\OSP_P$ from the interval
$[\tau\restr_P, \sigma_{\max{}}^{\mathrm{asc}}(\tau\restr_P,\eta\restr_P)]$,
concatenate them and obtain 
$\sigma=\sigma_{P_1} \sigma_{P_2}\cdots\sigma_{P_p} \in \OSP_n$.
Since $\widetilde{\mu}(\sigma,\pi)$ is the product of
$\frac{(-1)^{\sharp(\sigma\restr_P)-1}}{\sharp(\sigma\restr_P)}$ over
$P\in\pi$, the conclusion follows.  
\end{proof}
Examples of Weisner coefficients will be given in Example~\ref{exa:W}.

Observe that $w(\tau,\eta,\pi) \neq0$ for $\bar{\tau} \leq \bar{\eta}, \tau \leq \pi$
and since in general the expectation values $\varphi_\tau(X_1,X_2,\dots,X_n)$ 
among different $\tau$ do not cancel each other,  we cannot expect the vanishing of cumulants without further assumptions. 
However we can express cumulants with independent entries also in terms of cumulants of
lower orders. In this case it turns out that the coefficients are determined by the number of plateaux.

\begin{lem}\label{lem8}
Let $\eta, \tau \in\OSP_n$ such that $\bar{\tau} \leq \bar{\eta}$, then the
restriction of the map $\Psi\restr_{[\tau,\hat{1}_n]}$ from
Proposition~\ref{app:prop:OSPintervals} establishes a poset isomorphism
\begin{equation}\label{iso flat}
\{\sigma \in\OSP_n \mid \sigma \geq \tau, \bar{\sigma} \leq \bar{\eta} \}  \to
\IP_{q_1} \times  \IP_{q_2} \times \cdots \times \IP_{q_r},  
\end{equation}
where $1 \leq q_i \leq n$ are the lengths of the level runs and
the number $r$ is equal to $\abs{\tau}-\plat_\eta(\tau)$. 
In particular, there is an ordered set partition $\sigma_{\max{}}^{\mathrm{pla}}(\tau,\eta)$ such that 
$$
\{\sigma \in\OSP_n \mid \sigma \geq \tau, \bar{\sigma} \leq \bar{\eta} \}  =[\tau, \sigma_{\max{}}^{\mathrm{pla}}(\tau,\eta)]. 
$$
\end{lem}
\begin{proof}
The proof is similar to that of Lemma~\ref{lem7}. 
Write $\tau=(E_{m_1,n_1},E_{m_2,n_2}, \ldots, E_{m_s, n_s})$ and
$\eta=(E_1,E_2,\ldots,E_e)$ as in Definition~\ref{def:DFR}. 
Let $\sigma \geq \tau$, then from Proposition~\ref{app:prop:OSPintervals} we infer
that there is an interval partition 
$\lambda=(L_1,L_2,\ldots, L_l) \in \IP_{s}$ such that 
\begin{equation*}
\sigma=\biggl(\bigcup_{i\in L_1} E_{m_i,n_i},\bigcup_{i\in L_2} E_{m_i,n_i},  \ldots, \bigcup_{i\in L_l} E_{m_i,n_i}\biggr). 
\end{equation*}

Let $(m_i)_{i=1}^{i_1}, (m_i)_{i=i_1+1}^{i_2},\ldots, (m_i)_{i=i_{r-1}+1}^s$ be
the decomposition of $(m_i)_{i=1}^s$ into level runs. This decomposition
determines interval blocks $B_j:= (i_{j-1}+1,i_{j-1}+2,\ldots, i_j)$
($i_0=0,i_r=s$) and hence gives rise to an interval partition $\beta =
(B_1,B_2,\ldots,B_r)$. In order that $\bar{\sigma} \leq \bar{\eta}$, each $L_i$
connects only plateaux, which is equivalent to the condition that $\lambda\leq
\beta$.
Denoting by $q_j=\abs{B_j}=i_j-i_{j-1}$, we get the isomorphism \eqref{iso flat}.  The ordered
set partition $\sigma_{\max}^{\plat}(\tau,\eta)$ corresponds to the 
choice $\lambda=\beta$. 
\end{proof}

\begin{prop}[Goldberg coefficients]\label{prop:goldberg}    \
\begin{enumerate}[label=\rm(\roman*),leftmargin=1cm]
\item\label{cum-cum1} For $\tau,\eta \in \OSP_n$ the Goldberg coefficient
 \eqref{eq:defi:goldberg} evaluates to
\begin{equation*}
  \gold(\tau,\eta)
= 
\begin{cases}\displaystyle
\frac{1}{q_1! q_2! \cdots q_r!}\int_{-1}^0 x^{\des_\eta(\tau)}(1+x)^{\asc_\eta(\tau)}\prod_{j=1}^r P_{q_j}(x)\,dx, &\bar{\tau} \leq \bar{\eta},\\
0, & \bar{\tau} \not\leq \bar{\eta}, 
\end{cases}
\end{equation*}
where $r, q_1,q_2,\dots, q_r$ are the integers from Lemma~\ref{lem8},
$$
P_q(x)=\sum_{k=1}^q k! \, S(q,k)\, x^{k-1} = E_q(x,x+1),
$$
 $S(q,k)$ are the Stirling numbers of the second kind
and
$$
E_n(x,y)=\sum_{\sigma\in\SG_n}x^{\des\sigma}y^{\asc\sigma}
$$
are the
\emph{homogeneous Eulerian polynomials}
\cite[p.~62]{Reutenauer:1993:freeLiealg}.
The first few polynomials are
\[
P_1(x)=1, \quad P_2(x)=2x +1,\quad P_3(x) = 6x^2 + 6x +1,\quad\dots
\]

\item\label{cum-cum2} 
 For $\tau,\eta,\pi \in \OSP_n$ the partitioned Goldberg coefficient
 coefficient     \eqref{eq:defi:partitionedgoldberg}  is
 \begin{equation*}
   \gold(\tau,\eta,\pi)
= 
   \begin{cases}\displaystyle
     \prod_{P\in\pi} \gold(\tau\restr_P,\eta\restr_P), &\bar{\tau} \leq \bar{\eta},~\tau \leq \pi,\\
     0, & \text{otherwise}. 
   \end{cases}
 \end{equation*}
\end{enumerate}
\end{prop}

\begin{remark}
  The name ``Goldberg coefficients'' originates from the
  Campbell-Baker-Hausdorff formula, see Section~\ref{sec:freeLiealg} below,
  in particular Theorem~\ref{thm:coefficient}.
  Some examples of $g(\tau,\eta)$ will be computed in Example~\ref{exa:GB}. 
\end{remark}

\begin{remark}
  \label{rem:Frobenius}
  The expansion of $P_q(x)$ in terms of Stirling coefficients was first proved
  by Frobenius \cite[Theorem~E, p.~244]{Comtet:1974:advanced}.
\end{remark}

\begin{proof}\ref{cum-cum1}\,\, 
If $\bar{\tau} \not\leq \bar{\eta}$, then $w(\sigma,\eta)=0$ for all $\sigma \geq \tau$ and so $\gold(\tau,\eta)=0$. Assume hereafter that $\bar{\tau} \leq \bar{\eta}$. 
From Lemma~\ref{lem8} we have the isomorphism 
\begin{align*}
\{\sigma \in\OSP_n \mid \sigma \geq \tau, \bar{\sigma} \leq \bar{\eta} \}
&\cong \IP_{q_1} \times \IP_{q_2} \times \cdots \times \IP_{q_r}, \\
\sigma &\mapsto (\sigma_1,\sigma_2,\ldots,\sigma_r). 
\end{align*}
Since $r=\abs{\tau}-\plat_\eta(\tau)=\des_\eta(\tau)+\asc_\eta(\tau)+1$, we
have $\abs{\sigma}-\asc_\eta(\tau)-1 = \sum_{i=1}^r
(\abs{\sigma_i}-1)+\des_\eta(\tau)$. Note also that $[\tau:\sigma]!
=\prod_{i=1}^r \prod_{S\in\sigma_i}\abs{S}!$.  Since $\sigma$ just connects the
blocks of a level run of $(m_i)_{i=1}^{\abs{\tau}}$, it does not change the
number of ascents: $\asc_\eta(\sigma)=\asc_\eta(\tau)$ and we have 
\begin{multline*}
\sum_{\substack{\sigma \in \OSP_n \\ \sigma \geq \tau}}
  \widetilde{\zeta}(\tau,\sigma)
  \,
  w(\sigma,\eta)
  \\
  \begin{aligned}[t]
&=\sum_{\substack{\sigma \in \OSP_n \\ \sigma \geq \tau, \bar{\sigma} \leq \bar{\eta}}}\frac{1}{[\tau:\sigma]!} \int_{-1}^0 x^{\abs{\sigma} -\asc_\eta(\tau)-1}(1+x)^{\asc_\eta(\tau)}\,dx \\
&=  \sum_{(\sigma_1,\sigma_2,\ldots,\sigma_r) \in\IP_{q_1} \times\IP_{q_2} \times \cdots \times \IP_{q_r}} \int_{-1}^0 x^{\des_\eta(\tau)}(1+x)^{\asc_\eta(\tau)}\prod_{i=1}^r\frac{x^{\abs{\sigma_i}-1}}{\prod_{S\in\sigma_i }\abs{S}!} dx \\ 
&=   \int_{-1}^0 x^{\des_\eta(\tau)}(1+x)^{\asc_\eta(\tau)}\prod_{i=1}^r\left(\sum_{\rho \in \IP_{q_i}}\frac{x^{\abs{\rho}-1}}{\prod_{R\in\rho}\abs{R}!}\right) dx.  
  \end{aligned}
\end{multline*}
For $q\in\N$ we have  
\begin{equation*}
\sum_{\rho \in \IP_q}\frac{x^{\abs{\rho}-1}}{\prod_{R\in\rho}\abs{R}!} = \sum_{\substack{n_1+\cdots+n_k=q\\ n_i\geq1, k\geq1}}\frac{x^{k-1}}{n_1! \cdots n_k!}. 
\end{equation*}
For each fixed $k$, the sum $\sum_{\substack{n_1+\cdots+n_k=q\\
    n_i\geq1}}\frac{q!}{n_1! \cdots n_k!}$ is the number of ways of
distributing $q$ distinct objects among $k$ nonempty urns, so it equals  $k!
S(q,k)$ and the proof is complete.

\itemspacing
\ref{cum-cum2}\,\, The idea of the proof is similar to Proposition~\ref{prop:Weisner}\ref{it:Weis2} and we omit the proof. 
\end{proof}

Some Goldberg coefficients $\gold(\tau,\eta)$ are known to vanish even when
$\bar{\tau} \leq \bar{\eta}$ (see \cite{BV15}), while the Weisner coefficients
$w(\tau,\eta)$ do not vanish whenever $\bar{\tau} \leq \bar{\eta}$. 
The following Proposition describes two sufficient criteria for vanishing Goldberg
coefficients .

\begin{prop}\label{prop:vanish} Suppose $\tau,\eta\in\OSP_n$ are such that $\bar{\tau}\leq \bar{\eta}$. 
\begin{enumerate}[label=\rm(\roman*),leftmargin=1cm]
\item\label{vanish g} If $\des_\eta(\tau)=\asc_\eta(\tau)$ and $\abs{\tau}$ is even then $\gold(\tau,\eta)=0$. 
\item\label{non vanish g} If $|\tau|$ is prime then $g(\tau,\eta) \neq0$. 
\end{enumerate}
\end{prop}
\begin{proof}
\ref{vanish g} By \cite[Corollary 3.15]{Reutenauer:1993:freeLiealg}, the coefficient $\gold(\tau,\eta)$ of $K_\tau(X_1,\dots,X_n)$ vanishes if $q_1+\cdots +q_r +\des_\eta(\tau)+ \asc_\eta(\tau)-r$ is odd and $\des_\eta(\tau)=\asc_\eta(\tau)$. Since $\sum_{i=1}^r q_i = r +\plat_\eta(\tau)$ and $\abs{\tau}= \des_\eta(\tau)+\asc_\eta(\tau)+\plat_\eta(\tau)+1$, the conclusion follows. 

\itemspacing
\ref{non vanish g} The idea of the proof is taken from \cite{BV15}. Let $p:=\abs{\tau}$. By definition and Proposition~\ref{prop:Weisner} we have 
\begin{align*}
\gold(\tau,\eta)
&=\sum_{\substack{\sigma \in \OSP_n \\ \sigma \geq \tau}}
   \widetilde{\zeta}(\tau,\sigma)\,w(\sigma,\eta) 
 =\sum_{\substack{\sigma \in \OSP_n \\ \sigma \geq \tau}}
    \frac{1}{[\tau:\sigma]!} \frac{(-1)^{\abs{\sigma}-\asc_\eta(\sigma)-1}}{\abs{\sigma}\binom{\abs{\sigma}-1}{\asc_\eta(\sigma)}} \\
&=  \frac{(-1)^{p-\asc_\eta(\tau)-1}}{p\binom{p-1}{\asc_\eta(\tau)}} 
  + \sum_{\substack{\sigma \in \OSP_n \\ \sigma > \tau}} 
    \frac{1}{[\tau:\sigma]!} 
    \frac{(-1)^{\abs{\sigma}-\asc_\eta(\sigma)-1}}{\abs{\sigma}\binom{\abs{\sigma}-1}{\asc_\eta(\sigma)}}. 
\end{align*}
Since the number $[\tau:\sigma]!
\abs{\sigma}\binom{\abs{\sigma}-1}{\asc_\eta(\sigma)}$ never contains $p$ as a
factor for any $\sigma >\tau$,  $g(\tau,\eta)$ is nonzero.  
\end{proof}

\begin{remark}
  It is a difficult problem to characterize vanishing Goldberg coefficients.  
  The criterion  \ref{vanish g} from Proposition~\ref{prop:vanish} above
  does not cover all cases,
  for example one can show that $\gold(\tau,\eta)=0$ for
  $\tau=(\{3\},\{4\},\{2\},\{1\}), \eta=(\{1,2,3\},\{4\})$,
  although the pair $(\tau,\eta)$ does not satisfy the assumption of the criterion.
  More information about Goldberg coefficients can be found in
  \cite{Reutenauer:1993:freeLiealg,Thompson82}
  and in particular  \cite[Section IV]{BV15} concerning the question of vanishing coefficients.
\end{remark}

\begin{exa}[Goldberg coefficients and partial vanishing of cumulants]\label{exa:GB} 
We will write the ordered kernel set partition $\kappa(i_1,\dots, i_n)$ simply
as $i_1i_2\dots i_n$. 
\begin{enumerate}[label=\arabic*.,leftmargin=1cm]
 \item 
  Take $\eta= 12$. We compute $g(12,12)$. Now $\tau=\eta$ so $m_1=1, m_2=2, r=2, q_1=q_2=1$. Hence $\des_\eta(\tau)=0, \asc_\eta(\tau)=1$ and so
\begin{equation*}
g(12,12)= \frac{1}{1!1!}\int_{-1}^0 (1+x) \,dx=\frac{1}{2}. 
\end{equation*}
Similarly  we get $g(21,12)=-\frac{1}{2}, g(12,21)=-\frac{1}{2},
g(21,21)=\frac{1}{2}$ and hence
\begin{align*}
&K_{11}(X^{(1)},Y^{(2)}) = \frac{1}{2}K_{12}(X,Y)-\frac{1}{2} K_{21}(X,Y), \\
& K_{11}(X^{(2)}, Y^{(1)}) =  -\frac{1}{2}K_{12}(X,Y)+\frac{1}{2} K_{21}(X,Y). 
\end{align*}

In the most popular spreadability systems, like the tensor, free, Boolean or
monotone spreadability systems, partitioned cumulants factorize, e.g., 
$K_{12}(X,Y)=K_1(X)K_1(Y)$. 
Hence we get $K_{11}(X^{(1)},Y^{(2)}) =K_{11}(X^{(2)},Y^{(1)})=0$. 

\item 
We then consider the case $\eta=112= (\{1,2\},\{3\})$. The Goldberg coefficients can be nonzero only when $\bar{\tau} \leq \bar{\eta}$, so $\tau$ is one of 
$$
112, 221, 123, 132, 213, 231, 312, 321.
$$ 
If $\tau=112$ then $g(112,112)=\frac{1}{2}$ by the same calculation as $g(12,12)$. If $\tau=221$ then again $g(221,112)= -\frac{1}{2}$. 
If $\tau=123$ then $m_1=1,m_2=1,m_3=2, r=2, q_1=2, q_2=1$. So 
\begin{equation*}
g(123,112) = \frac{1}{1!2!}\int_{-1}^0 (1+x) P_2(x) \, dx = \frac{1}{12}. 
\end{equation*}
If we take $\tau= 132$ then $m_1=1,m_2=2, m_3=1$, so $r=3,q_1=q_2=q_3=1$, $\asc = 1, \des=1$. Thus we get 
\begin{equation*}
g(132,112) = \frac{1}{1!1!1!}\int_{-1}^0 x (1+x) \, dx = -\frac{1}{6}. 
\end{equation*}
If we take $\tau= 231$ then $m_1=2,m_2=1,m_3=1$. So $r=2, q_1=1, q_2=2$, $\asc=0,\des=1$. Hence 
\begin{equation*}
g(231,112) = \frac{1}{1!2!}\int_{-1}^0 x P_2(x) \, dx = \frac{1}{2}\int_{-1}^0 x (2x+1)\, dx=\frac{1}{12}. 
\end{equation*}
Similarly we can compute the remaining Goldberg coefficients and get 
\begin{multline*}
  K_{111}(X^{(1)},Y^{(1)},Z^{(2)})
   = \frac{1}{2} K_{112} -\frac{1}{2} K_{221} + \frac{1}{12} K_{123} -\frac{1}{6} K_{132} \\
    +\frac{1}{12}K_{213}+\frac{1}{12} K_{231} -\frac{1}{6}K_{312} + \frac{1}{12}K_{321}, 
\end{multline*}
where $X,Y,Z$ are omitted for simplicity. We can see that \eqref{eq permute} holds (now $\pi=111$): 
\begin{equation*}
\frac{1}{2} - \frac{1}{2}=0, \qquad  \frac{1}{12} -  \frac{1}{6} + \frac{1}{12} + \frac{1}{12}  - \frac{1}{6} + \frac{1}{12} =0.  
\end{equation*}

Again if the cumulants factorize then $K_{112}=K_{221}=K_{11}(X,Y)K_1(Z)$ and
$K_{123}=\cdots= K_{321} = K_1(X)K_1(Y)K_1(Z)$, so the mixed cumulant
$K_{111}(X^{(1)},Y^{(1)},Z^{(2)})$ vanishes.  

Similarly one can compute $g(\tau, 121)$ for all $\tau$ such that $\bar{\tau} \leq \overline{121}$ and get 
\begin{multline*}
  K_{111}(X^{(1)},Y^{(2)},Z^{(1)})
  = \frac{1}{2} K_{121} -\frac{1}{2} K_{212} - \frac{1}{6} K_{123} +\frac{1}{12} K_{132} \\
+\frac{1}{12}K_{213}+\frac{1}{12} K_{231} +\frac{1}{12}K_{312} - \frac{1}{6}K_{321}. 
\end{multline*}

Now in the tensor, free or Boolean spreadability system the mixed cumulant
vanishes. However in the monotone spreadability system it does not:  $K_{212}$
vanishes identically since $212$ is not a monotone partition (see Proposition
\ref{prop19}) and therefore, in the monotone case we have 
\[
K_{111}(X^{(1)},Y^{(2)},Z^{(1)}) = \frac{1}{2}K_{121}(X,Y,Z)=\frac{1}{2} K_{11}(X,Z)K_1(Y)
\]
which does not vanish in general. 
\end{enumerate}

\end{exa}

The calculation of these cumulants in terms of moments is easier. 

\begin{exa}[Weisner coefficients and partial vanishing of cumulants]\label{exa:W} We can reuse some results from Example~\ref{exa:GB}. 

  \begin{enumerate}[label=\arabic*., leftmargin=1cm]
   \item 
    If we take $\tau=\eta=12$ then $\asc_\eta(\tau)=1$ and so
    \begin{equation*}
      w(12,12)= \frac{(-1)^{2-1-1}}{2 \binom{1}{1}}=\frac{1}{2}. 
    \end{equation*}
    
   \item 
    Similarly if $\tau=21$ and  $\eta=12$ then $m_1=2, m_1=1$ so $\asc_\eta(\tau)=0$. 
    Therefore we get $w(21,12)=-\frac{1}{2}$. Similarly, $w(12,21)=-\frac{1}{2}, w(21,21)=\frac{1}{2}$. So 
    \begin{align*}
      K_{11}(X^{(1)},Y^{(2)}) &= \frac{1}{2}\phi_{12}(X,Y)-\frac{1}{2} \phi_{21}(X,Y), \\
      K_{11}(X^{(2)}, Y^{(1)}) &=  -\frac{1}{2}\phi_{12}(X,Y)+\frac{1}{2} \phi_{21}(X,Y). 
    \end{align*}
    In factorizing spreadability systems we have
    $\phi_{12}(X,Y)=\phi_{21}(X,Y)=\phi_1(X)\,\phi_1(Y)$ and hence $K_{11}(X^{(1)},Y^{(2)}) =K_{11}(X^{(2)},Y^{(1)})=0$. 

\item 
One can show that $w(112,112)=\frac{1}{2} = - w(221,112)$ by the same calculation as $w(12,12)$ and $w(21,12)$. 
If $\tau=123,\eta=112$ then $m_1=1,m_2=1,m_3=2$, so $\asc =1$ and  
\begin{equation*}
w(123,112) = \frac{(-1)^{3-1-1}}{3\binom{2}{1}} = -\frac{1}{6}. 
\end{equation*}
If we take $\tau= 231$ then $m_1=2,m_2=1,m_3=1$ and so $\asc=0$. Hence 
\begin{equation*}
w(231,112) = \frac{(-1)^{3-0-1}}{3\binom{2}{0}}=\frac{1}{3}. 
\end{equation*}
Similarly we can compute the remaining Weisner coefficients and get 
\begin{multline*}
  K_{111}(X^{(1)},Y^{(1)},Z^{(2)})
  = \frac{1}{2} \phi_{112} -\frac{1}{2} \phi_{221} - \frac{1}{6} \phi_{123} -\frac{1}{6} \phi_{132} \\
-\frac{1}{6}\phi_{213}+\frac{1}{3} \phi_{231} -\frac{1}{6}\phi_{312} + \frac{1}{3}\phi_{321}, 
\end{multline*}
where $X,Y,Z$ are omitted for simplicity. We can see that \eqref{eq permute2} holds (now $\pi=111$): 
\begin{equation*}
\frac{1}{2} - \frac{1}{2}=0, \qquad  -\frac{1}{6} -  \frac{1}{6} - \frac{1}{6} + \frac{1}{3}  - \frac{1}{6} + \frac{1}{3} =0.  
\end{equation*}
In factorizing spreadability systems the mixed cumulant $K_{111}(X^{(1)},Y^{(1)},Z^{(2)})$ vanishes by using the factorization of partitioned moments. 

\item 
Similarly one can compute $w(\tau, 121)$ for all $\tau$ such that $\bar{\tau} \leq \overline{121}$ and get 
\begin{multline*}
  K_{111}(X^{(1)},Y^{(2)},Z^{(1)})
  = \frac{1}{2} \phi_{121} -\frac{1}{2} \phi_{212} - \frac{1}{6} \phi_{123} -\frac{1}{6} \phi_{132} \\
+\frac{1}{3}\phi_{213}-\frac{1}{6} \phi_{231} +\frac{1}{3}\phi_{312} - \frac{1}{6}\phi_{321}. 
\end{multline*}
In the tensor, free or Boolean spreadability system mixed cumulants
vanishes, but in the monotone case in general they don't:
indeed $\phi_{212}=\phi(X)\phi(Y)\phi(Z)$ while
$\phi_{121}=\phi(X Z)\phi(Y)$. Therefore, in the monotone case we have
\[
K_{111}(X^{(1)},Y^{(2)},Z^{(1)}) = \frac{1}{2}(\phi_{121}-\phi_{212})=\frac{1}{2}(\phi(X Z)- \phi(X)\phi(Z)) \phi(Y)
\]
which does not vanish in general. 
  \end{enumerate}

\end{exa}

\section{Campbell-Baker-Hausdorff formula and Lie polynomials}
\label{sec:freeLiealg}
The material of the preceding section resembles some results from the theory of
free Lie algebras, cf.~the book by C.~Reutenauer
\cite{Reutenauer:1993:freeLiealg} already cited above. 
In particular,
the  \emph{Goldberg coefficients}  $\gold(\tau,\eta)$ from 
Proposition~\ref{prop:goldberg} coincide with 
the coefficients of the Campbell-Baker-Hausdorff series, i.e., 
\begin{equation}
  \label{eq:CBH}
\log(e^{a_1}e^{a_2}\dotsm e^{a_n}) = \sum_{w: \text{word}} g_w w
\end{equation}
when it is expanded in the ring of noncommutative formal power series,
see \cite{Goldberg:1956:formal}.
In this section we will provide a new probabilistic interpretation of the
coefficients of the CBH formula in terms of a certain variant of the tensor spreadability system.

\subsection{The unshuffle spreadability system $\cS_{\NCtensorT}$}  
\label{sec NCT} 
Given a \emph{unital} algebra $\alg{A}$ 
we introduce an operator-valued spreadability system with the following ingredients. 
\begin{enumerate}[label=\arabic*.,leftmargin=1cm]
\item Put $\varphi=\Id: \mathcal{A}\to \mathcal{A}$.
\item  $\alg{U}:=\otimes_{i=1}^\infty \mathcal{A}$ is the algebraic tensor
 product, cf.~Section~\ref{ssec:tensor}.
\item $\iota^{(j)}: \mathcal{A}\to\alg{U}$ is the natural embedding of $\mathcal{A}$ into the $j^{\mathrm{th}}$ component of $\alg{U}$: 
\[
\iota^{(j)}(X):= 1^{\otimes (j-1)} \otimes X \otimes 1^{\otimes \infty}. 
\]
\item $\tilde{\varphi}:=\conc_\infty:  \alg{U}\to \mathcal{A}$ is the concatenation product.
 The resulting functional on $\alg{U}$ gives rise to a
rearrangement of the letters
 \begin{equation}
   \label{eq:unshuffle:phimult}
\tilde{\varphi}(X_1^{(i_1)} X_2^{(i_2)}\cdots X_n^{(i_n)}):= \phi_\pi(X_1,X_2,\dots,X_n)=X_{P_1} X_{P_2}\cdots X_{P_k}, 
 \end{equation}
where $\pi=\kappa(i_1,i_2,\ldots,i_n)=(P_1,P_2,\ldots,P_k)$ and
the order of the letters inside each block is preserved according to  \eqref{product}.
\end{enumerate}
Thus $X_1,X_2,\dots,X_n$ are ``unshuffled'' by $\tilde{\phi}$ accordingly to the upper indices $(i_1,i_2,\dots,i_n)$. For example 
$$
\tilde{\varphi}(X_1^{(5)}X_2^{(2)}X_3^{(3)}X_4^{(2)}X_5^{(3)}):= X_2 X_4 X_3 X_5 X_1. 
$$
\begin{prop} 
The triple $\cS_{\NCtensorT}=\cS_{\NCtensorT}(\alg A)=(\alg{U}, 
\conc_\infty,(\iota^{(j)})_{j=1}^\infty)$ defined above constitutes a spreadability system for $(\mathcal{A},\Id)$.
\end{prop}
\begin{proof} 
Clearly $\tilde{\varphi}\circ \iota^{(i)}=\Id$ on
$\mathcal{A}$. The symmetry condition (\ref{eq:spreadability}) holds too since the value of
$\tilde{\varphi}$ only depends on the ordered kernel set partition of the upper indices. 
\end{proof}

We call the triple $\cS_{\NCtensorT}$ the {\it unshuffle spreadability system}.
The corresponding  cumulants $K_\pi^{\NCtensorT}$
satisfy (ordered) multiplicativity like the tensor cumulants, cf.~Proposition~\ref{prop19} \ref{ten}.   

\begin{prop}[Multiplicativity of partitioned cumulants] 
For $\pi = (P_1,P_2,\dots, P_p) \in \OSP_n$, we have 
\[
K_\pi^{\NCtensorT}(X_1,X_2,\dots, X_n) = K_{|P_1|}^{\NCtensorT}(X_{P_1})\,  K_{|P_2|}^{\NCtensorT}(X_{P_2}) \dotsm  K_{|P_p|}^{\NCtensorT}(X_{P_p}). 
\]
\end{prop}
\begin{proof}
For $\pi=(P_1,P_2,\dots, P_p)\in\OSP_n$, 
\begin{multline*}
  \varphi_\pi(N.X_1,N.X_2, \dots, N.X_n) \\
  \begin{aligned}
&= \sum_{i_1, i_2,\dots, i_n \in [N]} \varphi_{\pi\curlywedge \kappa(i_1,i_2,\dots, i_n)}(X_1,X_2, \dots, X_n) \\
&= \sum_{i_1,i_2, \dots, i_n \in [N]}  X_{\kappa(i_1,i_2,\dots, i_n) \restr_{P_1}} X_{\kappa(i_1,i_2,\dots, i_n) \restr_{P_2}} \cdots X_{\kappa(i_1,i_2,\dots, i_n) \restr_{P_p}}\\
&= \biggl(
 \sum_{\substack{i_k \in [N] \\ k\in P_1}}  X_{\kappa(i_1,i_2,\dots, i_n)  \restr_{P_1}}
  \biggr)
  \biggl(
  \sum_{\substack{i_k \in [N] \\ k\in P_2}}  X_{\kappa(i_1,i_2,\dots, i_n)  \restr_{P_2}}
  \biggr)
  \,\cdots\,
  \biggl(
  \sum_{\substack{i_k \in [N] \\ k\in P_p}} X_{\kappa(i_1,i_2,\dots, i_n) \restr_{P_p}}
  \biggr)
\\
  &= \tilde{\varphi}\biggl(\prod_{i\in P_1}N.X_i\biggr)\,
     \tilde{\varphi}\biggl(\prod_{i\in P_2}N.X_i\biggr)\,
    \cdots\,
    \tilde{\varphi}\biggl(\prod_{i\in P_p}N.X_i\biggr). 
  \end{aligned}
\end{multline*}
We conclude by observing that the coefficient of $N^{p}$ is the product of the
linear coefficients of the factors.
\end{proof}

After having established multiplicativity it suffices to compute
$K_n^{\NCtensorT}$. 
Theorem~\ref{thm:mc} reads as follows. 

\begin{prop}\label{prop:mc-NCT}
 For $n \in \N$ we have 
\[
K_n^{\NCtensorT} (X_1,X_2,\dots, X_n) = \sum_{\pi =(P_1,P_2,\dots,P_k)\in \OSP_n} \frac{(-1)^{|\pi|-1}}{|\pi|}X_{P_1} X_{P_2}\cdots X_{P_k}.  
\]
\end{prop}

\begin{exa} The reader can easily verify that 
\begin{align*}
K_1^{\NCtensorT}(X)&=X, \\
K_2^{\NCtensorT}(X_1,X_2) &= \frac{1}{2}[X_1,X_2], \\
K_3^{\NCtensorT}(X_1,X_2,X_3) &= \frac{1}{3} (X_1 X_2 X_3+X_3 X_2 X_1) \\
&\qquad- \frac{1}{6} (X_1 X_3 X_2 + X_2 X_1 X_3 + X_2 X_3 X_1 + X_3 X_1 X_2).   \notag
\end{align*}
\end{exa}
\begin{remark}
  We will see in Section~\ref{ssec:freemonoid} that $K_n^{\NCtensorT}~(n\geq2)$ can be expressed
  as a sum of commutators, i.e., a Lie polynomial.
\end{remark}
We can now express the CBH formula
  \eqref{eq:CBH}
on $\alg{A}$  in terms of cumulants. 

\begin{thm}[CBH formula]\label{Thm:CBH-c}
  As formal power series on $\mathcal{A}$ we have the identity   
\begin{multline*}
\log(e^{a_1} e^{a_2}\cdots e^{a_n}) \\
= \sum
\frac{1}{p_1!p_2! \cdots p_n!}
    K^{\NCtensorT}_{p_1+p_2+\cdots+p_n}(\underbrace{a_{1}, a_{1},\ldots, a_{1}}_{\text{$p_1$ times}},
    \underbrace{a_{2}, a_{2},\ldots, a_{2}}_{\text{$p_2$ times}},
    \ldots,
    \underbrace{a_{n},a_n,\ldots,a_{n}}_{\text{$p_n$  times}}).
  \end{multline*}
  where the sum runs over all $n$-tuples   $(p_1,p_2,\dots,p_n) \in
  (\N\cup\{0\})^n$ 
with the exception of the tuple $(0,0,\dots,0)$.
\end{thm}
\begin{proof}
  First observe that the right hand side of the claimed identity
  is the coefficient of $N$ in the series
  \begin{multline}
    \label{eq:log1}
1+ \sum_{\substack{(p_1,p_2,\dots,p_n) \in (\N\cup
    \{0\})^n,\\(p_1,p_2,\dots,p_n) \neq (0,0,\dots,0)  }} \frac{1}{p_1!p_2!
  \cdots
  p_n!}\tilde{\varphi}((N.a_{1})^{p_1}(N.a_2)^{p_2}\cdots\,(N.a_{n})^{p_n})\\
\begin{aligned}
&= \tilde{\varphi}(e^{N.a_1}e^{N.a_2} \cdots \, e^{N.a_n})
\\
&=  \tilde{\varphi}
   (e^{a_1^{(1)}+a_1^{(2)}+\cdots + a_1^{(N)}}
   e^{a_2^{(1)}+a_2^{(2)}+\cdots + a_2^{(N)}}
   \cdots\, e^{a_n^{(1)}+a_n^{(2)}+\cdots + a_n^{(N)}}). 
\end{aligned}
 \end{multline}
Our construction implies that $a_i^{(j)}, j=1,2,3,\dots$ mutually commute for
different $j$ and so
$$
e^{a_i^{(1)}+a_i^{(2)}+\cdots + a_i^{(N)}} = e^{a_i^{(1)}} e^{a_i^{(2)}}\cdots
e^{a_i^{(N)}}
.
$$
Thus \eqref{eq:log1} can be factorized as follows:
\begin{multline*}
  \tilde{\varphi}(e^{a_1^{(1)}+a_1^{(2)}+\cdots + a_1^{(N)}}
  e^{a_2^{(1)}+a_2^{(2)}+\cdots + a_2^{(N)}}
  \cdots e^{a_n^{(1)}+a_n^{(2)}+\cdots + a_n^{(N)}})
  \\
=   \tilde{\varphi}((e^{a_1^{(1)}}e^{a_1^{(2)}}\cdots e^{a_1^{(N)}})( e^{a_2^{(1)}} e^{a_2^{(2)}}\cdots e^{a_2^{(N)}}) \cdots ( e^{a_n^{(1)}} e^{a_n^{(2)}}\cdots e^{a_n^{(N)}}) ). 
 \end{multline*}
 Now $\tilde{\varphi}$ simply rearranges the factors
 according to the upper index and
 the last expression equals 
 \begin{equation*}
(e^{a_1} e^{a_2}\cdots e^{a_n})^N. 
\end{equation*}
On the other hand,
\begin{align*}
  (e^{a_1} e^{a_2}\cdots e^{a_n})^N 
  &= e^{N\log(e^{a_1} e^{a_2}\cdots e^{a_n})} \\
  &= 1 + N \log(e^{a_1} e^{a_2}\cdots e^{a_n})
    + \frac{N^2}{2!}( \log(e^{a_1} e^{a_2}\cdots e^{a_n}))^2 + \dots
\end{align*}
and
we conclude by comparing the coefficient of $N$ in this series with the one in
\eqref{eq:log1}.

\end{proof}
\begin{remark}
Theorem~\ref{Thm:CBH-c} is similar to the well-known formula for the generating
function of multivariate cumulants from classical probability theory, 
\begin{multline*}
\log \IE[e^{z_1 X_1 +z_2 X_2 +\cdots +z_n X_n}] \\
=  \sum
\frac{z_1^{p_1}z_2^{p_2}\cdots z_n^{p_n}}{p_1!p_2! \cdots p_n!}
K^{\tensorT}_{p_1+p_2+\cdots+p_n}(
  \underbrace{X_{1}, X_{1},\ldots,  X_{1}}_{\text{$p_1$  times}},
  \underbrace{X_{2}, X_{2},\ldots,  X_{2}}_{\text{$p_2$  times}},
\ldots, \underbrace{X_{n},X_n,\ldots,X_{n}}_{\text{$p_n$ times}}), 
\end{multline*}
where $X_1,X_2,\dots,X_n$ are $\C$-valued classical random variables and $z_1,z_2,\dots, z_n$ are commuting indeterminates.  
\end{remark}

\begin{prop}\label{NCT independence}   \ 
\begin{enumerate}[label=\rm(\roman*),leftmargin=1cm]
\item\label{LieCum1} 
A sequence $(\mathcal{A}_i)_{i=1}^\infty$ of subalgebras of $\mathcal{A}$ is
$\cS_{\NCtensorT}$-independent if and only if the subalgebras $\mathcal{A}_1,
\mathcal{A}_2, \dots$ commute mutually. This is obviously equivalent to
$K_2^{\NCtensorT}(X,Y)=0$ whenever $X \in \mathcal{A}_i, Y\in \mathcal{A}_j$
with $i\neq j$.  

\item\label{LieCum2} 
For fixed $n \geq2$, if $\{X_1,X_2,\dots, X_n\} \subseteq \mathcal{A}$ splits into
two mutually commuting families then 
$$
K_n^{\NCtensorT}(X_1,X_2,\dots, X_n)=0
.
$$  
\end{enumerate}
\end{prop}

\goodbreak{}
\begin{remark} \label{rem:commutativity=vanishing}    \   
\begin{enumerate}[label=\arabic*.,leftmargin=1cm]
\item
    Additivity of Lie polynomials in commuting variables is well known,
    see \cite[p.~20]{Reutenauer:1993:freeLiealg}.
   \item 
    We have thus shown that
    \[
    \text{$\cS_{\NCtensorT}$-independence $\iff$ commutativity $\iff$ vanishing of mixed cumulants}.
    \]
    This example illustrates that in general vanishing of mixed cumulants does not imply exchangeability. 
\end{enumerate}
\end{remark}
\begin{proof}[Proof of Proposition \ref{NCT independence}] \ref{LieCum1}\,\, Suppose that $(\mathcal{A}_i)_{i=1}^\infty$ is $\cS_{\NCtensorT}$-independent. Let $\pi =\hat{1}_2, \rho = (\{2\},\{1\}) \in\OSP_2$ and let $X \in \mathcal{A}_i, Y \in \mathcal{A}_j$ for fixed $i > j$. Then $\kappa(i,j)=\rho$ and 
\begin{equation*}
\varphi_\pi(X,Y) = X Y, \qquad \varphi_{\pi\curlywedge \rho}(X,Y) = Y X, 
\end{equation*}
and by $\cS_{\NCtensorT}$-independence these two must coincide, so $X Y = Y
X$. This shows that $[\mathcal{A}_i,\mathcal{A}_j]=0$.  Conversely, if the
subalgebras $\mathcal{A}_1, \mathcal{A}_2, \dots$ mutually commute then
for any $i_1,i_2,\dots, i_n \in \N$, any $X_k \in \mathcal{A}_{i_k},
k=1,2,\dots, n$ and any $\pi=(P_1,P_2,\dots, P_p) \in\OSP_n$ we have
\begin{equation*}
\varphi_{\pi\curlywedge \rho}(X_1,X_2,\dots, X_n) = (X_{P_1 \cap R_1} X_{P_1 \cap R_2} \cdots X_{P_1 \cap R_r} )  \cdots (X_{P_p \cap R_1} X_{P_p \cap R_2} \cdots X_{P_p \cap R_r})
\end{equation*}
where $\rho=\kappa(i_1,i_2,\dots, i_n)=(R_1,R_2,\dots, R_r)$.
Here $X_\emptyset$ is understood as the unit $\hat{1}_\mathcal{A}$. 
Now if $\{X_k \mid k \in R_i\}$ and $\{X_k \mid k \in R_j\}$ mutually commute
for distinct $i,j$ then for each $i =1,2,\dots, p$ we have
\begin{equation*}
X_{P_i \cap R_1} X_{P_i \cap R_2} \cdots X_{P_i \cap R_r} = X_{P_i},  
\end{equation*}
which shows that $\varphi_{\pi\curlywedge \rho}(X_1,X_2,\dots, X_n) =\varphi_{\pi}(X_1,X_2\dots, X_n) $. 

\itemspacing
\ref{LieCum2}\,\, Suppose that $\{X_i \mid i\in I\}$ and $\{X_i \mid i\in I^c\}$ commute with
each other and $\emptyset \subsetneq I \subsetneq \{1,2,\dots,n\}$. Then, for
commuting indeterminates $z_1,\dots, z_n$, we have  
\begin{equation}\label{sum log}
 \log(e^{z_1 X_1} e^{z_2 X_2} \cdots e^{z_n X_n}) = \log\Bigl(\prod_{i \in I} e^{z_ i X_i}\Bigr) + \log\Bigl(\prod_{i \in I^c} e^{z_ iX_i}\Bigr), 
 \end{equation}
  where the products $\prod_{i \in I}, \prod_{i \in I^c}$ preserve the natural
  orders on $I,I^c$, respectively.   On the other hand, by Theorem~\ref{Thm:CBH-c}
  we have 
  \begin{multline}\label{CBH d}
    \log(e^{z_1 X_1}e^{z_2 X_2} \cdots e^{z_n X_n}) \\
    =
 \sum \frac{z_1^{p_1} z_2^{p_2} \cdots z_n^{p_n}}{p_1!p_2! \cdots p_n!}
 K^{\NCtensorT}_{p_1+p_2+\cdots+p_n}(
 \underbrace{X_{1}, X_{1},\ldots, X_{1}}_{\text{$p_1$ times}},
 \underbrace{X_{2}, X_{2},\ldots, X_{2}}_{\text{$p_2$ times}},
 \ldots,
 \underbrace{X_{n},X_n,\ldots,X_{n}}_{\text{$p_n$ times}}).
\end{multline}
Comparing the coefficients of $z_1z_2\cdots z_n$ in \eqref{sum log} with those
of \eqref{CBH d}  we conclude that $K_n^{\NCtensorT}(X_1,X_2,\dots, X_n)=0$. 
\end{proof}

\subsection{Specialization to free algebras}
\label{ssec:freemonoid}
We restrict the unshuffle spreadability system $\cS_{\NCtensorT}$
to the case when the underlying algebra $\mathcal{A}$ is a free algebra. 
Our aim here is to show that in this case the cumulants are Lie
polynomials \cite{Reutenauer:1993:freeLiealg} and to indicate further
connections to the theory of Hopf algebras.
This section was rewritten after \cite{LehnerNovelliThibon:2020} was published,
to which we refer for further developments in this direction.

Let $\fA$ be an alphabet (a set), whose elements be denote by
$a_1,a_2,\dots$.
Let $\C\langle\fA\rangle$ be the unital free associative algebra generated by
$\fA$,
i.e., the unital polynomial ring in  noncommuting indeterminates $\fA$.
As a vector space it is spanned by the elements of the free monoid $\fA^\ast$
generated by $\fA$ endowed with the concatenation product.
An element $w=a_1 a_2\cdots a_n\in\fA^\ast$ with $a_i\in\fA$ is called a \emph{word}
and $n$ is its \emph{length}.
The length of the unit $1$ is understood to be 0.

\begin{definition}
    The specialization   $\cS_{\NCtensorT} (\C\langle\fA\rangle)$
    of the  unshuffle spreadability system
    from Section~\ref{sec NCT}   for the ncps $(\C\langle\fA\rangle,\text{Id})$
    is  called the \emph{free Lie spreadability system}  and denoted by $\cS_{\freeLie}(\fA)$.
\end{definition}

The free algebra $\C\langle\fA\rangle$ also carries a Hopf algebra structure
\cite{CartierPatras:2021:classical}.
Indeed, the \emph{unshuffle coproduct}
is the homomorphism $\delta:\C\langle\fA\rangle\to \C\langle\fA\rangle\otimes\C\langle\fA\rangle $
for which every generator is a primitive element
$$
\delta(a) = a\otimes 1 + 1\otimes a
.
$$
For an arbitrary word of length $n$ it yields the sum over all unshuffles
$$
\delta(w) = \sum_{I\sqcup J =[n]} w\restr_I\otimes w\restr_J
,
$$
i.e., splittings into two subwords which keep the order.
This coproduct is obviously cocommutative.

Denote further by $\delta_m: \C\langle\fA\rangle\to\C\langle\fA\rangle^{\otimes m}$
its $(m-1)$-fold iteration
which is uniquely determined by the values
$$
\delta_m(a)= \sum_{k=1}^{m} 1^{\otimes (k-1)} \otimes a \otimes 1^{\otimes (m-k)},\qquad a\in\fA
$$
on the generators (compare Definition~\ref{def:dot})
and in general yields the sum over all $m$-unshuffles, or equivalently, ordered 
pseudopartitions (see Definition \ref{app:def:pseudopartition})
\begin{equation}
\label{eq:delta}
\delta_k(a_1a_2\cdots a_n)
= \sum_{\substack{\pi =(P_1,P_2,\dots) \in \EOSP_n\\
         \abs{\pi}=k}}
    a_{P_1}\otimes a_{P_2}\otimes \cdots \otimes a_{P_k}. 
\end{equation}
This reproduces the dot operation  
\[
N.a = \delta_N(a) \otimes 1^{\otimes \infty}, \qquad a \in \fA. 
\]
Let $\conc_k: \C\langle\fA\rangle^{\otimes k} \to \C\langle\fA\rangle$ be the multiplication map defined by 
$$
\conc_k(w_1\otimes w_2\otimes\cdots \otimes w_k)=w_1w_2\cdots w_k,\qquad w_i\in
\fA^*
$$
and more generally
for endomorphisms $f_1, f_2,\dots, f_k$ of $\C\langle\fA\rangle$, we define the convolution 
$$
f_1 \ast f_2\ast\cdots \ast f_k :=\conc_k \circ (f_1\otimes f_2\otimes\cdots \otimes f_k) \circ \delta_k. 
$$
In particular, the convolution power $\Psi_k=\Id^{\ast k}$ is called the
\emph{$k$-dilation} or \emph{Adams operation}
\cite[Definition~4.1.1]{CartierPatras:2021:classical},
so 
\begin{equation}
  \tilde{\varphi}(N.a_1,N.a_2,\dots,N.a_n) = \Psi_N(a_1a_2\cdots a_n).
\end{equation}
The \emph{counit} is the unique linear map $\epsilon:\C\langle\fA\rangle \to \C$ such that
$$
\epsilon(1)=1,\qquad
\epsilon(w)=0, \quad w\in\fA^\ast \setminus\{1\}.  
$$

\begin{defi}
Let $\Pi: \C\langle\fA\rangle \to \C\langle\fA\rangle$ be the map defined by 
$$
\Pi=\sum_{k=1}^\infty \frac{(-1)^{k-1}}{k} (\Id-\epsilon)^{\ast k}, 
$$
where $\epsilon$ is regarded as an endomorphism of $\C\langle\fA\rangle$.  
\end{defi}
The following proposition provides an explicit expression for $\Pi(w)$
which shows that for any $a_i \in\fA$ the value $\Pi(a_1 a_2\cdots a_n)$
is a finite sum of words and thus indeed  $\Pi(\C\langle\fA\rangle ) \subseteq \C\langle\fA\rangle$.
\begin{prop}\label{Hausdorff1}
Let $a_1,a_2,\ldots,a_n \in \fA$. Then 
$$
\Pi(a_1a_2\cdots a_n) = \sum_{\pi = (P_1, P_2, \dots) \in \OSP_n}
\frac{(-1)^{\abs{\pi}-1}}{\abs{\pi}} a_{P_1} a_{P_2}\cdots a_{P_{\abs{\pi}}}. 
$$
\end{prop}
\begin{proof}
From \eqref{eq:delta} we have 
\begin{align*}
\Pi(a_1a_2\cdots a_n) 
&= \sum_{k=1}^\infty \frac{(-1)^{k-1}}{k} (\Id-\epsilon)^{\ast k}(a_1a_2\cdots a_n)\\
  &=\sum_{k=1}^\infty \frac{(-1)^{k-1}}{k} (\conc_k \circ (\Id-\epsilon)^{\otimes k}\circ \delta_k)(a_1a_2\cdots a_n) \\
  &=\sum_{k=1}^\infty \frac{(-1)^{k-1}}{k}
\sum_{(P_1,P_2,\dots,P_k)    \in \EOSP_n}
\!\!(\Id-\epsilon)(a_{P_1}) (\Id-\epsilon)(a_{P_2}) \cdots (\Id-\epsilon)(a_{P_k}). 
\end{align*}
Note that $(\Id-\epsilon)(1)=0$ and $(\Id-\epsilon)(w)=w$ for
$w\in\fA^\ast\setminus\{1\}$. If $P_i=\emptyset$ for some $i$, then
the product $(\Id-\epsilon)(a_{P_1}) (\Id-\epsilon)(a_{P_2}) \cdots
(\Id-\epsilon)(a_{P_k})$ vanishes.
Therefore only proper ordered set partitions contribute to the sum, and in
particular the sum extends over $k \leq n$ only.  This proves the claim. 
\end{proof}

Combining Proposition~\ref{Hausdorff1} with Proposition~\ref{prop:mc-NCT} we
conclude the following identity. 

\begin{thm}\label{thm:Hausdorff2}
Let $a_1,a_2,\ldots,a_n \in \fA$ and let $K^{\freeLie}_\pi$ be the cumulants associated to the spreadability system $\cS_{\freeLie}$. Then 
$$
\Pi(a_1a_2\cdots a_n) = K^{\freeLie}_n(a_1,a_2,\dots,a_n). 
$$ 
\end{thm}
\goodbreak{}

\begin{remark}\label{Rem:FL}   \
\begin{enumerate}[label=\arabic*.,leftmargin=1cm] 
\item The space of \emph{Lie polynomials} $\Lie(\fA)$ is the smallest subspace
 of $\C\langle\fA\rangle$ that contains $\fA$ and is closed with respect to the Lie bracket
 $[X,Y]=XY-YX$. It is well known that $\Pi$ is a Lie projector, i.e.,
 $\Lie(\fA)=\Pi(\C\langle\fA\rangle)$ \cite[Theorem 3.7]{Reutenauer:1993:freeLiealg}. 
So in the above setting cumulants with entries from $\fA$ are exactly Lie
polynomials.

\item
 The values at words of higher order are given by convolution
powers of $\Pi$, i.e.,
\begin{equation*}
\Pi_k=\frac{1}{k!}\Pi^{\ast k},\qquad k\in\N 
\end{equation*}
(see \cite{Reutenauer:1993:freeLiealg}), 
and we can show that
\begin{equation*}
\Pi_k(a_1a_2\cdots a_n) =\frac{1}{k!}\sum_{\substack{\pi\in \OSP_n\\\abs{\pi}=k}} K^{\freeLie}_\pi(a_1,a_2,\dots,a_n),\qquad a_i\in\fA,~ i\in[n]. 
\end{equation*}
is the coefficient of $N^{k}$ appearing in $\tilde{\phi}((N.a_1) (N.a_2)\cdots (N.a_n))$ by Theorem~\ref{thm:mc} and Proposition~\ref{ad}. In other words, 
\begin{equation*}
\tilde{\phi}((N.a_1) (N.a_2)\cdots (N.a_n))= \sum_{k=1}^n N^k \Pi_k(a_1a_2\cdots a_n),\qquad a_i\in\fA,~ i\in[n]. 
\end{equation*}
\end{enumerate}
\end{remark}

Now the combination of Theorem~\ref{Thm:CBH-c} and Theorem~\ref{thm:Hausdorff2}
reproduces the CBH formula
\begin{equation}\label{CBH}
\log(e^{a_1} e^{a_2}\cdots e^{a_n}) = \sum_{\substack{(p_1,p_2,\dots,p_n) \in
    (\N\cup \{0\})^n,\\(p_1,p_2,\dots,p_n) \neq (0,0,\dots,0)  }} \Pi \left(
  \frac{a_1^{p_1} a_2^{p_2}\cdots a_n^{p_n}}{p_1! p_2!\cdots p_n!}\right); 
\end{equation}
see \cite[Lemma 3.10]{Reutenauer:1993:freeLiealg}.

\subsection{Coefficients of the Campbell-Baker-Hausdorff formula}
We adopt the notations and definitions in the previous subsection. 
The coefficients of the Campbell-Baker-Hausdorff formula, when written out in
the monomial basis, were first computed using generating functions in
\cite{Goldberg:1956:formal} and are 
called \emph{Goldberg coefficients}; a combinatorial proof
can be found in \cite[Theorem 3.11]{Reutenauer:1993:freeLiealg}. 
In the following we give another derivation of the Goldberg coefficients;
see also a recent proof using the theory of noncommutative symmetric functions \cite{FoissyPatrasThibon:2016:deformations}.

\begin{lem}\label{lem0123}
Let $\eta=(E_1,E_2,\dots,E_e)$ be an interval partition of $[n]$ where the blocks are in canonical order, i.e.\ if $s<t$ then $i<j$ for all $ i\in E_s, j\in E_t$. 
Then for any $X_i \in \C\langle\fA\rangle$ and any $\pi\in \OSP_n$, we have
$$
\varphi_{\pi \curlywedge \eta}(X_1,X_2,\dots,X_n)=\varphi_{\pi}(X_1,X_2,\dots,X_n). 
$$
\end{lem}
\begin{proof}
The statement holds by definition because every block $P\in\pi$ is the concatenation
of $P\cap E_1, P\cap E_2, \dots, P\cap E_e$.
\end{proof}
\begin{remark} Thus a sequence of distinct letters satisfies some partial independence, but it is not $\cS_{\freeLie}$-independent. Indeed, $(a_1,a_2)\in \fA \times \fA$ is $\cS_{\freeLie}$-independent if and only if $a_1=a_2$ by Proposition~\ref{NCT independence}. 
\end{remark}

\begin{thm}[{\cite{Goldberg:1956:formal,Reutenauer:1993:freeLiealg}}]
  \label{thm:coefficient}  
  Let $a_1,a_2,\ldots,a_n$ be distinct letters from the alphabet  $\fA$, let $r\in\N$ and
  $q_j,i_j \in \N$ for $j\in [r]$ such that $i_j \neq i_{j+1}$ for $j\in[r-1]$. 
  Then the coefficient of the monomial $a_{i_1}^{q_1}a_{i_2}^{q_2}\cdots
a_{i_r}^{q_r}$ appearing in $\log(e^{a_1}\cdots e^{a_n})$ is given by 
\begin{equation*} \frac{1}{q_1!q_2!\cdots q_r!}\int_{-1}^0 x^{\des(\underline{i})} (1+x)^{\asc(\underline{i})} \prod_{j=1}^r P_{q_j}(x)\,dx, 
\end{equation*}
where $P_q(x)$ are the homogeneous Euler polynomials already encountered in Proposition~\ref{prop:goldberg}. 
\end{thm}

\begin{proof}  
  In order for a monomial $v=a_{i_1}^{q_1}a_{i_2}^{q_2}\cdots a_{i_r}^{q_r}$ to
  occur as a term in $\Pi(a_1^{p_1}a_2^{p_2}\cdots a_n^{p_n})$ it necessarily
  has to be a rearrangement of the word (= permutation of the multiset)
  $w=a_1^{p_1}a_2^{p_2}\dotsm a_n^{p_n}$, 
  since the projector in the CBH formula \eqref{CBH} does not
  change multiplicities and therefore
  every letter must occur the same number of times in $v$ and $w$. 
  
  Thus $p_k=\sum_{j\in B_k}q_j$ where  $B_k=\{j : i_j=k\}$.
  Let $p=p_1+p_2+\dots+p_n$ be the total length of $w$
  and $\eta=(A_1,A_2,\dots,A_n)\in \OIP_p$ be the ordered interval partition
  corresponding to the composition $(p_1,p_2,\dots,p_n)$, i.e.,
  $A_j=\{p_1+p_2+\dots+p_{j-1}+1,
  p_1+p_2+\dots+p_{j-1}+2,\dots,p_1+p_2+\dots+p_j\}$
  and $\abs{A_j}=p_j$.

  From Theorem~\ref{thm:Hausdorff2} we infer
  \begin{align*}
    \Pi(a_{1}^{p_1}a_{2}^{p_2} \cdots a_{n}^{p_n})
    &=   K^{\freeLie}_{p}(\underbrace{a_1, a_1,\ldots, a_1}_{\text{$p_1$ times}},
                   \underbrace{a_2, a_2,\ldots, a_2}_{\text{$p_2$ times}},
                   \ldots,
      \underbrace{a_{n},a_{n},\ldots,a_{n}}_{\text{$p_n$ times}})
      \\
    & =:K_p^{\freeLie}(a_1^{\dots p_1},a_2^{\dots p_2}, \dots, a_n^{\dots p_n}). 
  \end{align*}

 Note that by Lemma~\ref{lem0123} we have
 $$
 \varphi_{\pi}(a_1^{\dots p_1},a_2^{\dots p_2}, \dots, a_n^{\dots p_n}) 
= \varphi_{\pi \curlywedge \eta}(a_1^{\dots p_1},a_2^{\dots p_2},
\dots,a_n^{\dots p_n})
$$
for any $\pi \in \OSP_p$
and 
thus
\begin{align*}
K^{\freeLie}_{p}(a_1^{\dots p_1},a_2^{\dots p_2}, \dots, a_n^{\dots p_n}) 
&= \sum_{\sigma \in \OSP_p} \varphi_\sigma(a_1^{\dots p_1},a_2^{\dots p_2}, \dots, a_n^{\dots p_n}) 
\, \widetilde{\mu}(\sigma, \hat{1}_p)
  \\
&= \sum_{\sigma \in \OSP_p} \varphi_{\sigma\curlywedge\eta}
   (a_1^{\dots p_1},a_2^{\dots p_2}, \dots, a_n^{\dots p_n}) 
   \, \widetilde{\mu}(\sigma, \hat{1}_p)\\
&= \sum_{\substack{\tau \in \OSP_p \\ \bar{\tau}\leq \bar{\eta}}}
   \varphi_{\tau}(a_1^{\dots p_1},a_2^{\dots p_2}, \dots, a_n^{\dots p_n}) 
   \,w(\tau, \eta)
\end{align*}
by Proposition~\ref{prop:Weisner}.
Thus in order to determine the coefficient of $v$ we must
collect all ordered set partitions $\tau$ such that 
\begin{equation*} T_v =\{\tau\in\OSP_p : \bar{\tau}\leq\bar{\eta}, 
 \varphi_{\tau}(a_1^{\dots p_1},a_2^{\dots p_2}, \dots, a_n^{\dots p_n}) = v 
\}
\end{equation*}  
and then sum up the corresponding values of $w(\tau,\eta)$.
Let us now investigate the structure of this set.
First note that $T_v$ is an order ideal: if $\tau\in T_v$ and $\tau'\leq \tau$,
then $\tau'\in T_v$, because every block of $\tau$ contains only repetitions
of one letter and further refinement of $\tau$ leaves the end result
$v$ invariant.
Moreover $T_v$ is the disjoint union of the principal ideals
$\downset\tau_0=\{\tau : \tau\leq \tau_0\}$
where $\tau_0\in T_v$ is maximal.

The maximal partition $\tau_0$ arises as follows:
After application of $\phi_\tau$ for $\tau\in T_v$
each factor $a_j^{p_j}$ is divided into pieces
$a_j^{q_l}$, $l\in B_j$ and the number of such subdivisions is
the multinomial coefficient
$$
\binom{p_j}{q_l : l\in B_j}
.
$$
These subdivisions are in one-to-one correspondence with
ordered set partitions of $A_j$ which, pieced together in the order of $j$,
give rise to a maximal ordered set partition $\tau_0\in T_v$.
Thus in total there are  as many maximal ordered set partitions
as there are subdivisions, namely 
\begin{equation*} \binom{p_1}{q_l:l\in B_1}
\binom{p_2}{q_l:l\in B_2}
\dotsm
\binom{p_n}{q_l:l\in B_n}
=
 \frac{p_1!p_2!\cdots p_n!}{q_1!q_2!\cdots q_r!}
.
\end{equation*}
By Proposition~\ref{app:prop:princideals} the principal 
ideal $\downset\tau_0$ generated by a maximal ordered set partition $\tau_0$ 
is isomorphic to
$$
\OSP_{q_1}\times\OSP_{q_2}\times\dotsm\OSP_{q_r}
$$
and in particular, all principal ideals are isomorphic.
Moreover the number of ascents is $\asc_\eta(\tau)=\asc(\underline{i})$
and does not depend on the choice of $\tau\in T_v$.
Thus all ideals $\downset\tau_0$ deliver the same contribution
and as a consequence of the discussion above we are left with
$$
\sum_{\tau\in T_v} w(\tau,\eta)
=
 \frac{p_1!p_2!\cdots p_n!}{q_1!q_2!\cdots q_r!}
\sum_{\tau\leq\tau_0} w(\tau,\eta)
$$
for one fixed maximal element $\tau_0\in T_v$.

A canonical representative $\tau_0$ is obtained by concatenating
consecutive subintervals of length $q_j$ from $A_{i_j}$, $j=1,2,\dots,r$.
Let us now turn to the value of $w(\tau,\eta)$. 
As seen above, the numbers of both ascents and descents of $\tau$ 
only depend on those of the sequence $\underline{i}$
and since $\underline{i}$ has no plateaux we infer from \eqref{eq:desplaasc+1} 
that $r=\des(\underline{i}) + \asc(\underline{i})+1$.
On the other hand, if we denote by $(\tau_1,\tau_2,\dots,\tau_r)$
the image of $\tau\in\downset\tau_0$ 
under the isomorphism of Proposition~\ref{app:prop:princideals}
then the first exponent in formula    \eqref{eq:weisnercoeff}  of Proposition~\ref{prop:Weisner} becomes
$$
\abs{\tau}-\asc_\eta(\tau) - 1 
= \sum_{j=1}^r \abs{\tau_j}-\asc(\underline{i}) - 1
= \sum_{j=1}^r (\abs{\tau_j}-1) + \des(\underline{i})
$$
and thus
$$
w(\tau,\eta) = \int_{-1}^0 x^{\des(\underline{i})}(1+x)^{\asc(\underline{i})} \prod_{j=1}^r x^{\abs{\tau_j}-1} \, d x 
$$
and summing over the cartesian product yields the total value
$$
\sum_{\tau\in T_v} w(\tau,\eta)
=
 \frac{p_1!p_2!\cdots p_n!}{q_1!q_2!\cdots q_r!}
 \int_{-1}^0 x^{\des(\underline{i})}(1+x)^{\asc(\underline{i})} \prod_{j=1}^r
 \left(
   \sum_{\rho\in\OSP_{q_j}}
   x^{\abs{\rho}-1}
 \right)  d x
.
$$
Finally note that 
\begin{equation*}
\sum_{\rho\in\OSP_{q}}x^{\abs{\rho}-1}= \sum_{\sigma \in \SP_{q}} \abs{\sigma}!\, x^{\abs{\sigma}-1} =\sum_{k=1}^{q} k!\, S(q,k)  x^{k-1}. 
\end{equation*}
is indeed the homogeneous Euler polynomial as claimed, 
see Remark~\ref{rem:Frobenius}.
\end{proof}

\begin{remark} The authors were not able to prove Theorem~\ref{thm:coefficient}
  as a corollary of Propositions~\ref{cum-mom} and \eqref{prop:goldberg} although Goldberg coefficients appear in both formulas. 
\end{remark}

\section{Central limit theorem}\label{CLT}
Cumulants provide a natural framework to understand central limit theorems. 
Speicher and Waldenfels studied central limit theorems in a general setting
of noncommutative probability assuming a certain \emph{singleton
  condition}~\cite{SvW94} (see also \cite{AccardiHashimotoObata:1998:role}).
In this section we will see that a similar approach also applies 
in the setting of a spreadability system, provided that an appropriate
singleton condition holds.

\begin{defi} \ 
\begin{enumerate}[label=\rm(\roman*),leftmargin=1cm]
\item An element $k \in[n]$ is called a singleton of $\pi \in \OSP_n$ if $\{k\} \in \pi$. 
\item Let $\cS=(\alg{U}, \tilde{\varphi},(\iota^{(i)})_{i \geq 1})$ be a
 spreadability system for a $\alg{B}$-ncps $(\mathcal{A},
 \varphi)$. 
 We will say that the \emph{singleton condition holds for $\cS$} if
 $\varphi_\pi(X_1,X_2,\ldots,X_n)=0$ for every tuple $(X_1,X_2,\dots,X_n)$ and
 every partition $\pi$ containing a singleton $\{k\}$ such that
 $\varphi(X_k)=0$. 
\end{enumerate}
\end{defi}

Under this assumption we can show the following type of central limit theorem. 

\begin{thm} Assume that a spreadability system $(\alg{U}, \tilde{\varphi},(\iota^{(i)})_{i \geq 1})$ for a $\alg{B}$-ncps $(\mathcal{A}, \varphi)$ satisfies the singleton condition. 
Assume $\varphi(X)=0$ and let $Y_N:=\frac{N.X}{\sqrt{N}}$. Then, for each $n \in\N$ and $\rho \in \SP_n$, 
\[
\lim_{N\to\infty}\varphi_\rho(Y_N, Y_N, \dots, Y_N)=
\begin{cases}
\sum_{\substack{\pi\in\OSP_n^{(2)}\\ \pi\leq \rho}} \frac{1}{\abs{\pi}!}\varphi_\pi(X,X,\ldots,X), &n \text{~is even},\\
0,&n\text{~is odd}. 
\end{cases}
\] 
where $\OSP_n^{(2)}$ is the set of pair ordered set partitions, i.e.\ every block of $\pi \in \OSP_n^{(2)}$ contains exactly 2 entries. 
\end{thm}
\begin{proof}
Note that the following holds: if $\pi$ has a singleton at $k$ and $\varphi(X_k)=0$, then 
\begin{equation*}
K_\pi(X_1,X_2,\ldots,X_n)=0. 
\end{equation*}
This holds because $K_\pi(X_1,X_2,\ldots,X_n)$ is the coefficient of $N$ in $\varphi_\pi(N.X_1,N.X_2,\ldots,N.X_n)$, and 
\begin{equation*}
\varphi_\pi(N.X_1,N.X_2,\ldots,N.X_n)=\sum_{i_1,i_2,\ldots,i_n\in[n]}\varphi_{\pi\curlywedge \kappa(i_1,i_2,\ldots,i_n)}(X_1,X_2,\ldots,X_n)=0
\end{equation*} because each 
$\pi\curlywedge \kappa(i_1,i_2,\ldots,i_n)$ has a singleton at $k$. 

Now multilinearity and extensivity of cumulants imply
\begin{equation*}
K_\pi(Y_N,Y_N,\ldots,Y_N)= N^{-\frac{n}{2}+\abs{\pi}}K_\pi(X,X,\ldots,X). 
\end{equation*}
If $\pi$ has a singleton, this is zero. If $\pi$ does not have a singleton nor $\pi$ is not a pair ordered set partition, then $\abs{\pi}<\frac{n}{2}$. Therefore,  
\begin{equation*}
\lim_{N\to\infty}K_\pi(Y_N,Y_N,\ldots,Y_N)= 
\begin{cases}
K_\pi(X,X,\ldots,X),&\text{if~$n$ is even and~}\pi\in\OSP_n^{(2)}, \\
0, &\text{otherwise}. 
\end{cases} 
\end{equation*}
If $\pi\in\OSP_n^{(2)}$, then $K_\pi(X,X,\ldots, X)= \varphi_\pi(X,X,\ldots,X)$
from Theorem~\ref{thm:mc} because the expectations $\varphi_\sigma(X,X,\ldots,X)$ all
vanish for $\sigma <\pi$ from the singleton condition. Finally, from Theorem
\ref{thm:mc}, we obtain the conclusion.  
\end{proof}

Thus the use of cumulants simplifies the proof of the central limit
theorem. It may happen that the moments of the  limit distribution are not
uniquely determined only by the variance of $X$ alone, because in 
general $\varphi_\pi(X,X,\dots,X)$ cannot be written in terms of $\varphi(X^2)$.
For example, the limit distribution for the c-monotone spreadability system is
characterized by the moments  
\begin{equation*}
\lim_{N\to\infty}\varphi(Y_N^n)=
\begin{cases}
\sum_{\pi\in\MP_n^{(2)}} \frac{1}{\abs{\pi}!}\alpha^{2\abs{\Outer(\pi)}}\beta^{2\abs{\Inner(\pi)}}, &n \text{~is even},\\
0,&n\text{~is odd}, 
\end{cases}
\end{equation*}
where $\alpha^2=\varphi(X^2)$, $\beta^2=\psi(X^2)$ and $\MP_n^{(2)}$ is the set
of monotone pair partitions; the reader is referred to Theorems 4.7, 5.1 of
\cite{Hasebe:2011:conditionally}. The limit moments  are not uniquely
determined by $\alpha^2$ but also depend on the second linear map $\psi$.
A natural condition to ensure uniqueness is a calculation rule (Definition \ref{defi:universal}); then the limit distribution of the central limit theorem is determined only by the variance and by the constants $s(\bar{\pi}; \pi)$.  More precisely, we deduce from the calculation rule \eqref{eq:calculation_rule} that $\varphi_\pi(X,X,\ldots,X)=s(\bar{\pi}; \pi)\alpha^{2\abs{\pi}}$ for $\pi \in \OSP_n^{(2)}$ provided $\varphi(X)=0$, where $\alpha^2=\varphi(X^2)$. The limit moments can therefore be written as 
\begin{equation*}
  \lim_{N\to\infty}\varphi_\rho(Y_N, Y_N,\dots, Y_N)=
  \begin{cases}
    \sum_{\substack{\pi\in\OSP_n^{(2)}\\ \pi \leq \rho }} \frac{s(\bar{\pi}; \pi)}{\abs{\pi}!}\alpha^{2\abs{\pi}}, &n \text{~is even},\\
    0,&n\text{~is odd}. 
  \end{cases}
\end{equation*}

\section{Open Problems} 
\begin{enumerate}[label=\arabic*.,leftmargin=1cm]
\item Find a unified proof of the occurrence of Goldberg coefficients in Theorem~\ref{thm:coefficient} and Proposition~\ref{cum-mom}.
\item Find the values of the M\"obius function on the poset of monotone partitions.
\item Compute the cumulants for the V-monotone spreadability system and compare
 them with the operadic cumulants of Jekel and Liu \cite{JekelLiu:2020:operad}.
\item Define a notion of multifaced spreadability systems and the associated cumulants, cf.\ Remark~\ref{rem:multiface}.

\end{enumerate}

\section*{Acknowledgements} The first-named author was supported by European Commission, Marie Curie
Actions IIF 328112 ICNCP and by JSPS Grant-in-Aid for Young Scientists (B)
15K17549. We thank Hayato Saigo for discussions in the early stage of this work
and J.-C.~Novelli and J.Y.~Thibon for pointing out errors in an earlier
version. Finally we thank the anonymous referee for numerous corrections and suggestions
which helped us to substantially improve the organization of the paper.

\appendix

\section{Ordered set partitions}\label{app:partition composition}
\subsection{Set partitions} 
Let $\N$ denote the set of natural numbers $\{1,2,3,\ldots \}$, and let $[n]$
denote the finite set $\{1,2,\ldots, n\} \subseteq \N$.  

\begin{definition}[Set Partitions]
A \emph{set partition}, or simply partition, of a finite set $A$ is a set of mutually disjoint subsets
$\pi=\{P_1,P_2,\ldots,P_k\}$ such that $\cup_{i=1}^k P_i = A$.  
The number $k$ is the \emph{size} of the partition and denoted by $\abs{\pi}$. 
The elements $P\in\pi$ are called \textit{blocks} of $\pi$.
The set of partitions of $A$ is denoted by $\SP_A.$ 
We are mostly concerned with the case $A=[n]$
and in this case $\SP_{[n]}$ is abbreviated to $\SP_n$
and called (set) \textit{partitions of order $n$}.
As is well known there is a one-to-one
correspondence between set partitions $\pi$ of $[n]$ and equivalence relations
on $[n]$ by defining for $\pi\in\SP_n$ and $i,j \in [n]$ the
relation $i \sim_\pi j$ to hold if and only if there is a block $P \in \pi$
such that both $i,j \in P$.
\end{definition}

The set partitions of fixed order $n$ form a lattice under \emph{refinement
  order}:
\begin{definition}[Refinement Order]
For partitions $\pi$ and $\sigma$ we write $\pi \leq \sigma$ if for any block
$P \in \pi$, there exists a block $S \in \sigma$ such that $P \subseteq S$.  
In other words, every block of $\sigma$ is a union of blocks of $\pi$.
The minimal element of this lattice is
$\hat{0}_n=\{\{1\},\{2\},\dots,\{n\}\}$ and the maximal element
$\hat{1}_n=\{[n]\}$. 
\end{definition}

We proceed with the description of several classes of set partitions.

\begin{definition}[Classes of Set Partitions]
  \label{app:def:partitions}
   Let $\pi\in\SP_n$ be a set partition.
  \begin{enumerate}[label=\rm(\roman*),leftmargin=1cm]
   \item
    Two (distinct) blocks $B$ and $B'\in\pi$ are said to be \emph{crossing} if
    there are elements $i<i'<j<j'$ such that $i,j\in B$ and $i',j'\in B'$.
    $\pi$ is called \emph{noncrossing} if there are no crossing blocks,
    i.e., if there is no
    quadruple of elements $i<j<k<l$ s.t.\ $i\sim_\pi k$, $j\sim_\pi l$ and
    $i\not\sim_\pi j$.  The noncrossing partitions of order $n$ form a sublattice
    which we denote by $\NC_n$.
   \item\label{item:app:def:partitions2}
    Two blocks $B,B'$ of a noncrossing partition $\pi$ are said to form a
    \emph{nesting} if there are $i,j \in B$ such that $i<k<j$ for any $k \in B'$.
    In this case $B$ is called the \emph{outer block} of the nesting and 
    $B'$ is called the \emph{inner block} of the nesting.
   \item
    A block~$B$ of a noncrossing partition $\pi$ is \emph{inner} if 
    $B$  is the inner block of a nesting of $\pi$.
    If this is not the case $B$ is called an \emph{outer block} of $\pi$.
    The set of inner blocks of $\pi$ is denoted by $\Inner(\pi)$
    and the set of outer blocks of $\pi$ by $\Outer(\pi)$.
   \item
    An \emph{interval partition} is a partition $\pi$ for which every block is an interval.
    Equivalently, this means that~$\pi$ is noncrossing and has no nestings.
    The set of interval partitions of $[n]$ is denoted by $\IP_n$. 
    \end{enumerate}
Analogous definitions apply to any finite totally ordered set $A$ and the corresponding sets of partitions are denoted $\SP_A$, $\NC_A$, $\IP_A$ etc.
\end{definition}
Examples of partitions are shown in Fig.~\ref{app:fig:partitions}.
\begin{figure}
\begin{minipage}{.3\textwidth}
\centering
\setlength{\unitlength}{2\unitlength}
      \begin{picture}(56,6.5)(1,0)
        \put(2,0){\line(0,1){7.5}}
        \put(8,0){\line(0,1){4.5}}
        \put(14,0){\line(0,1){4.5}}
        \put(20,0){\line(0,1){4.5}}
        \put(26,0){\line(0,1){7.5}}
        \put(32,0){\line(0,1){4.5}}
        \put(38,0){\line(0,1){4.5}}
        \put(44,0){\line(0,1){7.5}}
        \put(50,0){\line(0,1){4.5}}
        \put(56,0){\line(0,1){4.5}}
        \put(8,4.5){\line(1,0){6}}
        \put(20,4.5){\line(1,0){0}}
        \put(32,4.5){\line(1,0){6}}
        \put(2,7.5){\line(1,0){42}}
        \put(50,4.5){\line(1,0){6}}
      \end{picture}

      a noncrossing partition
\end{minipage}
\begin{minipage}{.3\textwidth}
\centering
\setlength{\unitlength}{2\unitlength}
      \begin{picture}(56,3.5)(1,0)
        \put(2,0){\line(0,1){4.5}}
        \put(8,0){\line(0,1){4.5}}
        \put(14,0){\line(0,1){4.5}}
        \put(20,0){\line(0,1){4.5}}
        \put(26,0){\line(0,1){4.5}}
        \put(32,0){\line(0,1){4.5}}
        \put(38,0){\line(0,1){4.5}}
        \put(44,0){\line(0,1){4.5}}
        \put(50,0){\line(0,1){4.5}}
        \put(56,0){\line(0,1){4.5}}
        \put(2,4.5){\line(1,0){12}}
        \put(20,4.5){\line(1,0){6}}
        \put(32,4.5){\line(1,0){18}}
        \put(56,4.5){\line(1,0){0}}
      \end{picture}

      an interval partition
\end{minipage}
\caption{Examples of partitions}\label{app:fig:partitions}
\end{figure}
The lattice of interval partitions plays a central role in this paper
and has a particularly simple structure.
\begin{prop}\label{app:prop:interval-Boolean}
The lattice of interval partitions $\IP_n$ is anti-isomorphic 
to the Boolean lattice $\Bool_{n-1}$ via the lattice anti-isomorphism
$$
(I_1, I_2, \ldots, I_p) \mapsto\{r_1, r_2, \ldots,r_{p-1}\} \subseteq [n-1],  
$$
where the blocks $I_i$ are uniquely determined by their maximal elements $r_i$;
 note that always $r_p =n$. 
\end{prop}

The following construction inverts the previous bijection 
in a certain sense.
\begin{definition}
  \label{app:def:ncmax}
  Fix a number $n\in\N$ and a subset $A\subseteq [n]$. 
  \begin{enumerate}[label=\rm(\roman*),leftmargin=1cm]
   \item\label{app:it:outncmax}  Among all noncrossing partitions
  containing $A^c$ as an outer block there is a maximal one, 
  which we denote by $\outncmax(A)$. Removing $A^c$ we obtain
  an interval partition of $A$ which we denote by $\outintmax(A)$;
  in other words, the blocks of $\outintmax(A)$ consist
  of the maximal contiguous subintervals of $A$. 
  Yet in another interpretation, the blocks of $\outintmax(A)$ 
  are the connected components of the graph induced on $A$
  from the integer line.
   \item\label{app:it:ncmax}  Among all noncrossing partitions
    containing $A^c$ as a block there is a maximal one, 
  which we denote by $\ncmax(A)$. Removing $A^c$ we obtain
  a noncrossing partition of $A$ 
  which we denote by $\intmax(A)$;
  in other words, the blocks of $\intmax(A)$ consist
  of the maximal contiguous subintervals of $A$ when we consider
  it on the circle, i.e.,
  the blocks of $\intmax(A)$ 
  are the connected components of the graph induced on $A$
  from the Cayley graph of $\Z_n$.
\end{enumerate}
See Figure~\ref{app:fig:OSP:numax} for examples.
\end{definition}

 \begin{figure}
   \begin{minipage}{0.3\linewidth}
\begin{tikzpicture}[scale=0.1]
\draw[thick=10] (2,0)--(2,4.5);
\draw (8,0)--(8,7.5);
\draw[thick=10] (14,0)--(14,4.5);
\draw[thick=10] (20,0)--(20,4.5);
\draw (26,0)--(26,7.5);
\draw[thick=10] (32,0)--(32,4.5);
\draw (38,0)--(38,7.5);
\draw[thick=10] (44,0)--(44,4.5);
\draw[thick=10] (50,0)--(50,4.5);
\draw (2,4.5)--(2,4.5);
\draw[thick=10] (14,4.5)--(20,4.5);
\draw (32,4.5)--(32,4.5);
\draw (8,7.5)--(38,7.5);
\draw[thick=10] (44,4.5)--(50,4.5);
\phantom{\draw[thick=0] (2.5,0)--(2,10.5);}
\end{tikzpicture}
\begin{center}
    $\outncmax(A)$ 
\end{center}
\end{minipage}
\qquad{}
\begin{minipage}{0.3\linewidth}
\begin{tikzpicture}[scale=0.1]
\draw[thick=10] (2,0)--(2,10.5);
\draw (8,0)--(8,7.5);
\draw[thick=10] (14,0)--(14,4.5);
\draw[thick=10] (20,0)--(20,4.5);
\draw (26,0)--(26,7.5);
\draw[thick=10] (32,0)--(32,4.5);
\draw (38,0)--(38,7.5);
\draw[thick=10] (44,0)--(44,10.5);
\draw[thick=10] (50,0)--(50,10.5);
\draw[thick=10] (14,4.5)--(20,4.5);
\draw (32,4.5)--(32,4.5);
\draw (8,7.5)--(38,7.5);
\draw[thick=10] (2,10.5)--(50,10.5);
\end{tikzpicture}
  \begin{center}
 $\ncmax(A)$    
  \end{center}
\end{minipage}
    \caption{The partitions $\outncmax(A)$ and $\ncmax(A)$
     for $A=\{1,3,4,6,8,9\}$ (fat).
   }
   \label{app:fig:OSP:numax}
 \end{figure}

\begin{remark} \ 
  \begin{enumerate}[label=\arabic*.,leftmargin=1cm]
  
   \item  Construction \ref{app:it:outncmax}
    occurs in some examples, see, e.g., Examples~\ref{ex:Boole} and~\ref{ex:monotone-diff-eq}.
    It gives rise to the unshuffle coproduct of Ebrahimi-Fard and Patras
    \cite{EbrahimiFardPatras:2015:cumulants}.
   \item 
    Construction \ref{app:it:ncmax} occurs in the recursion 
    \eqref{eq:freerecursion} for free cumulants and
    Example~\ref{ex:freediff}.
  \end{enumerate}
\end{remark}

The lattices considered so far have the following structural property.
It is easy to see for both $\SP_n$ and $\IP_n$ while for $\NC_n$ it is proved in 
\cite{Speicher:1994:multiplicative}.
\begin{prop}
  \label{app:prop:intervals}
  Let $P_n$ be one of $\SP_n$, $\NC_n$ and $\IP_n$.
  Then for any pair of elements $\sigma,\pi\in P_n$ such that
  $\sigma\leq\pi$
  there are uniquely determined numbers $k_j$ such that
  the interval $[\sigma,\pi]$ is isomorphic (as a lattice) to
  the direct product
  $$
  P_1^{k_1}\times P_2^{k_2} \times \dotsm \times P_n^{k_n}
  .
  $$
\end{prop}

\subsection{Ordered set partitions}

\begin{definition}[Ordered set partitions]    \label{defi:OP}  \
  \begin{enumerate}[label=\rm(\roman*),leftmargin=1cm]
 
   \item 
An \emph{ordered set partition}   of a set $A$  is a sequence 
$(P_1, P_2, \ldots,P_p)$ of distinct blocks such that $\{P_1, P_2, \ldots,P_p\}$ 
is a set partition of $A$. In other words, it is a set partition
with a total ordering of its blocks.
The set of ordered set partitions of $A$ is denote by $\OSP_A$ and $\OSP_{[n]}$
is abbreviated to $\OSP_n$.

\item \label{item:OP2}
Ordered set partitions with a fixed number $p$ of blocks
are in bijection with surjective functions from $[n]$ to $[p]$.
This bijection is implemented by identifying
an ordered set partition
$\pi=(P_1,P_2,\dots, P_p) \in \OSP_n$
with the function mapping an element $i \in [n]$
to the label $k$ of the block $P_k$ such that $i\in P_k$,
which we denote by $\pi(i):=k$.  
This function can be represented by 
a \emph{packed word} \cite{NovelliThibon:2006:construction},
i.e., the sequence of the images $\pi(1)\pi(2)\dots\pi(n)$.
When $\pi$ consists of singletons only, i.e., $\# P_i=1$ 
for all $i\in [p]$ then $\pi$ can be identified with
a permutation of $[n]$, and  the notations introduced above are consistent
with the familiar notations for permutations.
  \end{enumerate}
\end{definition}

\begin{remark}
Ordered set partitions are also known under the name of \emph{set compositions}, see,
e.g., \cite{BergeronZabrocki:2009:hopf}, \emph{pseudopermutations} 
\cite{KrobLatapy:2000:pseudo}.
In particular, packed words can be represented  as multiset
permutations and this point of view is crucial in Section~\ref{sec:vanishing}.
\end{remark}

\begin{defi} 
  \label{app:def:kernelpartition} \ 
\begin{enumerate}[label=\rm(\roman*),leftmargin=1cm]
\item  The \emph{ordered kernel set partition} $\kappa(i_1,i_2,\ldots,i_n)$ of a multiindex
    $(i_1,i_2,\ldots,i_n)$ is defined as follows.  
  First, pick the smallest value, say $p_1$,
  from $i_1, i_2,\ldots, i_n$  and define the block $P_1=\{k \in[n]
  \mid i_k=p_1\}$.  
  Next, pick the second smallest value $p_2$ from $i_1, i_2,\ldots,
  i_n$ and define the block $P_2=\{k \in[n]\mid i_k=p_2 \}$.  By
  repeating this procedure, we obtain an ordered set partition $(P_1, P_2, \ldots)$,
  which we denote by  $\kappa(i_1, i_2,\ldots,i_n)$.

\item  The \emph{kernel set partition} $\bar{\kappa}(i_1, i_2,\ldots,i_n)$ of a sequence
  of indices is defined as the underlying set partition
  $\overline{\kappa(i_1, i_2,\ldots,i_n)}$ of the corresponding ordered kernel set partition.
  In other words, it is the equivalence relation such that by $p\sim q$ if and
  only if $i_p=i_q$.
  \end{enumerate}
\end{defi}

\begin{definition}[Order Dropping Map]
  \label{app:def:orderdrop}
Let $\pi \mapsto \bar{\pi}$ be the map from $\OSP_n$ onto $\SP_n$ which
drops the 
order on blocks, that is,
\begin{equation*}
(P_1,P_2,\ldots) \mapsto \{P_1,P_2,\ldots\}.
\end{equation*}
We say that an ordered set partition is \emph{in canonical order}
if the blocks are sorted in ascending order according to their minimal
elements. 
\end{definition}

It will be convenient to transfer as much structure as possible from ordinary
set partitions to ordered set partitions when no confusion can arise.
For example, the notation $P\in \pi$  indicates that $P$ is a block of $\bar{\pi}$.
Let us next introduce a natural partial order relation on $\OSP_n$.
\begin{definition}[Refinement Order]
  Given two ordered set partitions $\pi=(P_1, P_2,\dots,P_p)$ and $\sigma=(S_1, S_2,\dots,S_s) \in \OSP_n$,
  we define the order relation $\pi \leq \sigma$ by the following requirements:
\begin{enumerate}[label=\rm(\roman*)]
\item $\bar{\pi}\leq\bar{\sigma}$ as set partitions. 
\item If $P_i \subseteq S_k, P_j \subseteq S_l$ for $i<j$, then $k \leq l$. 
\end{enumerate}
\end{definition}
\begin{remark}
In other words, $\pi\leq\sigma$ if every block of $\sigma$ is a union of
a contiguous sequence of blocks of $\pi$.
The elements dominated by a given ordered set partition $\sigma$ are obtained
as follows:
\begin{enumerate}[label=\arabic*.] 
 \item pick a label $i$
 \item split the block with label $i$ into two
 \item label one of the pieces with $i$ and the other one with $i+1$
 \item increment by one all other labels larger than $i$.
\end{enumerate}
See Figure~\ref{app:fig:OSP:order} for examples.

 \begin{figure}
   \begin{tikzpicture}
\begin{scope}[shift = {(0,3)},scale=0.05]

\draw (2,0)--(2,7.5);
\draw (8,0)--(8,4.5);
\draw (14,0)--(14,7.5);
\draw (20,0)--(20,4.5);
\draw (26,0)--(26,7.5);
\draw (8,4.5)--(20,4.5);
\draw (2,7.5)--(26,7.5);

\node at (2,-5) {\tiny $1$};
\node at (8,-5) {\tiny $2$};
\node at (14,-5) {\tiny $1$};
\node at (20,-5) {\tiny $2$};
\node at (26,-5) {\tiny $1$};
\end{scope}

\begin{scope}[shift = {(-3,0)},scale=0.05]
\draw (2,0)--(2,10.5);
\draw (8,0)--(8,7.5);
\draw (14,0)--(14,4.5);
\draw (20,0)--(20,7.5);
\draw (26,0)--(26,10.5);
\draw (14,4.5)--(14,4.5);
\draw (8,7.5)--(20,7.5);
\draw (2,10.5)--(26,10.5);

\node at (2,-5) {\tiny $1$};
\node at (8,-5) {\tiny $2$};
\node at (14,-5) {\tiny $3$};
\node at (20,-5) {\tiny $2$};
\node at (26,-5) {\tiny $1$};
\end{scope}

\begin{scope}[shift = {(0,0)},scale=0.05]
\draw (2,0)--(2,10.5);
\draw (8,0)--(8,7.5);
\draw (14,0)--(14,4.5);
\draw (20,0)--(20,7.5);
\draw (26,0)--(26,10.5);
\draw (14,4.5)--(14,4.5);
\draw (8,7.5)--(20,7.5);
\draw (2,10.5)--(26,10.5);

\node at (2,-5) {\tiny $1$};
\node at (8,-5) {\tiny $3$};
\node at (14,-5) {\tiny $2$};
\node at (20,-5) {\tiny $3$};
\node at (26,-5) {\tiny $1$};
\end{scope}

\begin{scope}[shift = {(3,0)},scale=0.05]
\draw (2,0)--(2,10.5);
\draw (8,0)--(8,7.5);
\draw (14,0)--(14,4.5);
\draw (20,0)--(20,7.5);
\draw (26,0)--(26,10.5);
\draw (14,4.5)--(14,4.5);
\draw (8,7.5)--(20,7.5);
\draw (2,10.5)--(26,10.5);

\node at (2,-5) {\tiny $2$};
\node at (8,-5) {\tiny $3$};
\node at (14,-5) {\tiny $1$};
\node at (20,-5) {\tiny $3$};
\node at (26,-5) {\tiny $2$};
\end{scope}

\node (B) at (0.7,2.5) {};
\node (A1) at (-2.4,0.7) {};
\node (A1a) at (-0.85,1.6) {$\not\leq$};
\node (A2) at (0.7,0.7) {};
\node (A2a) at (0.7,1.6) {$\leq$};
\node (A3) at (3.7,0.7) {};
\node (A3a) at (2.2,1.6) {$\leq$};
\draw (A1)--(A1a)--(B);
\draw (A2)--(A2a)--(B);
\draw (A3)--(A3a)--(B);
\end{tikzpicture}
    \caption{Examples of the order on $\OSP_n$.}
   \label{app:fig:OSP:order}
 \end{figure}

This order makes $(\OSP_A,\leq)$ a poset (but not a lattice as will be seen
shortly). The Hasse diagram of $\OSP_3$ is shown in Figure~\ref{app:fig:OSP3};
the order is read from outside to inside.
\begin{figure}
  \begin{tikzpicture}
    \node (A111) at (0,0) {$123$};
    \node (A121) at (0:2) {$13/2$};
    \node (A221) at (60:2) {$3/12$};
    \node (A211) at (120:2) {$23/1$};
    \node (A212) at (180:2) {$2/13$};
    \node (A112) at (240:2) {$12/3$};
    \node (A122) at (300:2) {$1/23$};                        

    \node (A231) at (30:3.7) {$3/1/2$};
    \node (A321) at (90:3.7) {$3/2/1$};
    \node (A312) at (150:3.7) {$2/3/1$};
    \node (A213) at (210:3.7) {$2/1/3$};
    \node (A123) at (270:3.7) {$1/2/3$};
    \node (A132) at (330:3.7) {$1/3/2$};

    \draw (A111)--(A121);
    \draw (A111)--(A221);
    \draw (A111)--(A211);
    \draw (A111)--(A212);
    \draw (A111)--(A112);
    \draw (A111)--(A122);

    \draw (A231)--(A121);
    \draw (A321)--(A221);
    \draw (A312)--(A211);
    \draw (A213)--(A212);
    \draw (A123)--(A112);
    \draw (A132)--(A122);

    \draw (A231)--(A221);
    \draw (A321)--(A211);
    \draw (A312)--(A212);
    \draw (A213)--(A112);
    \draw (A123)--(A122);
    \draw (A132)--(A121);
  \end{tikzpicture}
  \caption{The poset $\OSP_3$, ordered from outside to inside. The maximal element is in the center, the minimal
    elements are on the periphery.}
  \label{app:fig:OSP3}
\end{figure}

Note that $\pi\leq\sigma$ implies
$\bar\pi\leq\bar\sigma$, but not vice versa.
\end{remark}
We consider the following subclasses of ordered set partitions on $[n]$.
\begin{definition}[Classes of Ordered Set Partitions]
  \label{app:def:OSPclass} \ 
  \begin{enumerate}[label=\rm(\roman*),leftmargin=1cm]
 
    \item
    \label{app:def:ONC}
     An ordered set partition $\pi\in\OSP_n$ is called \textit{noncrossing}
    if the underlying set partition $\bar{\pi}$ has this property.
   The set of ordered noncrossing partitions of $[n]$ is denoted by $\ONP_n$. 
   Outer blocks and inner blocks of an ordered noncrossing partition are defined according to the case of noncrossing partitions. 
  \item \label{app:def:OI}
   An \textit{interval ordered set partition} of $[n]$ is an ordered set partition $\pi$ such that $\bar{\pi} \in \IP_n$.  The set of  interval ordered set partitions of $[n]$ is denoted by $\OIP_n$.
  \item
    \label{app:def:monopart}
    A
    noncrossing ordered set partition
  $\pi=(P_1, P_2, \dots,P_{\abs{\pi}})$ is called  \emph{monotone partition} if
  for every nesting  the outer 
  block precedes the inner block;
  in other words, the order of the blocks implements a linearization of the partial order given by the nesting relation.
  The set of monotone partitions is denoted by $\MP_n$. 
  \end{enumerate}
\end{definition}
\begin{exa}
  Figure~\ref{fig:monotonepartitions} shows some examples of monotone and
  non-monotone partitions whose underlying noncrossing partition is
  $\bar{\pi}=\begin{picture}(32,6.5)(1,0)
    \put(2,0){\line(0,1){7.5}}
    \put(8,0){\line(0,1){4.5}}
    \put(14,0){\line(0,1){4.5}}
    \put(20,0){\line(0,1){7.5}}
    \put(26,0){\line(0,1){7.5}}
    \put(32,0){\line(0,1){4.5}}
    \put(8,4.5){\line(1,0){6}}
    \put(2,7.5){\line(1,0){24}}
    \put(32,4.5){\line(1,0){0}}
  \end{picture}$.

\begin{figure}[h]
  \centering
\begin{minipage}{0.2\linewidth}
\setlength{\unitlength}{2\unitlength}
\begin{picture}(32,6.5)(1,0)
  \put(2,0){\line(0,1){7.5}}
  \put(8,0){\line(0,1){4.5}}
  \put(14,0){\line(0,1){4.5}}
  \put(20,0){\line(0,1){7.5}}
  \put(26,0){\line(0,1){7.5}}
  \put(32,0){\line(0,1){4.5}}
  \put(8,4.5){\line(1,0){6}}
  \put(2,7.5){\line(1,0){24}}
  \put(32,4.5){\line(1,0){0}}
\end{picture}

\tiny{}  1 \hskip0.6em 2 \hskip4.8em 3
\end{minipage}
\begin{minipage}{0.2\linewidth}
\setlength{\unitlength}{2\unitlength}
\begin{picture}(32,6.5)(1,0)
  \put(2,0){\line(0,1){7.5}}
  \put(8,0){\line(0,1){4.5}}
  \put(14,0){\line(0,1){4.5}}
  \put(20,0){\line(0,1){7.5}}
  \put(26,0){\line(0,1){7.5}}
  \put(32,0){\line(0,1){4.5}}
  \put(8,4.5){\line(1,0){6}}
  \put(2,7.5){\line(1,0){24}}
  \put(32,4.5){\line(1,0){0}}
\end{picture}

\tiny{}  1 \hskip0.6em 3 \hskip4.8em 2
\end{minipage}
\begin{minipage}{0.2\linewidth}
\setlength{\unitlength}{2\unitlength}
\begin{picture}(32,6.5)(1,0)
  \put(2,0){\line(0,1){7.5}}
  \put(8,0){\line(0,1){4.5}}
  \put(14,0){\line(0,1){4.5}}
  \put(20,0){\line(0,1){7.5}}
  \put(26,0){\line(0,1){7.5}}
  \put(32,0){\line(0,1){4.5}}
  \put(8,4.5){\line(1,0){6}}
  \put(2,7.5){\line(1,0){24}}
  \put(32,4.5){\line(1,0){0}}
\end{picture}

\tiny{}  2 \hskip0.6em 3 \,\hskip4.65em 1
\end{minipage}

\vskip3em

\begin{minipage}{0.2\linewidth}
\setlength{\unitlength}{2\unitlength}
\begin{picture}(32,6.5)(1,0)
  \put(2,0){\line(0,1){7.5}}
  \put(8,0){\line(0,1){4.5}}
  \put(14,0){\line(0,1){4.5}}
  \put(20,0){\line(0,1){7.5}}
  \put(26,0){\line(0,1){7.5}}
  \put(32,0){\line(0,1){4.5}}
  \put(8,4.5){\line(1,0){6}}
  \put(2,7.5){\line(1,0){24}}
  \put(32,4.5){\line(1,0){0}}
\end{picture}

\tiny{}  2 \hskip0.6em 1 \,\hskip4.65em 3
\end{minipage}
\begin{minipage}{0.2\linewidth}
\setlength{\unitlength}{2\unitlength}
\begin{picture}(32,6.5)(1,0)
  \put(2,0){\line(0,1){7.5}}
  \put(8,0){\line(0,1){4.5}}
  \put(14,0){\line(0,1){4.5}}
  \put(20,0){\line(0,1){7.5}}
  \put(26,0){\line(0,1){7.5}}
  \put(32,0){\line(0,1){4.5}}
  \put(8,4.5){\line(1,0){6}}
  \put(2,7.5){\line(1,0){24}}
  \put(32,4.5){\line(1,0){0}}
\end{picture}

\tiny{}  3 \hskip0.6em 1 \,\hskip4.65em 2
\end{minipage}
\begin{minipage}{0.2\linewidth}
\setlength{\unitlength}{2\unitlength}
\begin{picture}(32,6.5)(1,0)
  \put(2,0){\line(0,1){7.5}}
  \put(8,0){\line(0,1){4.5}}
  \put(14,0){\line(0,1){4.5}}
  \put(20,0){\line(0,1){7.5}}
  \put(26,0){\line(0,1){7.5}}
  \put(32,0){\line(0,1){4.5}}
  \put(8,4.5){\line(1,0){6}}
  \put(2,7.5){\line(1,0){24}}
  \put(32,4.5){\line(1,0){0}}
\end{picture}

\tiny{}  3 \hskip0.6em 2 \,\hskip4.65em 1
\end{minipage}

  \caption{Monotone partitions (upper row) and non-monotone partitions (lower
    row). The labeled numbers denote the order of blocks.}
\label{fig:monotonepartitions}
\end{figure}

\begin{remark}
  Interval partitions are characterized by the property
  that the canonical order defined  in Definition~\ref{app:def:orderdrop}
  extends to the entire blocks, i.e.,
    the blocks of an interval partition $\{I_1,I_2,\dots,I_p\}$ can be uniquely ordered so that $i<j$ whenever $i\in I_s$, $j\in I_t$, $s<t$. 
    This ordering provides a natural embedding $\IP_n \subseteq \OIP_n$ 
and we may write $(I_1, I_2, \ldots, I_p) \in \IP_n$ rather than
    $\{I_1, I_2, \dots,I_p\}\in\IP_n$. 
\end{remark}

\end{exa}
In section~\ref{sec:freeLiealg}, we need the following extension of ordered set partitions. 
\begin{defi} 
  \label{app:def:pseudopartition}
  An \emph{ordered pseudopartition} of $[n]$ is a sequence
  $(P_1, P_2, \dots,P_p)$ of disjoint subsets of $[n]$ 
  such that $\cup_{i=1}^p P_i=[n]$ with empty blocks allowed.
  We keep the notation $\abs{\pi}=p$ for the length,
  now including empty blocks. 
  The set of ordered pseudopartitions of $[n]$ is denoted by $\EOSP_n$. 
\end{defi}

\begin{lemma}[{see, e.g.,~\cite{BilleraSarangarajan:1996:combinatorics}}]
The poset of ordered set partitions $\OSP_n$ is isomorphic to  the poset of nonempty chains in the boolean lattice $2^n$   with the reverse refinement order, i.e.,   a chain $\emptyset= A_0\subset A_1\subset \dotsm \subset A_k= [n]$ is smaller than a chain $\emptyset= B_0\subset B_1\subset \dotsm \subset B_l= [n]$ if it is finer, i.e., as sets $\{B_0,B_1,\dots,B_l\} \subseteq \{A_0,A_1,\dots,A_k\}$.
\end{lemma}
\begin{proof}
  The bijection is given by the map
  $$
  \Phi_n:(A_1\subset A_2\subset \dotsm \subset A_k)
 \mapsto (A_1,A_2\setminus A_1,A_3\setminus A_2,\dots,A_k\setminus A_{k-1}). 
  $$
\end{proof}

Moreover, when the empty chain is added,
$\OSP_n$ becomes a lattice isomorphic to the face lattice
of the \emph{permutohedron}
\cite{BilleraSarangarajan:1996:combinatorics,Tonks:1997:relating}.
With this alternative picture it is now easy to see a join
semilattice structure, namely the join operation corresponds
to the intersection of chains in the chain poset.

We denoted by $\hat{1}_n:=([n])$ the unique maximal ordered set partition.
On the other hand there are several minimal elements, namely
all permutations of the minimal set partition 
$\hat{0}_n:=(\{1\}, \{2\}, \dots,\{n\})$. 
Consequently there is no meet operation, but we define the following associative but noncommutative replacement.
It turns $\OSP_n$ into a \emph{band}, i.e.,  a semigroup in which
every element is an idempotent; however it is not a skew lattice. 
\begin{defi} \label{app:def:ordered_set_partitions}  \ 
\begin{enumerate}[label=\rm(\roman*),leftmargin=1cm] 
\item For an ordered set partition $\sigma=(S_1, S_2, \dots,S_s) \in \OSP_n$ 
  and a nonempty subset
  $P\subseteq [n]$, let the \emph{restriction} $\sigma\restr_P\in\OSP_P$ 
  be the ordered set partition $(P \cap S_1,P\cap S_2, \dots,P\cap S_s)\in
  \OSP_P$, where empty sets are dropped; see \cite{PS2008}, where this
  operation arises in the context of Hopf algebras. 
\item
 For $\pi =(P_1, P_2,\ldots,P_p), \sigma=(S_1,S_2,\ldots,S_s) \in \OSP_n$,
 we  define the \emph{quasi-meet operation}
  $\pi \curlywedge \sigma \in \OSP_n$ to be $(P_1 \cap S_1, P_1\cap S_2 ,
  \ldots, P_1 \cap S_s, P_2 \cap S_1, P_2\cap S_2,\ldots, P_2\cap S_s, \dots, P_p \cap S_s)$, where empty
  sets are skipped. In other words, the quasi-meet
  is the concatenation of the restrictions $\pi\curlywedge\sigma = \sigma\restr_{P_1}
  \sigma\restr_{P_2}\dotsm \sigma\restr_{P_p}$.
\end{enumerate}
\end{defi}
\begin{remark}
  The quasi-meet operation is associative and coincides with the multiplication operation
  in the Solomon-Tits algebra of the symmetric group 
  \cite{Solomon:1976:mackey} which resurfaced recently
  in the domain of Markov chains
  \cite{Brown:2000:semigroups,Brown:2004:semigroup} 
  and Hopf algebras of noncommutative quasi-symmetric functions
  \cite{BergeronZabrocki:2009:hopf}.
\end{remark}

\begin{prop}
  \label{app:prop:quasimeet}  \ 
\begin{enumerate}[label=\rm(\roman*),leftmargin=1cm]
\item \label{app:it:barpiwsi=barwbar}
    The quasi-meet operation on $\OSP_n$ is compatible with the meet
    operation on $\SP_n$, in the sense that
    $\overline{\pi\curlywedge\sigma} = \bar\pi \wedge\bar\sigma$.
   \item \label{app:it:piwsi<pi}
    $\pi\curlywedge \sigma \leq \pi$ for any $\pi,\sigma \in\OSP_n$.
   \item
    If $\sigma \leq \rho \in \OSP_n$, then $\sigma \curlywedge \pi \leq
    \rho \curlywedge \pi$ for any $\pi \in \OSP_n$.
   \item \label{app:it:piwsi=pi}
    $\pi\curlywedge\sigma = \pi$ $\iff$ $\bar\pi\leq\bar\sigma$. 
   \item \label{app:it:piwsi=si} $\pi\curlywedge\sigma = \sigma$ $\iff$ $\sigma\leq\pi$. 
\end{enumerate}
\end{prop}
\begin{proof}
  The first three items are immediate from the definition.
  To see  \ref{app:it:piwsi=pi},
  assume first $\pi\curlywedge\sigma=\pi$. Then by  \ref{app:it:barpiwsi=barwbar}
  also 
  $\bar\pi=\bar\pi\wedge\bar\sigma \leq\bar\sigma$. Conversely,
  if $\bar\pi\leq\bar\sigma$, then $\bar\pi\wedge\bar\sigma=\bar\pi$
  and by  \ref{app:it:piwsi<pi} 
  $\pi\curlywedge\sigma\leq \pi$. Since the number of blocks of
  $\bar\pi\wedge\bar\sigma$ and $\pi\curlywedge\sigma$ are equal,
  we must have $\pi\curlywedge\sigma=\pi$.

  As for \ref{app:it:piwsi=si}, if $\pi\curlywedge\sigma=\sigma$
  then it follows from \ref{app:it:piwsi<pi} that $\sigma\leq\pi$.
  On the other hand, if $\sigma\leq\pi$ and $\pi=(B_1,B_2,\dots,B_k)$,
  then   $\pi\curlywedge\sigma=\sigma\restr_{B_1}\sigma\restr_{B_2}\dotsm \sigma\restr_{B_k}$ 
  and the order of the blocks of $\sigma$ remains unchanged, so
  $\pi\curlywedge\sigma=\sigma$. 
\end{proof}

To describe
the interval structure of the poset $\OSP_n$ we start with principal ideals.
\begin{definition}
Let $P$ be a poset.
 A \emph{down-set} or \emph{order ideal} 
 is a subset $I$ such that
$x\in I$ and $y\leq x$ implies $y\in I$. 
  The \emph{principal ideal generated by $x$},
denoted by $\downset x$, is the smallest down-set containing $x$, i.e.,
$$
\downset x = \{y\in P:y\leq x\}
.
$$
\end{definition}
The following proposition is immediate.
\begin{prop}
  \label{app:prop:princideals}
  The principal ideal generated by an element
  $\pi=(P_1,P_2,\dots,P_p)\in\OSP_n$ is canonically isomorphic to
  $$
  \OSP_{P_1}\times  \OSP_{P_2}\times\dotsm\times  \OSP_{P_p}
  $$
  via the map
  $$
  \sigma
  \mapsto
  (\sigma\restr_{P_1},\sigma\restr_{P_2},\dots,\sigma\restr_{P_p})
  $$
\end{prop}
There is no direct analogue of Proposition~\ref{app:prop:intervals} 
for ordered set partitions; instead the next proposition shows
that the interval structure can be expressed in terms of lattices 
of interval partitions.
\begin{prop}
  \label{app:prop:OSPintervals}
  Let $\sigma,\pi \in \OSP_n$ be ordered set partitions of
  size $\abs{\sigma}=s$ and $\abs{\pi}=p$ respectively 
  such that $\sigma\leq\pi$.
  If $\pi=(P_1,P_2,\dots P_p)$, let $k_j$ be the number of blocks of $\sigma$
  contained in $P_j$, $j\in\{1,2,\dots,p\}$.
  Then  $k_1+k_2+\dots+k_p=s$ and as a poset the interval 
  $[\sigma,\pi]$ is canonically isomorphic to
  $ \IP_{k_1} \times \IP_{k_2}\times \cdots \times \IP_{k_p}$. 
  More precisely, if $\sigma=(S_1, S_2,\ldots, S_s) \leq \pi \in \OSP_n$,
  then there exists a unique $\tau =(T_1, T_2, \ldots,T_p) \in \IP_s$ such that  
  $$
  \pi= \biggl(\bigcup_{i\in T_1} S_i,
       \bigcup_{i\in T_2} S_i,
       \ldots,
       \bigcup_{i\in T_p} S_i\biggr). 
  $$
  The map $\Phi: \IP_{T_1}\times  \IP_{T_2}\times \cdots \times \IP_{T_p} \to [\sigma,\pi]$ by 
  $$
  (\tau_1,\tau_2,\ldots,\tau_p) \mapsto \biggl(\bigcup_{i\in T_{1,1}} S_i, \bigcup_{i\in T_{1,2}} S_i,\ldots, \bigcup_{i\in T_{1,k_1}} S_i, \ldots, \bigcup_{i\in T_{p,1}} S_i, \ldots, \bigcup_{i\in T_{p,k_p}} S_i \biggr), 
  $$
  where $\tau_i= (T_{i,1}, T_{i,2}, \ldots, T_{i,k_i})$,  is a bijection, and so its inverse establishes a bijection
  $$
\Psi:=  \Phi^{-1}: [\sigma,\pi]\to \IP_{k_1} \times \IP_{k_2} \times \cdots \times \IP_{k_p}
  $$
  with  $k_i:=\abs{T_i}$. 
  The composition $(k_1,k_2,\dots,k_p)$ is called the \emph{type}
  of the interval $[\sigma,\pi]$. 
\end{prop}
The example in 
Figure~\ref{app:fig:ospintervalexample} is isomorphic to $\IP_2\times\IP_3$.
\begin{figure}

\begin{tikzpicture}[every node/.style={minimum width=2.6cm,minimum height=1.4cm}]

\node (12222211) at (1.2,7.1) {};
\begin{scope}[shift={(0,7)},scale=0.05]
\draw (2,0)--(2,7.5);
\draw (8,0)--(8,4.5);
\draw (14,0)--(14,4.5);
\draw (20,0)--(20,4.5);
\draw (26,0)--(26,4.5);
\draw (32,0)--(32,4.5);
\draw (38,0)--(38,7.5);
\draw (44,0)--(44,7.5);
\draw (8,4.5)--(32,4.5);
\draw (2,7.5)--(44,7.5);

\node at (2,-5) {\tiny $1$};
\node at (8,-5) {\tiny $2$};
\node at (14,-5) {\tiny $2$};
\node at (20,-5) {\tiny $2$};
\node at (26,-5) {\tiny $2$};
\node at (32,-5) {\tiny $2$};
\node at (38,-5) {\tiny $1$};
\node at (44,-5) {\tiny $1$};
\end{scope} 

\node (13233211) at (-2.8,5.1) {};
\begin{scope}[shift={(-4,5)},scale=0.05]
\draw (2,0)--(2,10.5);
\draw (8,0)--(8,4.5);
\draw (14,0)--(14,7.5);
\draw (20,0)--(20,4.5);
\draw (26,0)--(26,4.5);
\draw (32,0)--(32,7.5);
\draw (38,0)--(38,10.5);
\draw (44,0)--(44,10.5);
\draw (8,4.5)--(26,4.5);
\draw (14,7.5)--(32,7.5);
\draw (2,10.5)--(44,10.5);

\node at (2,-5) {\tiny $1$};
\node at (8,-5) {\tiny $3$};
\node at (14,-5) {\tiny $2$};
\node at (20,-5) {\tiny $3$};
\node at (26,-5) {\tiny $3$};
\node at (32,-5) {\tiny $2$};
\node at (38,-5) {\tiny $1$};
\node at (44,-5) {\tiny $1$};
\end{scope} 

\node (13232211) at (1.2,5.1) {}; uuu
\begin{scope}[shift={(0,5)},scale=0.05]
\draw (2,0)--(2,10.5);
\draw (8,0)--(8,4.5);
\draw (14,0)--(14,7.5);
\draw (20,0)--(20,4.5);
\draw (26,0)--(26,7.5);
\draw (32,0)--(32,7.5);
\draw (38,0)--(38,10.5);
\draw (44,0)--(44,10.5);
\draw (8,4.5)--(20,4.5);
\draw (14,7.5)--(32,7.5);
\draw (2,10.5)--(44,10.5);

\node at (2,-5) {\tiny $1$};
\node at (8,-5) {\tiny $3$};
\node at (14,-5) {\tiny $2$};
\node at (20,-5) {\tiny $3$};
\node at (26,-5) {\tiny $2$};
\node at (32,-5) {\tiny $2$};
\node at (38,-5) {\tiny $1$};
\node at (44,-5) {\tiny $1$};
\end{scope}

\node (14243211) at (-2.8,2.1) {};
\begin{scope}[shift={(-4,2)},scale=0.05]
\draw (2,0)--(2,10.5);
\draw (8,0)--(8,4.5);
\draw (14,0)--(14,7.5);
\draw (20,0)--(20,4.5);
\draw (26,0)--(26,4.5);
\draw (32,0)--(32,7.5);
\draw (38,0)--(38,10.5);
\draw (44,0)--(44,10.5);
\draw (8,4.5)--(20,4.5);
\draw (26,4.5)--(26,4.5);
\draw (14,7.5)--(32,7.5);
\draw (2,10.5)--(44,10.5);

\node at (2,-5) {\tiny $1$};
\node at (8,-5) {\tiny $4$};
\node at (14,-5) {\tiny $2$};
\node at (20,-5) {\tiny $4$};
\node at (26,-5) {\tiny $3$};
\node at (32,-5) {\tiny $2$};
\node at (38,-5) {\tiny $1$};
\node at (44,-5) {\tiny $1$};
\end{scope} 

\node (13333322) at (5.2,5.1) {}; uuu
\begin{scope}[shift={(4,5)},scale=0.05]
\draw (2,0)--(2,4.5);
\draw (8,0)--(8,4.5);
\draw (14,0)--(14,4.5);
\draw (20,0)--(20,4.5);
\draw (26,0)--(26,4.5);
\draw (32,0)--(32,4.5);
\draw (38,0)--(38,4.5);
\draw (44,0)--(44,4.5);
\draw (2,4.5)--(2,4.5);
\draw (8,4.5)--(32,4.5);
\draw (38,4.5)--(44,4.5);

\node at (2,-5) {\tiny $1$};
\node at (8,-5) {\tiny $3$};
\node at (14,-5) {\tiny $3$};
\node at (20,-5) {\tiny $3$};
\node at (26,-5) {\tiny $3$};
\node at (32,-5) {\tiny $3$};
\node at (38,-5) {\tiny $2$};
\node at (44,-5) {\tiny $2$};
\end{scope} 

\node (14344322) at (1.2,2.1) {};
\begin{scope}[shift={(0,2)},scale=0.05]
\draw (2,0)--(2,4.5);
\draw (8,0)--(8,4.5);
\draw (14,0)--(14,7.5);
\draw (20,0)--(20,4.5);
\draw (26,0)--(26,4.5);
\draw (32,0)--(32,7.5);
\draw (38,0)--(38,4.5);
\draw (44,0)--(44,4.5);
\draw (2,4.5)--(2,4.5);
\draw (8,4.5)--(26,4.5);
\draw (14,7.5)--(32,7.5);
\draw (38,4.5)--(44,4.5);

\node at (2,-5) {\tiny $1$};
\node at (8,-5) {\tiny $4$};
\node at (14,-5) {\tiny $3$};
\node at (20,-5) {\tiny $4$};
\node at (26,-5) {\tiny $4$};
\node at (32,-5) {\tiny $3$};
\node at (38,-5) {\tiny $2$};
\node at (44,-5) {\tiny $2$};
\end{scope} 

\node (14343322) at (5.2,2.1) {};
\begin{scope}[shift={(4,2)},scale=0.05]
\draw (2,0)--(2,4.5);
\draw (8,0)--(8,4.5);
\draw (14,0)--(14,7.5);
\draw (20,0)--(20,4.5);
\draw (26,0)--(26,7.5);
\draw (32,0)--(32,7.5);
\draw (38,0)--(38,4.5);
\draw (44,0)--(44,4.5);
\draw (2,4.5)--(2,4.5);
\draw (8,4.5)--(20,4.5);
\draw (14,7.5)--(32,7.5);
\draw (38,4.5)--(44,4.5);

\node at (2,-5) {\tiny $1$};
\node at (8,-5) {\tiny $4$};
\node at (14,-5) {\tiny $3$};
\node at (20,-5) {\tiny $4$};
\node at (26,-5) {\tiny $3$};
\node at (32,-5) {\tiny $3$};
\node at (38,-5) {\tiny $2$};
\node at (44,-5) {\tiny $2$};
\end{scope}

\node (15354322) at (1.2,0.1) {};
\begin{scope}[shift={(0,0)},scale=0.05]
\draw (2,0)--(2,4.5);
\draw (8,0)--(8,4.5);
\draw (14,0)--(14,7.5);
\draw (20,0)--(20,4.5);
\draw (26,0)--(26,4.5);
\draw (32,0)--(32,7.5);
\draw (38,0)--(38,4.5);
\draw (44,0)--(44,4.5);
\draw (2,4.5)--(2,4.5);
\draw (8,4.5)--(20,4.5);
\draw (26,4.5)--(26,4.5);
\draw (14,7.5)--(32,7.5);
\draw (38,4.5)--(44,4.5);

\node at (2,-5) {\tiny $1$};
\node at (8,-5) {\tiny $5$};
\node at (14,-5) {\tiny $3$};
\node at (20,-5) {\tiny $5$};
\node at (26,-5) {\tiny $4$};
\node at (32,-5) {\tiny $3$};
\node at (38,-5) {\tiny $2$};
\node at (44,-5) {\tiny $2$};
\end{scope}

\draw (13233211)--(12222211);
\draw (13232211)--(12222211);
\draw (14243211)--(13233211);
\draw (14243211)--(13232211);
\draw (13333322)--(12222211);
\draw (14344322)--(13233211);
\draw (14344322)--(13333322);
\draw (14343322)--(13232211);
\draw (14343322)--(13333322);
\draw (15354322)--(14243211);
\draw (15354322)--(14344322);
\draw (15354322)--(14343322);
\end{tikzpicture}
   \caption{An interval isomorphic to $\IP_2\times\IP_3$.}
 \label{app:fig:ospintervalexample}
\end{figure}
\begin{proof} Let $\pi=(P_1, P_2,\ldots,P_p)$. Each block of $\pi$ is the union of
  blocks of $\sigma$, and so we can find $A \subseteq [s]$ such that
  $P_1=\cup_{i\in A} S_i$. We show that there exists $k$ such that $A =
  [k]$. Suppose that there are $1\leq u<v \leq p$ such that $u\notin A$ and
  $v\in A$. Then there is $j\geq2$ such that $S_u \subseteq P_j$. This
  contradicts the fact that $u<v$, $S_v \subseteq P_1$ and $\sigma\leq \pi$.
  Hence $A = [k]$ for some $k$. 
  Removing the first block of $\pi$ and the first $k$ blocks of $\sigma$
  we can repeat the argument with the ordered set partitions
  $(S_{k+1},S_{k+2},\ldots,S_s) \leq (P_2,P_3,\ldots,P_p)$
  and find
  $P_2=\cup_{k+1 \leq i \leq k+l}S_i$ for some $l$. 
  After a finite number of iterations we thus
  construct a unique interval partition
  $\tau =(T_1,T_2,\ldots,T_p) \in \IP_s$ such that  
  $$
  \pi=\biggl(\bigcup_{i\in T_1} S_i, 
       \bigcup_{i\in T_2} S_i, 
       \ldots,
       \bigcup_{i\in T_p} S_i\biggr). 
  $$ 
  Clearly the image of $\Phi$ is contained in $[\sigma,\pi]$ and $\Phi$ is injective and it remains to show surjectivity.
  To this end pick an arbitrary $\rho \in [\sigma,\pi]$. 
  From the first part of the proposition we infer that
  $\rho \geq \sigma$ is of the form 
  $$
  \rho=\biggl(\bigcup_{i \in G_1} S_i,
      \bigcup_{i \in G_2} S_i,
      \ldots,
      \bigcup_{i\in G_t} S_i\biggr)
  $$
for some $\gamma =(G_1,G_2,\ldots, G_t) \in \IP_s$. Since $\rho \leq \pi$, the partition $\gamma$ must be finer than $\tau$. Hence, each restriction $\gamma\restr_{T_i} \in \IP_{T_i}, i\in[p]$, consists of sequence of consecutive blocks of $\gamma$ without splitting any original block $G_1,G_2,\dots,G_t \in\gamma$.
Hence we obtain $\Phi((\gamma_1,\gamma_2, \ldots, \gamma_p)) = \rho$ and
the map $\Phi$ is indeed bijective.
\end{proof}

\subsection{Incidence algebras and multiplicative functions}
Let $(P,\leq)$ be a (finite) partially ordered set.
The incidence algebra $\IncAlg(P)=\IncAlg(P,\C)$ is the algebra of 
functions supported on the set of pairs $\{(x,y)\in P\times P : x,y\in P; x\leq y\}$
with convolution
$$
f*g(x,y) = \sum_{x\leq z\leq y} f(x,z)\, g(z,y)
$$
For example, if $P$ is the $n$-set~$\{1,2,\dots,n\}$ with the natural order,
then $\IncAlg(P)$
is the algebra of $n\times n$ upper triangular matrices.
The algebra~$\IncAlg(P)$ is unital with the Kronecker function $\delta(x,y)$ 
serving as the unit element and a function $f\in\IncAlg(P)$
is invertible if and only if $f(x,x)$ is nonzero for every $x\in P$.
An example of an invertible function is the \emph{Zeta function},
which is defined as $\zeta(x,y)\equiv 1$.
Its inverse is called the \emph{M{\"o}bius function} of~$P$, 
denoted $\mu(x,y)$.
For functions $F,G:P\to\C$ we have the fundamental equivalence
 (``M{\"o}bius inversion formula'')
$$
\left(
  \forall x\in P:  F(x) = \sum_{y\leq x} G(y)
\right)
\qquad\iff\qquad
\left(
  \forall x\in P:  G(x) = \sum_{y\leq x} F(y)\,\mu(y,x)
\right)
$$
A function $f\in\IncAlg(\SP_n)$ (actually a family of functions)
is called \emph{multiplicative} if there is a \emph{characteristic sequence}
 $(f_n)_{n\geq1}$ such
that for any pair $\sigma,\pi\in\SP$ we have
$$
f(\sigma,\pi) = \prod f_i^{k_i}
$$
where $k_i$ are the structural constants of the interval $[\sigma,\pi]$
from Proposition~\ref{app:prop:intervals}. 
It can be shown \cite{DoubiletRotaStanley:1972:foundations6} that the 
multiplicative functions form a subalgebra of the incidence algebra
$\IncAlg(\SP_n)$.
For example, the Zeta function is multiplicative with characteristic sequence
$(1,1,\dots)$ and the M\"obius function is multiplicative as well
with characteristic sequence $\mu_n=(-1)^{n-1}(n-1)!$,
cf.~\cite{Schutzenberger:1954:contribution,Rota:1964:theory};
more precisely, if $\pi=\{P_1,P_2,\dots,P_p\}$ then
\begin{equation}
\label{app:eq:SPmoebius}
  \mu_{\SP}(\sigma,\pi) = \prod_{i=1}^p (-1)^{k_i-1}(k_i-1)!
\end{equation}
where $k_i=\#(\sigma\restr_{P_i})$.
Multiplicative functions on the lattice of set partitions provide  
a combinatorial model for Fa\`a di Bruno's formula
which expresses the Taylor coefficients of a composition of exponential 
formal power series in terms of the coefficients of the original functions, 
see
\cite{DoubiletRotaStanley:1972:foundations6,Stanley:1999:enumerative2};
in the case of noncrossing partitions the convolution is commutative
and can be modeled as multiplication of certain power series (``$S$-transforms''), see \cite{NicaSpeicher:2006:lectures}.
The lattice of interval partitions combinatorially models the composition
of ordinary formal power series, see \cite{Joyal:1981:theorie}.
From Proposition~\ref{app:prop:OSPintervals} one might guess that convolution
on the poset of ordered set partitions is also related to some kind
of function composition, and Proposition~\ref{app:prop:multiplicative} below shows that this is indeed the case for a certain class of functions to be defined next.

\begin{defi}

Denote by $\Nfin^\infty$  the set of finite sequences of positive integers:  
\[
\Nfin^\infty=
\bigcup_{p=1}^\infty \N^p. 
\]
For $m \in \N$ let $F_{\OSP^m}$ be the set of $\C$-valued functions $f$ on the set of $m$-chains
 \[
 \{(\sigma_1, \dots, \sigma_m) \in  \cup_{n=1}^\infty \underbrace{(\OSP_n\times \cdots \times \OSP_n)}_{\text{$m$ fold}}  \mid \sigma_1 \leq \sigma_2 \leq \cdots \leq \sigma_m\}. 
\] 
We are concerned with different levels of reduced incidence algebras.
 \begin{enumerate}[label=\rm(\roman*),leftmargin=1cm]

\item  A function $f \in F_{\OSP^2}$ is said to be \emph{adapted}
if 
 there is a family $(f_{\uk})_{\uk\in \Nfin^\infty} \subseteq \C$ such that 
$$
f(\sigma,\pi) = f_{k_1,\dots,k_p}, 
$$ 
where $(k_i)_{i=1}^p$ is the type of the interval $[\sigma,\pi]$ as 
defined in Proposition~\ref{app:prop:OSPintervals}. The family $(f_{\uk})_{\uk\in \Nfin^\infty}$ is called the {\it defining family}. 

\item To any adapted function $f\in F_{\OSP^2}$ we associate its multivariate generating function 
$$
\uZ_f (\uz)= \sum_{\uk\in\Nfin^\infty} f_{\uk} \uz^{\uk},\qquad \uz=(z_1,z_2,z_3,\dots)
$$
with the usual multiindex convention $\uz^{\uk}= z_1^{k_1}z_2^{k_2} \cdots$, 
where $z_1,z_2,\dots$ are commuting indeterminates.

\item  A function $f \in F_{\OSP^2}$ is said to be \emph{multiplicative} if it is adapted and moreover the defining family satisfies
$$
f_{k_1,\dots, k_p} = \prod_{i=1}^p f_{k_i}.  
$$ 
If $f$ is multiplicative, then the sequence of values $f_n=f(\hat{0}_n,\hat{1}_n)$ is called the \emph{defining
  sequence} of $f$.

\item To a multiplicative function $f\in F_{\OSP^2}$ we associate the (univariate)
 generating function
$$
Z_f (z)= \sum_{n=1}^\infty f_n z^n. 
$$

\item   A function $f \in F_{\OSP^3}$ is said to be \emph{quasi-multiplicative} if
 there is an array of coefficients $(f_{j k})_{j,k =1}^\infty \subseteq \C$ such that 
$$
f(\sigma,\rho,\pi) = \prod_{i=1}^p \prod_{G\in\gamma_i} f_{i, \abs{G}}, 
$$ 
where $(\gamma_1,\dots, \gamma_p)$ is the image of $\rho$ under the map $\Psi$ in Proposition
\ref{app:prop:OSPintervals}. The array $(f_{j k})_{j,k =1}^\infty$ is called the {\it defining array} of $f$. 

\item To a  quasi-multiplicative function $f\in F_{\OSP^3}$  we associate a sequence of (univariate) generating functions
$$
Z_{f}^{(j)}( z)= \sum_{k=1}^\infty f_{j k} z^k,\qquad j \in \N. 
$$

\item 
 For $f \in F_{\OSP^3}$ and $g \in F_{\OSP^2}$ we define the convolution 
 \begin{equation*}
 (f\otriangle g)(\sigma, \pi):= \sum_{\rho\in[\sigma,\pi]} f(\sigma,\rho,\pi)\,g(\rho,\pi),\qquad \sigma \leq \pi. 
 \end{equation*}
 
\end{enumerate}
\end{defi}

This latter provides a combinatorial model for the composition 
of multivariate functions. 
\begin{prop}
  \label{app:prop:multiplicative}
 If  a function $f\in F_{\OSP^3}$ is quasi-multiplicative and $g \in F_{\OSP^2}$  is adapted then $f\otriangle g \in F_{\OSP^2}$ is adapted and 
$$
\uZ_{f\otriangle g}(\uz) = \uZ_g(Z_f^{(1)}(z_1), Z_f^{(2)}(z_2), \dots),\qquad \uz =(z_1,z_2,\dots).  
$$
\end{prop}
\begin{proof}
We use the notations in the statement of Proposition
\ref{app:prop:OSPintervals}. Pick any $\rho \in[\sigma,\pi]$ and let
$(\gamma_1,\gamma_2, \ldots, \gamma_p):=\Psi(\rho)$ be its image under $\Psi$.
Thus  $\gamma_i=(G_{i,1}, G_{i,2}, \ldots, G_{i,r_i})\in \IP_{k_i}$ is
an interval partition and we have the bijective images
\begin{align*}
\Psi([\sigma,\rho])&= \IP_{\abs{G_{1,1}}}\times  \IP_{\abs{G_{1,2}}}\times 
 \cdots \times \IP_{\abs{G_{1,r_1}}} \times \IP_{\abs{G_{2,1}}} \times \cdots \times \IP_{\abs{G_{p,r_p}}}, \\
\Psi([\rho,\pi])&= \IP_{\abs{\gamma_1}}\times
                   \IP_{\abs{\gamma_2}}\times
                   \cdots \times
                   \IP_{\abs{\gamma_p}}.
\end{align*}
Hence
\begin{align*}
(f\otriangle g)\,(\sigma,\pi) 
&= \sum_{\rho \in [\sigma,\pi]} f(\sigma,\rho,\pi)\,g(\rho, \pi) \\
&= \sum_{(\gamma_1,\gamma_2,\ldots,\gamma_p)\in \IP_{k_1}\times \IP_{k_2}\times \cdots \times \IP_{k_p}}\prod_{i=1}^p\left(\prod_{G\in\gamma_i} f_{i,\abs{G}}\right)  g_{\abs{\gamma_1}, \abs{\gamma_2}, \dots, \abs{\gamma_p}} \\
&= \sum_{r_i\in[k_i], 1 \leq i \leq p}\sum_{\substack{(n_{i k})_{i \in [p], k\in[r_i]} \\ n_{i1}+n_{i2}+\cdots + n_{i r_i}=k_i, 1\leq i \leq p}}\prod_{i=1}^p\left(\prod_{k=1}^{r_i} f_{i, n_{ik}}\right)  g_{r_1, r_2, \dots, r_p} \\
&=: (f\otriangle g)_{k_1,k_2,\dots, k_p}. 
\end{align*}
This shows that for every pair $(\sigma,\pi)$ the
value $f\otriangle g(\sigma,\pi)$ is determined by the structural sequence
$(k_1,\dots, k_p)$ and thus $f\otriangle g$ is adapted.
Now multiplying the terms with $z^{\uk}$ and summing over $\uk$ we obtain 
 \begin{align*}
\uZ_{f\otriangle g}(\uz) 
&= \sum_{\uk=(k_1,k_2,\dots) \in \Nfin^\infty} (f\otriangle g)_{\uk} \uz^{\uk} \\
&= \sum_{p\geq1} \sum_{\underline{r}=(r_1,r_2,\dots,r_p) \in\N^p} \sum_{\substack{(n_{i1}, \dots, n_{i r_i}) \in \N^{r_i} \\ 1 \leq i \leq p}}\prod_{i=1}^p\left(\prod_{k=1}^{r_i} f_{i, n_{ik}}z_i^{n_{ik}}\right)  g_{r_1, r_2, \dots, r_p} \\
&= \sum_{p\geq1} \sum_{\underline{r}=(r_1,r_2,\dots,r_p) \in\N^p} g_{r_1, r_2, \dots, r_p}\prod_{i=1}^p Z_f^{(i)}(z_i)^{r_i}  \\
&= \uZ_g(Z_f^{(1)}(z_1), Z_f^{(2)}(z_2),\dots). 
\end{align*}

\end{proof}

\begin{corollary}
  \label{app:cor:multiplicative2}
 If $f,g \in F_{\OSP^2}$ are multiplicative, then so is $f\ast g$ and 
$$
Z_{f\ast g}(z) = Z_g(Z_f(z)). 
$$
\end{corollary}
\begin{proof}
Given a multiplicative function $f\in F_{\OSP^2}$,
we lift it to a quasi-multiplicative function $\tilde{f}\in  F_{\OSP^3}$ 
via its defining family $\tilde{f}_{j k}= f_k$, i.e.,
$\tilde{f}(\sigma,\rho,\pi):=f(\sigma,\rho)$ and the generating functions
are $Z_{\tilde{f}}^{(j)}(z)= Z_{\tilde{f}}(z)$ for all $j \geq 1$. 
On the other hand, $g$ is multiplicative, therefore adapted with
$g_{k_1,\dots, k_p}= g_{k_1} \cdots g_{k_p}$ and has generating function
\begin{equation*}
\uZ_g(z_1,z_2,\dots)= \sum_{p=1}^\infty \sum_{(k_1,k_2,\dots, k_p) \in \N^p} g_{k_1}g_{k_2} \dotsm g_{k_p}z_1^{k_1}z_2^{k_2}\cdots z_p^{k_p} = \sum_{p=1}^\infty Z_g(z_1) Z_g(z_2) \dotsm Z_g(z_p). 
\end{equation*}
By Proposition~\ref{app:prop:multiplicative}
$$
\uZ_{\tilde{f}\otriangle g}(z_1,z_2,\dots)=  \sum_{p=1}^\infty Z_g(Z_f(z_1))Z_g(Z_f(z_2)) \dotsm Z_g(Z_f(z_p)).  
$$
This shows that $(\tilde{f}\otriangle g)_{k_1,k_2,\dots, k_p}= \prod_{i=1}^p h_{k_i}$, where $h_k:= \frac{1}{k!}\frac{d^k}{dz^k}\big|_{z=0}Z_g(Z_f(z)).$ Therefore $\tilde{f} \otriangle g = f \ast g$ is multiplicative and $Z_{f\ast g}= Z_g(Z_f(z))$.  
\end{proof}

\subsection{Special functions on the poset of ordered set partitions} We define several special functions in the case of ordered set partitions, and compute their generating functions.  

\begin{defi}\label{app:def:factorials}
 Given ordered set partitions $\sigma\leq\rho\leq \pi =(P_1,P_2,\dots,P_p)\in \OSP_n$, a sequence $\ut=(t_1,t_2,\dots)\in\R^\N$ and a number $t\in \R$ we define 
\begin{align*}
\beta_{\ut}(\sigma,\pi)&=\prod_{i=1}^p\binom{t_i}{\sharp(\sigma\restr_{P_i})},   \\
\beta_t(\sigma,\pi) &= \beta_{(t,t,\dots)}(\sigma,\pi),  \\
\gamma_{\ut}(\sigma,\rho,\pi) &= \prod_{i=1}^p \prod_{G\in\gamma_i} \binom{t_i}{\abs{G}},
\end{align*}
where $\binom{t}{n}$ is the generalized binomial coefficient and $(\gamma_1,\gamma_2,\dots, \gamma_p)$ is the image of $\rho$ by the map $\Psi$ in Proposition
\ref{app:prop:OSPintervals}.
Moreover, for $\sigma\leq\pi$ we define 
\begin{align}
  [\sigma:\pi] &= \prod_{P \in \bar{\pi}} \sharp(\sigma\restr_P),
                 \nonumber
  \\
  [\sigma:\pi]!&= \prod_{P \in \bar\pi} \sharp(\sigma\restr_P)!,
                 \label{eq:[sigma:pi]!}
  \\
  \widetilde{\zeta}(\sigma,\pi) &= \frac{1}{[\sigma:\pi]!},
                 \nonumber
  \\
  \label{eq:mutilde}
  \widetilde{\mu}(\sigma,\pi)&= \frac{(-1)^{\abs{\sigma}-\abs{\pi}}}{[\sigma:\pi]}. 
\end{align}
\end{defi}
\begin{remark} \label{rem:beta_gamma} \ 
 \begin{enumerate}[label=\rm(\roman*),leftmargin=1cm] 
\item The combinatorial significance of the numbers \eqref{eq:[sigma:pi]!} is
   \begin{equation}
     \label{eq:[sigma:pi]!=fixed}
     [\sigma:\pi]! = \#\{\rho~ |~ \rho\leq \pi \text{ and } \bar{\rho}=\bar{\sigma}\}
   \end{equation}
   
  \item The values $\beta_{\ut}(\sigma,\pi)$ and $\gamma_{\ut}(\sigma,\rho,\pi)$ 
  depend only on the first $\abs{\pi}$ elements of $\ut$.  
 \item\label{item:beta_gamma}
  $\beta_{\ut}, \beta_t, \gamma_{\ut}$ are related via 
  \[
 \beta_{\ut}(\sigma,\pi) = \gamma_{\ut}(\sigma, \pi,\pi) \quad \text{and}\quad  \gamma_{\ut}(\sigma, \rho,\pi) = \prod_{i=1}^p \beta_{t_i}(\sigma\restr_{P_i}, \rho\restr_{P_i}). 
   \]
   \end{enumerate}
\end{remark}

Since every interval in $\OSP_n$ is isomorphic to a product of lattices of
interval partitions (Proposition~\ref{app:prop:OSPintervals})
and thus is a Boolean lattice, it follows that the semilattice of ordered set partitions is Eulerian,
i.e., its M\"obius function only depends on the rank and $\mu_{\OSP}(\sigma,\pi)
= (-1)^{\abs{\sigma}-\abs{\pi}}$,
see \cite[Proposition 3 and its Corollary]{Rota:1964:theory}. 
Hence if $\sigma\leq\pi$ we may write,  
\begin{align*}
&\widetilde{\zeta}(\sigma,\pi) = \frac{\zeta_\SP(\bar{\sigma}, \bar{\pi})}{[\sigma:\pi]!}, \\
&\widetilde{\mu}(\sigma,\pi)= \frac{\mu_{\OSP}(\sigma,\pi)}{[\sigma:\pi]} =
  \frac{\mu_{\SP}(\bar{\sigma},\bar{\pi})}{[\sigma:\pi]!}
  =\mu_{\SP}(\bar{\sigma},\bar{\pi}) \, \tilde{\zeta}(\sigma,\pi)
  ,  
\end{align*}
where the M\"obius function $\mu_{\SP}$ was defined in formula~\eqref{app:eq:SPmoebius}.

As a consequence of observation \eqref{eq:[sigma:pi]!=fixed} we have the
following connection between the M\"{o}bius inversion on $\OSP$ and $\SP$.
\begin{prop}
  \label{app:prop:f*tzeta=fbar*zeta}
  Let $f:\OSP\to \C$ be a function which is invariant under permutations of
  the blocks, i.e., $f(\pi)=f(\pi')$ whenever $\overline{\pi}=\overline{\pi'}$
  and denote by $\bar{f}:\SP\to\C$ the corresponding function on $\SP$,
  then the convolutions
  $f*_\OSP\tilde{\zeta}$
  and
  $f*_\OSP\tilde{\mu}$ are invariant as well and are given by
  \begin{align*}
    f*_\OSP\tilde{\zeta}(\pi)
    &= \bar{f}*_\SP\zeta_\SP(\overline{\pi})
      \\
    f*_\OSP\tilde{\mu}(\pi)
    &= \bar{f}*_\SP\mu_\SP(\overline{\pi})
      ,
\intertext{
  i.e.,
  }
    \sum_{\substack{\rho\in\OSP\\ \rho\leq\pi}}f(\rho)\,\tilde{\zeta}(\rho,\pi)
    &=    \sum_{\substack{\sigma\in\SP\\
    \sigma\leq\overline{\pi}}}\bar{f}(\sigma)
    \\
    \sum_{\substack{\rho\in\OSP\\ \rho\leq\pi}}f(\rho)\,\tilde{\mu}_{\OSP}(\rho,\pi)
    &=    \sum_{\substack{\sigma\in\SP\\
    \sigma\leq\overline{\pi}}}\bar{f}(\sigma)\,\mu(\sigma,\overline{\pi})
    .
  \end{align*}
\end{prop}

\begin{prop} \ 
\begin{enumerate}[label=\rm(\roman*),leftmargin=1cm]
\item  The function $\beta_{\ut} \in F_{\OSP^2}$ is adapted with defining family 
\begin{equation*}
(\beta_{\ut})_{\uk} = \prod_{i=1}^p \binom{t_i}{k_i}. 
\end{equation*}

\item The function $\gamma_{\ut}\in F_{\OSP^3}$ is quasi-multiplicative with defining array 
\begin{equation*}
(\gamma_{\ut})_{jk}= \binom{t_j}{k}. 
\end{equation*}

\item The functions $\beta_t,\widetilde{\mu},\widetilde{\zeta}\in F_{\OSP^2}$ are
  multiplicative with defining sequences
\begin{align*}
&\beta_t(\hat{0}_n,\hat{1}_n)= \binom{t}{n}, \\ 
&\widetilde{\zeta}(\hat{0}_n,\hat{1}_n)= \frac{1}{n!}, \\
&\widetilde{\mu}(\hat{0}_n,\hat{1}_n) = \frac{(-1)^{n-1}}{n}. 
\end{align*}
\end{enumerate}
\end{prop}
\begin{proof}
 The claims follow by definition and by Proposition~\ref{app:prop:OSPintervals}. 
\end{proof}
\begin{corollary}
  \label{app:cor:inversion} \ 
\begin{enumerate}[label=\rm(\roman*),leftmargin=1cm]
\item
  \label{app:it:inversion0} The inverse function (with respect to the convolution $\ast$) of $\widetilde{\mu}$ is $\widetilde{\zeta}$. 
 \item  
  \label{app:it:semigroup2} For $\us,\ut \in \R^\infty$ we have 
\[
\gamma_{\us}\otriangle\beta_{\ut} = \beta_{\us \circ \ut}, 
\]
where $\us \circ \ut=(s_1 t_1, s_2 t_2, \dots)$.

\item 
 \label{app:it:semigroup} $\beta_t$ satisfies the semigroup property $\beta_s \ast \beta_t = \beta_{s t}$ for $s,t \in\R$. 
\end{enumerate}
\end{corollary}
\begin{proof}
\ref{app:it:inversion0}\,\, Since $Z_{\widetilde{\mu}}(z)=\log (1+z)$ and $Z_{\widetilde{\zeta}}(z)= e^z-1$, we have $Z_{\widetilde{\mu}}(Z_{\widetilde{\zeta}}(z))=z$. 

\itemspacing
\ref{app:it:semigroup2}\,\,  We have 
\begin{align*}
&Z_{\gamma_{\us}}^{(j)}(z) = \sum_{k=1}^\infty \binom{s_j}{k} z^k = (1+z)^{s_j}-1, \\
  &\uZ_{\beta_{\ut}}(\uz) = \sum_{\uk} \prod_{i=1}^p \binom{t_i}{k_i} z_i^{k_i} 
      = \sum_{p\geq1}\prod_{i=1}^p\left((1+z_i)^{t_i}-1\right). 
\end{align*}
By Proposition~\ref{app:prop:multiplicative}, $\uZ_{\gamma_{\us}\otriangle\beta_{\ut}}(\uz) = \uZ_{\beta_{\ut}}(Z_{\gamma_{\us}}^{(1)}(z_1), Z_{\gamma_{\us}}^{(2)}(z_2),\dots)$, which equals $\uZ_{\beta_{\us\circ \ut}}(\uz)$. 

\itemspacing
\ref{app:it:semigroup}\,\, We can use $Z_{\beta_t}(z)=\sum_{n=1}^\infty \binom{t}{n}z^n = (1+z)^t-1$ and Corollary~\ref{app:cor:multiplicative2}. 
\end{proof}

\bibliographystyle{amsalpha}
\newcommand{\etalchar}[1]{$^{#1}$}
\providecommand{\bysame}{\leavevmode\hbox to3em{\hrulefill}\thinspace}
\providecommand{\MR}{\relax\ifhmode\unskip\space\fi MR }
\providecommand{\MRhref}[2]{\href{http://www.ams.org/mathscinet-getitem?mr=#1}{#2}
}
\providecommand{\href}[2]{#2}

\end{document}